\documentclass{amsart}
\usepackage{amssymb,amsmath,stmaryrd,mathrsfs}

\def\definetac{\newif\iftac}    
\ifx\tactrue\undefined
  \definetac
  \ifx\state\undefined\tacfalse\else\tactrue\fi
\fi

\def\definebeamer{\newif\ifbeamer}
\ifx\beamertrue\undefined
  \definebeamer
  \ifx\uncover\undefined\beamerfalse\else\beamertrue\fi
\fi

\iftac\else\usepackage{amsthm}\fi
\usepackage[all,2cell]{xy}
\ifbeamer\else
  \usepackage{enumitem}
  \usepackage{xcolor}
  \definecolor{darkgreen}{rgb}{0,0.45,0} 
  \usepackage[pagebackref,colorlinks,citecolor=darkgreen,linkcolor=darkgreen]{hyperref}
\fi
\usepackage{mathtools}          
\usepackage{graphics}           
\usepackage{ifmtarg}            
\usepackage{microtype}
\usepackage{braket}             

\usepackage{url}                
\usepackage{xspace}             
\usepackage{subcaption}

\makeatletter
\let\ea\expandafter

\def\mdef#1#2{\ea\ea\ea\gdef\ea\ea\noexpand#1\ea{\ea\ensuremath\ea{#2}\xspace}}
\def\alwaysmath#1{\ea\ea\ea\global\ea\ea\ea\let\ea\ea\csname your@#1\endcsname\csname #1\endcsname
  \ea\def\csname #1\endcsname{\ensuremath{\csname your@#1\endcsname}\xspace}}

\DeclareRobustCommand\widecheck[1]{{\mathpalette\@widecheck{#1}}}
\def\@widecheck#1#2{%
    \setbox\z@\hbox{\m@th$#1#2$}%
    \setbox\tw@\hbox{\m@th$#1%
       \widehat{%
          \vrule\@width\z@\@height\ht\z@
          \vrule\@height\z@\@width\wd\z@}$}%
    \dp\tw@-\ht\z@
    \@tempdima\ht\z@ \advance\@tempdima2\ht\tw@ \divide\@tempdima\thr@@
    \setbox\tw@\hbox{%
       \raise\@tempdima\hbox{\scalebox{1}[-1]{\lower\@tempdima\box
\tw@}}}%
    {\ooalign{\box\tw@ \cr \box\z@}}}

\newcount\foreachcount

\def\foreachletter#1#2#3{\foreachcount=#1
  \ea\loop\ea\ea\ea#3\@alph\foreachcount
  \advance\foreachcount by 1
  \ifnum\foreachcount<#2\repeat}

\def\foreachLetter#1#2#3{\foreachcount=#1
  \ea\loop\ea\ea\ea#3\@Alph\foreachcount
  \advance\foreachcount by 1
  \ifnum\foreachcount<#2\repeat}

\def\definescr#1{\ea\gdef\csname s#1\endcsname{\ensuremath{\mathscr{#1}}\xspace}}
\foreachLetter{1}{27}{\definescr}
\def\definecal#1{\ea\gdef\csname c#1\endcsname{\ensuremath{\mathcal{#1}}\xspace}}
\foreachLetter{1}{27}{\definecal}
\def\definebold#1{\ea\gdef\csname b#1\endcsname{\ensuremath{\mathbf{#1}}\xspace}}
\foreachLetter{1}{27}{\definebold}
\def\definebb#1{\ea\gdef\csname l#1\endcsname{\ensuremath{\mathbb{#1}}\xspace}}
\foreachLetter{1}{27}{\definebb}
\def\definefrak#1{\ea\gdef\csname k#1\endcsname{\ensuremath{\mathfrak{#1}}\xspace}}
\foreachletter{1}{27}{\definefrak}
\foreachLetter{1}{27}{\definefrak}
\def\definesf#1{\ea\gdef\csname i#1\endcsname{\ensuremath{\mathsf{#1}}\xspace}}
\foreachletter{1}{6}{\definesf}
\foreachletter{7}{14}{\definesf}
\foreachletter{15}{27}{\definesf}
\foreachLetter{1}{27}{\definesf}
\def\definebar#1{\ea\gdef\csname #1bar\endcsname{\ensuremath{\overline{#1}}\xspace}}
\foreachLetter{1}{27}{\definebar}
\foreachletter{1}{8}{\definebar} 
\foreachletter{9}{15}{\definebar} 
\foreachletter{16}{27}{\definebar}
\def\definetil#1{\ea\gdef\csname #1til\endcsname{\ensuremath{\widetilde{#1}}\xspace}}
\foreachLetter{1}{27}{\definetil}
\foreachletter{1}{27}{\definetil}
\def\definehat#1{\ea\gdef\csname #1hat\endcsname{\ensuremath{\widehat{#1}}\xspace}}
\foreachLetter{1}{27}{\definehat}
\foreachletter{1}{27}{\definehat}
\def\definechk#1{\ea\gdef\csname #1chk\endcsname{\ensuremath{\widecheck{#1}}\xspace}}
\foreachLetter{1}{27}{\definechk}
\foreachletter{1}{27}{\definechk}
\def\defineul#1{\ea\gdef\csname u#1\endcsname{\ensuremath{\underline{#1}}\xspace}}
\foreachLetter{1}{27}{\defineul}
\foreachletter{1}{27}{\defineul}

\def\autofmt@n#1\autofmt@end{\mathrm{#1}}
\def\autofmt@b#1\autofmt@end{\mathbf{#1}}
\def\autofmt@l#1#2\autofmt@end{\mathbb{#1}\mathsf{#2}}
\def\autofmt@c#1#2\autofmt@end{\mathcal{#1}\mathit{#2}}
\def\autofmt@s#1#2\autofmt@end{\mathscr{#1}\mathit{#2}}
\def\autofmt@f#1\autofmt@end{\mathsf{#1}}
\def\autofmt@k#1\autofmt@end{\mathfrak{#1}}
\def\autofmt@u#1\autofmt@end{\underline{\smash{\mathsf{#1}}}}
\def\autofmt@U#1\autofmt@end{\underline{\underline{\smash{\mathsf{#1}}}}}
\def\autofmt@h#1\autofmt@end{\widehat{#1}}
\def\autofmt@r#1\autofmt@end{\overline{#1}}
\def\autofmt@t#1\autofmt@end{\widetilde{#1}}
\def\autofmt@k#1\autofmt@end{\check{#1}}

\def\auto@drop#1{}
\def\autodef#1{\ea\ea\ea\@autodef\ea\ea\ea#1\ea\auto@drop\string#1\autodef@end}
\def\@autodef#1#2#3\autodef@end{%
  \ea\def\ea#1\ea{\ea\ensuremath\ea{\csname autofmt@#2\endcsname#3\autofmt@end}\xspace}}
\def\autodefs@end{blarg!}
\def\autodefs#1{\@autodefs#1\autodefs@end}
\def\@autodefs#1{\ifx#1\autodefs@end%
  \def\autodefs@next{}%
  \else%
  \def\autodefs@next{\autodef#1\@autodefs}%
  \fi\autodefs@next}

\DeclareSymbolFont{bbold}{U}{bbold}{m}{n}
\DeclareSymbolFontAlphabet{\mathbbb}{bbold}
\newcommand{\lDelta}{\ensuremath{\mathbbb{\Delta}}\xspace}

\newcommand{\albar}{\ensuremath{\overline{\alpha}}\xspace}

\mdef\delbar{\overline{\partial}}

\newcommand{\inv}{^{-1}}

\mdef\hf{\textstyle\frac12 }
\mdef\thrd{\textstyle\frac13 }
\mdef\qtr{\textstyle\frac14 }

\newcommand{\op}{^{\mathrm{op}}}

\let\adj\dashv
\SelectTips{cm}{}
\newdir{ >}{{}*!/-10pt/\dir{>}}    
\newcommand{\pushout}[1][dr]{\save*!/#1+1.2pc/#1:(1,-1)@^{|-}\restore}

\let\eqv\simeq

\mdef\Id{\mathrm{Id}}
\mdef\id{\mathrm{id}}
\alwaysmath{ell}
\alwaysmath{infty}
\alwaysmath{odot}
\def\frc#1/#2.{\frac{#1}{#2}}   
\mdef\ten{\mathrel{\otimes}}

\mdef\sqten{\mathrel{\boxtimes}}

\DeclareMathOperator\lan{Lan}

\DeclareMathOperator\colim{colim}

\newcommand{\too}[1][]{\ensuremath{\overset{#1}{\longrightarrow}}}
\newcommand{\ot}{\ensuremath{\leftarrow}}

\let\toto\rightrightarrows
\let\into\hookrightarrow

\mdef\we{\overset{\sim}{\longrightarrow}}
\mdef\leftwe{\overset{\sim}{\longleftarrow}}

\let\epi\twoheadrightarrow

\let\xto\xrightarrow
\let\xot\xleftarrow
\def\rightarrowtailfill@{\arrowfill@{\Yright\joinrel\relbar}\relbar\rightarrow}
\newcommand\xrightarrowtail[2][]{\ext@arrow 0055{\rightarrowtailfill@}{#1}{#2}}

\def\twoheadrightarrowfill@{\arrowfill@{\relbar\joinrel\relbar}\relbar\twoheadrightarrow}
\newcommand\xtwoheadrightarrow[2][]{\ext@arrow 0055{\twoheadrightarrowfill@}{#1}{#2}}

\def\slashedarrowfill@#1#2#3#4#5{%
  $\m@th\thickmuskip0mu\medmuskip\thickmuskip\thinmuskip\thickmuskip
   \relax#5#1\mkern-7mu%
   \cleaders\hbox{$#5\mkern-2mu#2\mkern-2mu$}\hfill
   \mathclap{#3}\mathclap{#2}%
   \cleaders\hbox{$#5\mkern-2mu#2\mkern-2mu$}\hfill
   \mkern-7mu#4$%
}
\def\rightslashedarrowfill@{%
  \slashedarrowfill@\relbar\relbar\mapstochar\rightarrow}
\newcommand\xslashedrightarrow[2][]{%
  \ext@arrow 0055{\rightslashedarrowfill@}{#1}{#2}}
\mdef\hto{\xslashedrightarrow{}}
\mdef\htoo{\xslashedrightarrow{\quad}}

\def\toiso{\xto{\smash{\raisebox{-.5mm}{$\scriptstyle\sim$}}}}
\def\otiso{\xot{\smash{\raisebox{-.5mm}{$\scriptstyle\sim$}}}}

\def\jd#1{\@jd#1\ej}
\def\@jd#1|-#2\ej{\@@jd#1,,\;\vdash\;\left(#2\right)}
\def\@@jd#1,{\@ifmtarg{#1}{\let\next=\relax}{\left(#1\right)\let\next=\@@@jd}\next}
\def\@@@jd#1,{\@ifmtarg{#1}{\let\next=\relax}{,\,\left(#1\right)\let\next=\@@@jd}\next}
\def\jdm#1{\@jdm#1\ej}
\def\@jdm#1|-#2\ej{\@@jd#1,,\\\vdash\;\left(#2\right)}

\long\def\my@drawfill#1#2;{%
\@skipfalse
\fill[#1,draw=none] #2;
\@skiptrue
\draw[#1,fill=none] #2;
}
\newif\if@skip
\newcommand{\skipit}[1]{\if@skip\else#1\fi}
\newcommand{\drawfill}[1][]{\my@drawfill{#1}}

\newif\ifhyperref
\@ifpackageloaded{hyperref}{\hyperreftrue}{\hyperreffalse}
\iftac
  \let\your@state\state
  \def\state#1{\my@state#1}
  \def\my@state#1.{\gdef\currthmtype{#1}\your@state{#1.}}
  \let\your@staterm\staterm
  \def\staterm#1{\my@staterm#1}
  \def\my@staterm#1.{\gdef\currthmtype{#1}\your@staterm{#1.}}
  \let\defthm\newtheorem
  \def\switchtotheoremrm{\let\defthm\newtheoremrm}
  \def\currthmtype{}
  \ifhyperref
    \def\autoref#1{\ref*{label@name@#1}~\ref{#1}}
  \else
    \def\autoref#1{\ref{label@name@#1}~\ref{#1}}
  \fi
  \AtBeginDocument{%
    \let\old@label\label%
    \def\label#1{%
      {\let\your@currentlabel\@currentlabel%
        \edef\@currentlabel{\currthmtype}%
        \old@label{label@name@#1}}%
      \old@label{#1}}
  }
\else
  \ifhyperref
    \def\defthm#1#2{%
      \newtheorem{#1}{#2}[section]%
      \expandafter\def\csname #1autorefname\endcsname{#2}%
      \expandafter\let\csname c@#1\endcsname\c@thm}
  \else
    \def\defthm#1#2{\newtheorem{#1}[thm]{#2}}
    \ifx\SK@label\undefined\let\SK@label\label\fi
    \let\old@label\label
    \let\your@thm\@thm
    \def\@thm#1#2#3{\gdef\currthmtype{#3}\your@thm{#1}{#2}{#3}}
    \def\currthmtype{}
    \def\label#1{{\let\your@currentlabel\@currentlabel\def\@currentlabel%
        {\currthmtype~\your@currentlabel}%
        \SK@label{#1@}}\old@label{#1}}
    \def\autoref#1{\ref{#1@}}
  \fi
\fi

\newtheorem{thm}{Theorem}[section]

\defthm{cor}{Corollary}
\defthm{prop}{Proposition}
\defthm{lem}{Lemma}
\defthm{sch}{Scholium}
\defthm{assume}{Assumption}
\defthm{claim}{Claim}
\defthm{conj}{Conjecture}
\defthm{hyp}{Hypothesis}
\iftac\switchtotheoremrm\else\theoremstyle{definition}\fi
\defthm{defn}{Definition}
\defthm{notn}{Notation}
\iftac\switchtotheoremrm\else\theoremstyle{remark}\fi
\defthm{rmk}{Remark}
\defthm{eg}{Example}
\defthm{egs}{Examples}
\defthm{ex}{Exercise}
\defthm{ceg}{Counterexample}

\iftac
  \let\your@endproof\endproof
  \def\my@endproof{\your@endproof}
  \def\endproof{\my@endproof\gdef\my@endproof{\your@endproof}}
  \def\qedhere{\tag*{\endproofbox}\gdef\my@endproof{\relax}}
\fi

\iftac
  \def\pr@@f[#1]{\subsubsection*{\sc #1.}}
\fi

\def\thmqedhere{\expandafter\csname\csname @currenvir\endcsname @qed\endcsname}

\ifbeamer\else

\fi

\ifbeamer\else
  \setitemize[1]{leftmargin=2em}
  \setenumerate[1]{leftmargin=*}
\fi

\iftac
  \let\c@equation\c@subsection
\else
  \let\c@equation\c@thm
\fi
\numberwithin{equation}{section}

\@ifpackageloaded{mathtools}{\mathtoolsset{showonlyrefs,showmanualtags}}{}

\alwaysmath{alpha}
\alwaysmath{beta}
\alwaysmath{gamma}
\alwaysmath{Gamma}
\alwaysmath{delta}
\alwaysmath{Delta}
\alwaysmath{epsilon}
\mdef\ep{\varepsilon}
\alwaysmath{zeta}
\alwaysmath{eta}
\alwaysmath{theta}
\alwaysmath{Theta}
\alwaysmath{iota}
\alwaysmath{kappa}
\alwaysmath{lambda}
\alwaysmath{Lambda}
\alwaysmath{mu}
\alwaysmath{nu}
\alwaysmath{xi}
\alwaysmath{pi}
\alwaysmath{rho}
\alwaysmath{sigma}
\alwaysmath{Sigma}
\alwaysmath{tau}
\alwaysmath{upsilon}
\alwaysmath{Upsilon}
\alwaysmath{phi}
\alwaysmath{Pi}
\alwaysmath{Phi}
\mdef\ph{\varphi}
\alwaysmath{chi}
\alwaysmath{psi}
\alwaysmath{Psi}
\alwaysmath{omega}
\alwaysmath{Omega}
\let\al\alpha
\let\be\beta
\let\gm\gamma

\let\de\delta

\let\ze\zeta

\makeatother

\title{Reedy categories and their generalizations}
\author{Michael Shulman}
\date{\today}
\thanks{This material is based on research sponsored by The United States Air Force Research Laboratory under agreement number FA9550-15-1-0053.  The U.S. Government is authorized to reproduce and distribute reprints for Governmental purposes notwithstanding any copyright notation thereon.  The views and conclusions contained herein are those of the author and should not be interpreted as necessarily representing the official policies or endorsements, either expressed or implied, of the United States Air Force Research Laboratory, the U.S. Government, or Carnegie Mellon University.}

\setlist[itemize,2]{label={$\star$}}

\let\M\cM
\let\N\cN
\let\B\cB
\let\C\cC
\let\D\cD
\let\E\cE
\let\F\cF
\let\G\cG

\let\J\cJ

\let\L\cL
\let\R\cR
\let\V\cV
\let\W\cW
\let\O\cO

\mdef\FO{\cF_{\O}}
\mdef\im{\mathrm{im}}

\def\Cbar{\ensuremath{\overrightarrow{\vphantom{\mathcal{C}}\smash{\overleftarrow{\mathcal{C}}}}}\xspace}
\def\Cbarde{\ensuremath{{\overrightarrow{\vphantom{\mathcal{C}}\smash{\overleftarrow{\mathcal{C}}}}\!}_\de}}

\def\FObar{\ensuremath{\overrightarrow{\vphantom{\mathcal{F}_{\mathcal{O}}}\smash{\overleftarrow{\mathcal{F}_{\mathcal{O}}}}}}\xspace}

\def\OGbar{\ensuremath{\overrightarrow{\vphantom{\mathcal{O}_G}\smash{\overleftarrow{\mathcal{O}_G}}}}\xspace}

\let\cref\autoref

\autodefs{\cCell\cCat\nSet\nlim\cProf\nAut\ncoeq\nCat\nWfs\nModel}
\mdef\vprof{\V\text{-}\cProf}
\mdef\prewfscl{\mathrm{PreWfs}_{\mathrm{cl}}}
\mdef\wfscl{\nWfs_{\mathrm{cl}}}
\mdef\modelcl{\nModel_{\mathrm{cl}}}
\mdef\prewfsmorcl{\mathrm{PreWfsMor}_{\mathrm{cl}}}
\mdef\wfsmorcl{\mathrm{WfsMor}_{\mathrm{cl}}}
\mdef\quillencl{\mathrm{Quillen}_{\mathrm{cl}}}
\mdef\prewfsbimorcl{\mathrm{PreWfsBiMor}_{\mathrm{cl}}}
\mdef\wfsbimorcl{\mathrm{WfsBiMor}_{\mathrm{cl}}}
\mdef\biquillencl{\mathrm{BiQuillen}_{\mathrm{cl}}}

\def\eql{\mathrm{eq}}

\let\aut\nAut
\newcommand{\biglue}[1]{\ensuremath{\mathcal{G}\hspace{-1pt}\ell(#1)}\xspace}
\let\glue\biglue
\newcommand{\blank}{\mathord{\hspace{1pt}\text{--}\hspace{1pt}}}
\newcommand{\aand}{\qquad\text{and}\qquad}
\newcommand{\wlim}[3]{\{#1,#2\}^{#3}}
\newcommand{\wlimc}[2]{\wlim{#1}{#2}{\C}}
\newcommand{\wcolim}[3]{#1 \otimes_{#3} #2}
\newcommand{\wcolimc}[2]{\wcolim{#1}{#2}{\C}}
\newcommand{\coll}[1]{[#1]}
\let\tc\lone
\newcommand{\rep}[1]{{#1}_\bullet}
\newcommand{\corep}[1]{{#1}^\bullet}
\newcommand{\cpw}{\cdot}

\newcommand{\rup}[1]{\overrightarrow{#1}}
\newcommand{\rdn}[1]{\overleftarrow{#1}}

\newcommand{\strsl}[2]{#1 \sslash #2}
\let\codisc\nabla

\begin{document}

\begin{abstract}
  We observe that the Reedy model structure on a diagram category can be constructed by iterating an operation of ``bigluing'' model structures along a pair of functors and a natural transformation.
  This yields a new explanation of the definition of Reedy categories: they are almost exactly those small categories for which the category of diagrams and its model structure can be constructed by iterated bigluing.
  It also gives a consistent way to produce generalizations of Reedy categories, including the generalized Reedy categories of Cisinski and Berger--Moerdijk and the enriched Reedy categories of Angeltveit, but also new versions such as a combined notion of ``enriched generalized Reedy category''.
\end{abstract}

\maketitle

\setcounter{tocdepth}{1}
\tableofcontents

\section{Introduction}
\label{sec:introduction}

\subsection{Overview}
\label{sec:overview}

In the preprint~\cite{reedy}, Reedy described an inductive procedure for defining simplicial objects and maps between them, and used it to construct a fairly explicit model structure on the category of simplicial objects in any model category.
Kan then generalized this model structure to diagrams indexed by any \emph{Reedy category}, a notion he defined by abstracting some properties of the domain category $\lDelta\op$ for simplicial objects.
The notion of Reedy category has since become standard in the model-categorical literature, appearing among other places in the books~\cite[Ch.~5]{hovey:modelcats}, \cite[Ch.~15]{hirschhorn:modelcats}, \cite[\S22]{dhks:holim}\footnote{But see \cref{sec:freedy}.}, \cite[\S A.2.9]{lurie:higher-topoi}, and \cite[Ch.~14]{riehl:cht}.

However, there is no obvious reason why Kan's definition is the \emph{only} or the \emph{best} context in which to formulate Reedy's construction abstractly.
And indeed, several people have generalized the definition of Reedy category, maintaining the important property that Reedy diagrams in a model category inherit a fairly explicit model structure.
For instance, Cisinski~\cite[Def.~8.1.1]{cisinski:presheaves} and Berger--Moerdijk~\cite{bm:extn-reedy} have modified the definition to allow a Reedy category to contain nontrivial automorphisms, while Angeltveit~\cite{angeltveit:enr-reedy} has given a version for enriched categories.
However, none of these definitions appear inevitable either.
Perhaps the most satisfying account of Reedy categories to date is that of Riehl and Verity~\cite{rv:reedy}, who show that Kan's definition implies a decomposition of the hom-functor as a cell complex, which in turn induces analogous decompositions of all diagrams.
But the definition of Reedy category itself is still taken as given.

In this paper we analyze Reedy categories and Reedy model structures from the other direction: instead of ``what is a Reedy category?''\ we ask ``what is a Reedy model structure?''\ and work backwards to the definition of a Reedy category.
The classical result is that if \C is a Reedy category and \M a model category, then the diagram category $\M^\C$ has a model structure in which a map $A\to B$ is
\begin{itemize}
\item \dots a weak equivalence iff $A_x\to B_x$ is a weak equivalence in $\M$ for all $x\in \C$.
\item \dots a cofibration iff the induced map $A_x \sqcup_{L_x A} L_x B \to B_x$ is a cofibration in $\M$ for all $x\in \C$.
\item \dots a fibration iff the induced map $A_x \to B_x \times_{M_x B} M_x A$ is a fibration in $\M$ for all $x\in \C$.
\end{itemize}
Here $L_x$ and $M_x$ are the \emph{latching object} and \emph{matching object} functors, which are defined in terms of the Reedy structure of \C.
However, at the moment all we care about is that if $x$ has degree $n$ (part of the structure of a Reedy category is an ordinal-valued degree function on its objects), then $L_x$ and $M_x$ are functors $\M^{\C_n} \to \M$, where $\C_n$ is the full subcategory of \C on the objects of degree less than $n$.
In the prototypical example of $\lDelta\op$, where $\M^{\C}$ is the category of simplicial objects in $\M$, $L_n A$ is the ``object of degenerate $n$-simplices'' whereas $M_n A$ is the ``object of simplicial $(n-1)$-spheres (potential boundaries for $n$-simplices)''.

The fundamental observation which makes the Reedy model structure tick is that if we have a diagram $A\in \M^{\C_n}$, then to extend it to a diagram defined at $x$ as well, it is necessary and sufficient to give an object $A_x$ and a factorization $L_x A \to A_x \to M_x A$ of the canonical map $L_x A \to M_x A$ (and similarly for morphisms of diagrams).
For $\lDelta\op$, this means that if we have a partially defined simplicial object with objects of $k$-simplices for all $k<n$, then to extend it with $n$-simplices we have to give an object $A_n$, a map $L_n A \to A_n$ including the degeneracies, and a map $A_n \to M_n A$ assigning the boundary of every simplex, such that the composite $L_n A \to A_n \to M_n A$ assigns the correct boundary to degenerate simplices.

Categorically speaking, this observation can be reformulated as follows.
Given a natural transformation $\alpha : F\to G$ between parallel functors $F,G:\M\to \N$, let us define the \emph{bigluing category} $\biglue{\alpha}$ to be the category of quadruples $(M,N,\phi,\gm)$ such that $M\in \M$, $N\in\N$, and $\phi:F M \to N$ and $\gm : N \to G M$ are a factorization of $\alpha_M$ through $N$.
The above observation is then that $\M^{C_x}\simeq \biglue{\alpha}$, where $\alpha: L_x \to M_x$ is the canonical map between functors $\M^{C_n} \to \M$ and $C_x$ is the full subcategory of \C on $\C_n \cup \{x\}$.
Moreover, it is an easy exercise to reformulate the usual construction of the Reedy model structure as a theorem that if \M and \N are model categories and $F$ and $G$ are left and right Quillen respectively, then $\biglue{\alpha}$ inherits a model structure.

Therefore, our answer to the question ``what is a Reedy model structure?''\ is that it is one obtained by repeatedly (perhaps transfinitely) bigluing along a certain kind of transformation between functors $\M^\C \to \M$ (where $\C$ is a category playing the role of $\C_n$ previously).
This motivates us to ask, given \C, how can we find functors $F,G : \M^{\C}\to \M$ and a map $\alpha : F \to G$ such that $\biglue\alpha$ is of the form $\M^{\C'}$ for some new category $\C'$?

Of course, we expect $\C'$ to be obtained from $\C$ by adding one new object ``$x$''.
Thus, it stands to reason that $F$, $G$, and $\alpha$ will have to specify, among other things, the morphisms \emph{from} $x$ \emph{to} objects in $\C$, and the morphisms \emph{to} $x$ \emph{from} objects of $\C$.
These two collections of morphisms form diagrams $W:\C\to \nSet$ and $U:\C\op \to \nSet$, respectively; and given such $U$ and $W$ we do have canonical functors $F$ and $G$, namely the $U$-weighted colimit and the $W$-weighted limit.
Moreover, a natural transformation from the $U$-weighted colimit to the $W$-weighted limit can naturally be specified by giving a map $W\times U \to C(\blank,\blank)$ in $\nSet^{\C\op\times \C}$.
In $\C'$, this map will supply the composition of morphisms through $x$.

It remains only to specify the hom-set $\C'(x,x)$ (and the relevant composition maps), and for this there is a ``universal choice'': we take $\C'(x,x) = (W \otimes_\C U) \sqcup \{\id_x\}$.
That is, we throw in composites of morphisms $x\to y \to x$, freely subject to the associative law, and also an identity morphism.
This $\C'$ has a universal property (it is a ``collage'' in the bicategory of profunctors) which ensures that the resulting biglued category is indeed equivalent to $\M^{\C'}$.

A category with degrees assigned to its objects can be obtained by iterating this construction if and only if any nonidentity morphism between objects of the same degree factors uniquely-up-to-zigzags through an object of strictly lesser degree (i.e. the category of such factorizations is connected).
What remains is to ensure that the resulting latching and matching objects are left and right Quillen.
It turns out that this is equivalent to requiring that morphisms between objects of \emph{different} degrees also have connected or empty categories of factorizations through objects of strictly lesser degree.
For a precise statement, see \cref{defn:almost-reedy}; we will call these \textbf{almost-Reedy categories}.

This doesn't look much like the usual definition of Reedy category, but it turns out to be very close to it.
If $\C$ is almost-Reedy, let $\rup \C$ (resp.~$\rdn \C$) be the class of morphisms $f:x\to y$ such that $\deg(x)\le \deg(y)$ (resp.~$\deg(y)\le \deg(x)$) and that do not factor through any object of strictly lesser degree than $x$ and $y$.
We will show that just as in a Reedy category, every morphism factors uniquely into a $\rdn\C$-morphism followed by a $\rup\C$-morphism.

The only thing missing from the usual definition of a Reedy category, therefore, is that $\rdn\C$ and $\rup\C$ be subcategories, i.e.\ closed under composition.
And indeed, this can fail to be true; see \cref{eg:almost-reedy}.
However, this is all that can go wrong: $\C$ is a Reedy category if and only if it is an almost-Reedy category such that $\rdn\C$ and $\rup\C$ are closed under composition.
(In particular, this means that $\rdn\C$ and $\rup\C$ don't have to be given as data in the definition of a Reedy category; they are recoverable from the degree function.
This was also noticed by~\cite{rv:reedy}.)

It is unclear whether the extra generality of almost-Reedy categories is useful: it is more difficult to verify in practice and seems unlikely to have natural applications.
The point is rather that the notion of Reedy category (very slightly generalized) arises inevitably from the process of iterated bigluing.


Moreover, as often happens when we reformulate a definition more conceptually, it becomes easier to generalize.
Iterated bigluing can also be applied to more general types of input data, unifying various sorts of generalized Reedy-like categories in a common framework.
For instance, bigluing with $\M^\G$ rather than $\M$ itself, where \G is a groupoid, yields essentially the notion of \emph{generalized Reedy category} from~\cite{bm:extn-reedy}.
Bigluing with $\M^\D$ for an arbitrary category \D yields a ``more generalized'' notion of Reedy category.
Bigluing enriched categories instead yields a notion of \emph{enriched Reedy category} that generalizes the definition of~\cite{angeltveit:enr-reedy}.
Finally, combining these two generalizations we obtain a new notion of \emph{enriched generalized Reedy category}.

Other potential generalizations also suggest themselves; the idea of ``Reedy-ness'' as presented here applies naturally in any ``category theory'' where we have profunctors and collages.
For instance, one could consider \emph{internal} Reedy categories in a well-behaved model category, or \emph{enriched indexed} Reedy categories in the sense of~\cite{shulman:eicats}.

\subsection{Outline}
\label{sec:outline}

In \cref{sec:model} we make some simple observations about model categories and weak factorization systems which may not be exactly in the literature.
Then in \cref{sec:bigluing} we introduce the \emph{bigluing construction} that forms the foundation of the paper.
But rather than jumping right into its application to Reedy categories, to smooth the exposition in \cref{sec:inverse-categories} we first describe the special case of \emph{inverse categories}, which correspond to the special case of bigluing called \emph{gluing}.
This allows us to introduce the central idea of \emph{collages} in a less complicated setting.

In \cref{sec:functoriality} we return to the general theory of bigluing, proving some functoriality results that are necessary for the application to Reedy categories.
Then in \cref{sec:iterated-bigluing} we study general diagram categories that are obtained by iterated bigluing, and their domains, which we call \emph{bistratified categories}.
Reedy categories are a special sort of bistratified category; they are almost exactly those bistratified categories with discrete strata whose diagram categories inherit model structures from the bigluing construction.
In \cref{sec:reedy} we explain this: we introduce the intermediate notion of \emph{almost-Reedy} category, which are \emph{precisely} the above class of bistratified categories, and show that Reedy categories are the almost-Reedy categories for which $\rup\C$ and $\rdn\C$ are closed under composition.
The proof of this makes essential use of a factorization lemma which was also central to~\cite{rv:reedy}.
We show furthermore that Reedy categories are also exactly the almost-Reedy categories such that latching and matching objects can be computed as colimits and limits over $\rup\C$ and $\rdn\C$ only, respectively; this gives an additional argument for their naturalness.

The last few sections are concerned with generalizations.
In \cref{sec:auto} we relax the requirement of discrete strata, obtaining a notion of \emph{c-Reedy category} which generalizes the ``generalized Reedy categories'' of~\cite{bm:extn-reedy}.
In~\cite{bm:extn-reedy} the strata are required to be groupoids; we can even do without that hypothesis.
In \cref{sec:enriched} we generalize instead to enriched category theory, obtaining a notion of \emph{Reedy \V-category} which generalizes the Reedy \V-categories of~\cite{angeltveit:enr-reedy}.
Finally, in \cref{sec:enriched-generalized} we combine the previous two generalizations, obtaining a notion of \emph{c-Reedy \V-category} which permits non-discrete strata.
As an example, we show that the construction in~\cite{angeltveit:enr-reedy} of a Reedy \V-category associated to a non-symmetric operad can be generalized to symmetric operads.

\subsection{Prerequisites}
\label{sec:prerequisites}

Technically, this paper does not depend on any prior knowledge of Reedy categories, or even their definition.
However, space does not permit the inclusion of appropriate motivation and examples for this paper to serve as a good \emph{introduction} to Reedy categories.
Thus I recommend the reader have some previous acquaintance with them; this can be obtained from any of the books cited in the first paragraph of this introduction.
I will also assume familiarity with basic notions of model category theory, which can be obtained from the same sources.

\subsection{A Formalization}
\label{sec:a-formalization}

This paper is slightly unusual, for a paper in category theory, in that one of its main results (\cref{thm:reedy}: almost-Reedy categories are almost Reedy categories) depends on a sequence of technical lemmas, and as far as I know there is no particular reason to \emph{expect} it to be true.
That Reedy categories are almost-Reedy is essentially necessitated by the fact that Kan's definition suffices to construct Reedy model structures, but I have no intuitive argument for why adding closure under composition is enough for the converse; it just happens to work out that way.
Thus, it is natural to worry that a small mistake in one of the technical lemmas might bring the whole theorem crashing down.

To allay such concerns (which I shared myself), I have formally verified this theorem using the computer proof assistant Coq.
Verifying \emph{all} the results of this paper would require a substantial library of basic category theory, but fortunately the proof of \cref{thm:reedy} (including the technical lemmas) is largely elementary, requiring little more than the definition of a category.
(In fact, the formalization is a proof of the more general \cref{thm:creedy}, assuming as given the elementary characterization of almost-c-Reedy categories from \cref{thm:creedy-char}.)
The Coq code is available online on the arXiv and from my website; in \cref{sec:formalization} below I include some remarks on how it relates to the proofs as given in the paper proper.

\subsection{Acknowledgments}
\label{sec:acknowledgments}

I would like to thank Emily Riehl for several useful discussions about~\cite{rv:reedy} and feedback on a draft of this paper; David White for drawing my attention to the existing terminology ``couniversal weak equivalence'' (\cref{defn:univwe}); Karol Szumi\l{}o for the proof of \cref{thm:acof-cat}; Michal Przybylek for the example in footnote~\ref{fn:chu}; Charles Rezk for \cref{eg:fsreedy1}; and Justin Noel for a conversation about enriched generalized Reedy categories.
Some of the results herein were obtained independently (and a little earlier) by Richard Garner.
I would also like to thank Arthur, for sleeping peacefully while I wrote this paper.

\section{Weak factorization systems and model categories}
\label{sec:model}

Recall the following definitions.

\begin{defn}\label{defn:wfs}
  A \textbf{weak factorization system}, or \textbf{wfs}, on a category \M is a pair $(\L,\R)$ of classes of maps such that
  \begin{enumerate}
  \item \L is the class of maps having the left lifting property with respect to \R,\label{item:wfs1}
  \item \R is the class of maps having the right lifting property with respect to \L, and\label{item:wfs2}
  \item Every morphism in \M factors as a map in \L followed by a map in \R.\label{item:wfs-fact}
  \end{enumerate}
\end{defn}

\begin{defn}
  A \textbf{model structure} on \M is a triple $(\C,\F,\W)$ of classes of maps such that 
  \begin{itemize}
  \item $(\C\cap\W,\F)$ and $(\C,\F\cap\W)$ are wfs, and
  \item \W satisfies the 2-out-of-3 property.
  \end{itemize}
\end{defn}

This is equivalent to Quillen's original definition of (closed) model category; it can be found in~\cite[14.2.1]{mp:more-concise} or~\cite[11.3.1]{riehl:cht}.
If there is only one wfs of interest on a category, we denote it by $(\L,\R)$; context makes it clear which category we are referring to.
We say that a morphism $f:X\to Y$ in \L is an \textbf{\L-map}, and that $X$ is an \textbf{\L-object} if $\emptyset\to X$ is an \L-map.
Dually, we speak of \R-maps and \R-objects.


We will also occasionally need the following weaker notion.

\begin{defn}
  A \textbf{pre-wfs} is a pair $(\L,\R)$ satisfying \cref{defn:wfs}\ref{item:wfs1} and~\ref{item:wfs2}, but not necessarily~\ref{item:wfs-fact}.
\end{defn}

Our main reason for considering pre-wfs is the following examples.

\begin{lem}
  If \M is a complete and cocomplete category with a pre-wfs and \C is any small category, then $\M^\C$ has two pre-wfs defined as follows:
  \begin{itemize}
  \item In the \textbf{projective pre-wfs}, the \R-maps are the natural transformations that are objectwise in \R, with \L being the maps having the left lifting property with respect to these.
  \item In the \textbf{injective pre-wfs}, the \L-maps are the natural transformations that are objectwise in \L, with \R being the maps having the right lifting property with respect to these.
  \end{itemize}
\end{lem}
\begin{proof}
  We prove the projective case; the injective one is dual.
  With the definitions as given, it suffices to show that any map having the right lifting property with respect to the projective \L-maps is objectwise in \R.

  Note that for any $c\in \C$, the functor ``evaluate at $c$'' $\mathrm{ev}_c:\M^\C\to \M$ has a left adjoint defined by $M\mapsto \C(c,\blank)\cpw M$.
  Thus, if $f:M\to M'$ is an \L-map in \M, then the induced map $\C(c,\blank)\cpw f : \C(c,\blank)\cpw M\to \C(c,\blank)\cpw M'$ has the left lifting property with respect to all maps $X\to Y$ in $\M^\C$ such that $X_c \to Y_c$ is in \R.
  In particular, $\C(c,\blank)\cpw f$ is a projective \L-map.

  But conversely, $X\to Y$ has the right lifting property with respect to all maps of the form $\C(c,\blank)\cpw f$, for $f$ an \L-map in \M, if and only if $X_c \to Y_c$ is in \R.
  Thus, if $X\to Y$ has the right lifting property with respect to all projective \L-maps, it must be objectwise in \R.
\end{proof}

Even if the given pre-wfs on \M is a wfs, this is not necessarily the case for the projective and injective pre-wfs.
This holds for the projective case if the wfs on \M is \emph{cofibrantly generated}, and for the injective case if \M is furthermore a \emph{locally presentable category} (in which case one says that a cofibrantly generated wfs is \emph{combinatorial}).
In either of these cases, model structures also lift to diagram categories.



If \V, \M, and \N are equipped with pre-wfs and \N has pushouts, we say a functor $S:\V\times \M\to\N$ is a \textbf{left wfs-bimorphism} if for any \L-maps $V\to V'$ and $M\to M'$ in \V and \M respectively, the induced dotted map from the pushout
\begin{equation}
  \vcenter{\xymatrix@C=3pc{
      S(V,M)\ar[r]\ar[d] &
      S(V',M)\ar[d] \ar[ddr]\\
      S(V,M')\ar[r] \ar[drr] &
      S(V,M')\sqcup_{S(V,M)} S(V',M)\pushout \ar@{.>}[dr] \\
      && S(V',M')
      }}
\end{equation}
is an \L-map.
We say that $S:\V\op\times \M\to \N$ is a \textbf{right wfs-bimorphism} if $S\op:\V\times \M\op \to \N\op$ is a left wfs-bimorphism.
We will need the well-known fact that these conditions for any of the three functors in a two-variable adjunction imply the corresponding conditions for the other two; see e.g.~\cite[Lemma 11.1.10]{riehl:cht}.
Note in particular that this makes no use of factorizations, so it is true whether or not the pre-wfs in question are wfs.

The ``most basic'' wfs is (injections, surjections) on \nSet.
This is ``universally enriching'' for all pre-wfs, in the following sense.

\begin{thm}\label{thm:copower}
  If \M is complete and cocomplete and equipped with a pre-wfs, then
  \begin{enumerate}
  \item The copower functor $\nSet\times\M\to\M$ is a left wfs-bimorphism,\label{item:copower1}
  \item The power functor $\nSet\op\times\M\to\M$ is a right wfs-bimorphism, and\label{item:copower2}
  \item The hom-functor $\M\op\times \M \to \nSet$ is a right wfs-bimorphism.\label{item:copower3}
  \end{enumerate}
\end{thm}
\begin{proof}
  It suffices to prove one of~\ref{item:copower1}--\ref{item:copower3}.
  The easiest is~\ref{item:copower3}, which claims that if $A\to B$ is an \L-map and $C\to D$ an \R-map, then the induced function
  \[ \M(B,C) \to \M(A,C) \times_{\M(A,D)} \M(B,D) \]
  is a surjection.
  But this is just the lifting property of a pre-wfs.
\end{proof}

We will need to be able to take limits of diagrams of categories equipped with wfs or pre-wfs (hence also model categories).
We could do this by using either a pseudo sort of limit or a strict sort of morphism; for simplicity we choose the latter.

\begin{defn}
  A pre-wfs is \textbf{cloven} if we have chosen a particular lift in every commutative square from an \L-map to an \R-map.
  A wfs is \textbf{cloven} if we have additionally chosen a particular $(\L,\R)$-factorization of every morphism (which need not be functorial).
  A \textbf{strict functor} between cloven wfs or pre-wfs
is one which preserves \emph{both} classes of maps, and also the chosen lifts (and factorizations, in the wfs case) on the nose.
\end{defn}

Let $\prewfscl$ and $\wfscl$
denote the category of categories with cloven pre-wfs and wfs, respectively,
with strict functors between them.

Strict functors do not occur often, but the basic construction in \cref{sec:bigluing} will produce them.
The value of noticing this lies in the following fact.

\begin{thm}\label{thm:clovenlim}
  The forgetful functors $\prewfscl \to\nCat$ and $\wfscl\to\nCat$
  create limits.
\end{thm}
\begin{proof}
  Define a map in $\lim_c \M_c$ to be in \L or \R if it becomes so in each $\M_c$.
  Lifts and factorizations can be defined in each $\M_c$, fitting together because the functors are strict.
  Similarly, \L and \R determine each other because they do in each $\M_c$.
\end{proof}

We can also take limits of functors between diagrams of cloven wfs with strict maps.
Precisely, consider the following categories:
\begin{itemize}
\item \wfsmorcl and $\prewfsmorcl$, whose objects are triples $(\M,\N,F)$ where \M and \N have cloven wfs and pre-wfs, respectively, and $F:\M\to\N$ preserves \L-maps (but is not strict).
  Its morphisms are commutative squares in which the functors $\M\to\M'$ and $\N\to\N'$ are strict.
\item \wfsbimorcl and \prewfsbimorcl, whose objects are quadruples $(\V,\M,\N,S)$ where \V, \M, and \N have cloven wfs or pre-wfs, \N has chosen pushouts, and $S:\V\times \M\to\N$ is a left wfs-bimorphism.
  Their morphisms consist of strict functors $\V\to\V'$, $\M\to\M'$, and $\N\to\N'$ making the obvious square commute, and such that $\N\to\N'$ preserves the chosen pushouts on the nose.
\end{itemize}
We have forgetful functors
\begin{alignat}{2}
  \wfsmorcl &\to \mathrm{Func} &\qquad \wfsbimorcl &\to \mathrm{BiFunc}\\
  \prewfsmorcl &\to \mathrm{Func} &\qquad \prewfsbimorcl &\to \mathrm{BiFunc}
\end{alignat}
where $\mathrm{Func}$ is the category whose objects are functors and whose morphisms are commutative squares of functors, while $\mathrm{BiFunc}$ is the similar category whose objects are bifunctors $S:\V\times \M\to\N$.

\begin{thm}\label{thm:clovenlimmor}
  The above forgetful functors create limits.\qed
\end{thm}

We will primarily be interested in ordinal-indexed limits.
If $\be$ is an ordinal, regarded as a category in the usual way, by a \textbf{continuous $\be$-tower} we mean a continuous functor $\M:\be\op \to \nCat$.
Thus, if $\de<\be$ is a limit ordinal, we have $\M_\de = \nlim_{\de'<\de} \M_{\de'}$.

The dual notion is a \textbf{cocontinuous \be-cotower}.
If $\C:\be\to\nCat$ is such a cotower and \M is a fixed category, the universal property of colimits says that
\begin{equation}
  \M^{\colim_{\de<\be} \C_\de} \cong \lim_{\de<\be} \M^{\C_\de}.\label{eq:cotower}
\end{equation}




\section{Bigluing of model categories}
\label{sec:bigluing}

Let \M and \N be categories, $F,G:\M\to\N$ functors, and $\alpha:F\to G$ a natural transformation.

\begin{defn}
  The \textbf{bigluing} category of $\alpha$, denoted $\biglue{\alpha}$, is defined as follows.
  \begin{itemize}
  \item An object of $\biglue{\al}$ consists of an object $M\in\M$, an object $N\in\N$, and a factorization of $\al_M:FM\to GM$ through $N$, i.e.\ morphisms $\phi:FM \to N$ and $\gm:N\to GM$ whose composite is $\al_M$.
  \item A morphism of $\biglue{\al}$ consists of morphisms $\mu:M\to M'$ and $\nu:N\to N'$ such that the following two squares commute.
    \begin{equation}
      \vcenter{\xymatrix@-.5pc{
          FM\ar[r]^{\phi}\ar[d]_{F\mu} &
          N\ar[r]^{\gm}\ar[d]^{\nu} &
          GM\ar[d]^{G\mu}\\
          FM'\ar[r]_{\phi'} &
          N'\ar[r]_{\gm'} &
          GM'
        }}
    \end{equation}
  \end{itemize}
\end{defn}

The \textbf{gluing construction} is the special case when $F$ is constant at an initial object.
In this case, $\al$ is uniquely determined, so we speak of \emph{gluing along the functor $G$} and write $\glue{G}$ instead of $\biglue{\emptyset\to G}$.
The objects of the resulting category consist of $M\in \M$ and $N\in \N$ together with a map $N\to GM$.
Dually, if $G$ is constant at a terminal object, we may speak of \emph{cogluing along $F$}; this explains the terminology ``bigluing''.

Note that $\biglue{\alpha}\op = \biglue{\al\op}$, where $\al\op :G\op \to F\op$ is the induced transformation between functors $\M\op\to\N\op$.
There are forgetful functors $\biglue{\al} \to \M$ and $\biglue{\al}\to \N$.

\begin{lem}\label{thm:biglue-adjoint}
  The forgetful functor $\biglue{\al} \to \M$ has a left adjoint with identity counit.
  If $G$ has a left adjoint and \N has binary coproducts, then the forgetful functor $\biglue{\al}\to \N$ also has a left adjoint.
\end{lem}
\begin{proof}
  The first left adjoint takes $M$ to $(M,FM,\id,\al_M)$.

  Let $K$ be the left adjoint of $G$; we now construct a left adjoint to $\biglue{\al}\to\N$.
  Given $B\in\N$, consider the object
  \[\Ktil B = (KB,\,FKB\sqcup B,\,i_1,\,[\al_{KB},\eta])\]
  of $\biglue{\al}$, where $i_1$ is the coproduct injection $FKB \to FKB\sqcup B$, and $[\al_{KB},\eta]:FKB\sqcup B \to GKB$ has components $\al_{KB}:FKB \to GKB$ and the unit $\eta:B\to GKB$ of the adjunction $K\dashv G$.
  A map from $\Ktil B$ to $(M,N,\phi,\gm)$ consists of a map $\mu:KB \to M$ and $\nu:FKB\sqcup B \to N$ such that the following two squares commute.
  \begin{equation}
  \vcenter{\xymatrix@C=3pc{
      FKB\ar[r]^-{i_1}\ar[d]_{F\mu} &
      FKB\sqcup B\ar[r]^-{[\al_{KB},\eta]}\ar[d]^{\nu = [\nu_1,\nu_2]} &
      GKB\ar[d]^{G\mu}\\
      FM\ar[r]_\phi &
      N\ar[r]_\gm &
      GM
      }}
  \end{equation}
  The first square just ensures that the first component $\nu_1$ of $\nu$ is the composite $FKB \to FM \to N$.
  Given this, the second square consists of naturality for $\al$ (which is automatically true) together with the assertion that the adjunct $B \to GM$ of $\mu$ under the adjunction $K\dashv G$ is the composite $B \xto{\nu_2} N \xto{\gm} GM$.
  Thus, a map $\Ktil B \to (M,N,\phi,\gm)$ is uniquely determined by a map $\nu_2:B\to N$; so $\Ktil$ is left adjoint to $\biglue{\al} \to \N$.
\end{proof}

Dually, we have:

\begin{lem}
  The forgetful functor $\biglue{\al} \to \M$ has a right adjoint with identity unit.
  If $F$ has a right adjoint and \N has binary products, then the forgetful functor $\biglue{\al}\to \N$ also has a right adjoint.\qed
\end{lem}

We now observe:

\begin{lem}
  If \M and \N are cocomplete and $F$ is cocontinuous, then $\biglue{\al}$ is cocomplete.
  Dually, if \M and \N are complete and $G$ is continuous, then $\biglue{\al}$ is complete.

  In either case, if \M and \N have chosen limits or colimits, we can choose those in $\biglue{\al}$ to be preserved strictly by both forgetful functors.\qed
\end{lem}

Thus, it at least makes sense to ask whether $\biglue{\al}$ inherits a model structure from \M and \N.
We expect to need some further conditions to make this true.

Suppose first that $\M$ and $\N$ are equipped with wfs.
In the following theorem we see the fundamental Reedy-style constructions.

\begin{thm}\label{thm:wfs}
  If \M and \N each have a wfs and \N has pushouts and pullbacks, then $\biglue{\al}$ inherits a wfs defined as follows.
  \begin{itemize}
  \item The \L-maps are those where $M\to M'$ and the induced map $FM' \sqcup_{FM} N \to N'$ are both in \L.
  \item The \R-maps are those where $M\to M'$ and the induced map $N \to GM \times_{GM'} N'$ are both in \R.
  \end{itemize}
  If the wfs on \M and \N are cloven, then the wfs on $\biglue{\al}$ is canonically cloven and the forgetful functor $\biglue{\al} \to \M$ is strict.
\end{thm}
\begin{proof}
  Since both classes of maps are closed under retracts, it suffices to show the factorization and lifting properties.
  The construction should look familiar to any reader who has seen Reedy model structures before.
  We factor a morphism $(\mu,\nu):(M,N,\phi,\gm) \to (M',N',\phi',\gm')$ by first factoring $\mu$ as $M \xto{\L} M'' \xto{\R} M'$ in \M, then performing the dotted factorization below in \N:
  \begin{equation}
    \vcenter{\xymatrix{
        FM\ar[r]\ar[d] &
        N\ar[d] \ar[rrr] &&&
        GM\ar[d]\\
        FM''\ar[r] \ar[d] &
        FM'' \sqcup_{FM} N \ar@{.>}[r]&
        N'' \ar@{.>}[r] &
        GM'' \times_{GM'} N' \ar[d] \ar[r] &
        GM'' \ar[d]\\
        FM' \ar[rrr] &&&
        N' \ar[r] &
        GM'
      }}
  \end{equation}
  For lifting, we suppose given commutative squares
  \begin{equation}\label{eq:wfs-liftsq}
    \vcenter{\xymatrix{
        M_1\ar[r]\ar[d] &
        M_3\ar[d]\\
        M_2\ar[r] &
        M_4
      }}\qquad
    \vcenter{\xymatrix{
        N_1\ar[r]\ar[d] &
        N_3\ar[d]\\
        N_2\ar[r] &
        N_4
      }}
  \end{equation}
  commuting with the $\phi$'s and $\gm$'s and such that the maps
  \[ M_1 \to M_2 \aand FM_2 \sqcup_{FM_1} N_1 \to N_2 \]
  are in \L and the maps
  \[ M_3 \to M_4 \aand N_3 \to GM_3 \times_{GM_4} N_4 \]
  are in \R.
  We lift in the first square in~\eqref{eq:wfs-liftsq} to obtain a map $M_2 \to M_3$, then lift in the following square:
  \begin{equation}
    \vcenter{\xymatrix{
        FM_2 \sqcup_{FM_1} N_1 \ar[r]\ar[d] &
        N_3\ar[d]\\
        N_2 \ar[r] &
        GM_3 \times_{GM_4} N_4
      }}
  \end{equation}
  whose top and bottom maps are induced by the given ones and the lift $M_2 \to M_3$.
  
  The final claim is clear from the above explicit constructions.
\end{proof}

This is also true for pre-wfs if our functors have adjoints.

\begin{thm}\label{thm:prewfs}
  If \M and \N each have a pre-wfs, \N has finite limits and colimits, and $F$ has a right adjoint and $G$ has a left adjoint, then $\biglue{\al}$ inherits a pre-wfs defined as in \cref{thm:wfs}.
\end{thm}
\begin{proof}
  It remains to show that the classes \L and \R in $\biglue{\al}$ determine each other by lifting properties.
  (In the case of wfs, this follows from the factorization property and closure under retracts.)

  Suppose that a map $(M,N,\phi,\gm) \to (M',N',\phi',\gm')$ in $\biglue{\al}$ has the right lifting property with respect to \L.
  Since the left adjoint of the forgetful functor $\biglue{\al} \to \M$ preserves \L-maps by inspection, the underlying map $M\to M'$ must be in \R.

  Now let $K$ be the left adjoint of $G$ and $\Ktil$ the resulting left adjoint to $\biglue{\al} \to \N$ as in \cref{thm:biglue-adjoint}.
  Let $f:A\to B$ be an \L-map in \N, and consider the object 
  \[\Khat f = (KB,\,FKB\sqcup A,\,i_1,\,[\al_{KB},\eta f])\]
  of $\biglue{\al}$.
  A similar analysis to the proof of \cref{thm:biglue-adjoint} reveals that a map from $\Khat f$ to $(M,N,\phi,\gm)$ consists of a commutative square
  \begin{equation}\label{eq:khat-sq}
  \vcenter{\xymatrix{
      A\ar[r]\ar[d]_f &
      N\ar[d]^{\gm}\\
      B\ar[r] &
      GM.
      }}
  \end{equation}
  We have an induced map $\Khat f \to \Ktil B$, precomposition with which induces a commutative square~\eqref{eq:khat-sq} from knowing its diagonal.
  It follows that to give a commutative square as on the left in $\biglue{\al}$:
  \begin{equation}
    \vcenter{\xymatrix{
        \Khat f\ar[r]\ar[d] &
        (M,N,\phi,\gm)\ar[d]\\
        \Ktil B\ar[r] &
        (M',N',\phi',\gm')
      }}
    \qquad
    \vcenter{\xymatrix{
        A\ar[r]\ar[d] &
        N\ar[d]\\
        B\ar[r] &
        N' \times_{GM'} GM
      }}  
  \end{equation}
  is equivalent to giving a commutative square on the right in \N, and diagonal liftings likewise correspond.

  Finally, this map $\Khat f \to \Ktil B$ is in \L, since $\id_{KB}:KB \to KB$ is an \L-map in \M and $FKB\sqcup A \to FKB\sqcup B$ is an \L-map in \N (being the coproduct of the two \L-maps $\id_{FKB}$ and $f:A\to B$).
  Thus, if $(M,N,\phi,\gm) \to (M',N',\phi',\gm')$ has the right lifting property with respect to \L, then it must have the right lifting property with respect to $\Khat f \to \Ktil B$, and hence $N \to N'\times_{GM'} GM$ must be an \R-map in \N, as desired.
\end{proof}

In particular, if \M and \N are model categories, then both of their constituent wfs lift to $\biglue{\al}$.
To fit together these wfs into a model structure on $\biglue{\al}$, however, we need an additional hypothesis, for which purpose we recall the following definition from~\cite{bb:htapm}.



\begin{defn}\label{defn:univwe}
  A morphism $f$ in a model category is a \textbf{couniversal weak equivalence} if every pushout of $f$ is a weak equivalence.
  Dually, $f$ is a \textbf{universal weak equivalence} if every pullback of it is a weak equivalence.
\end{defn}

Every acyclic cofibration is a couniversal weak equivalence.
More generally, if there happens to be some other model structure on the same category with the same or smaller class of weak equivalences, then any acyclic cofibration in that other model structure will also be a couniversal weak equivalence.
In addition, we have the following observation:

\begin{lem}\label{thm:objwise-cwe}
  Let \M be a model category and \C any small category, and suppose $\M^\C$ has any model structure in which the weak equivalences are objectwise.
  Then any objectwise acyclic cofibration in $\M^\C$ is a couniversal weak equivalence in this model structure, and similarly any objectwise acyclic fibration is a universal weak equivalence.
\end{lem}
\begin{proof}
  This follows immediately from the fact that limits, colimits, and weak equivalences in $\M^\C$ are objectwise.
\end{proof}

Note that \cref{thm:objwise-cwe} does not require the existence of a projective or injective model structure on $\M^\C$.

\begin{thm}\label{thm:model}
  Suppose \M and \N are model categories and we have $F,G:\M\to\N$ and $\al:F\to G$ such that
  \begin{itemize}
  \item $F$ is cocontinuous and takes acyclic cofibrations to couniversal weak equivalences.
  \item $G$ is continuous and takes acyclic fibrations to universal weak equivalences.
  \end{itemize}
  Then $\biglue{\al}$ is a model category, and the forgetful functors to \M and \N preserve cofibrations, fibrations, and weak equivalences.
\end{thm}
\begin{proof}
  By \cref{thm:wfs}, $\biglue{\al}$ inherits two wfs which we call (cofibration, acyclic fibration) and (acyclic cofibration, fibration).
  (For the moment, we treat ``acyclic cofibration'' and ``acyclic fibration'' as atomic names for classes of morphisms in $\biglue{\al}$.)
  We define a map in $\biglue{\al}$ to be a weak equivalence just when both $M\to M'$ and $N\to N'$ are.

  Now if a cofibration in $\biglue{\al}$ is also a weak equivalence, then its underlying map $M\to M'$ in \M is an acyclic cofibration.
  Hence $FM \to FM'$ is a couniversal weak equivalence, thus its pushout $N \to FM' \sqcup_{FM} N$ is a weak equivalence.
  Since $N\to N'$ is a weak equivalence, by the 2-out-of-3 property, $FM' \sqcup_{FM} N \to N'$ is also a weak equivalence, and hence an acyclic cofibration.
  Thus, our original map was an acyclic cofibration.

  The same argument with the 2-out-of-3 property applied in the other direction shows that every acyclic cofibration is a weak equivalence.
  Thus, the acyclic cofibrations in $\biglue{\al}$ are precisely the cofibrations that are also weak equivalences, and dually for the acyclic fibrations.
  Since the weak equivalences obviously satisfy the 2-out-of-3 property, we have a model structure.
\end{proof}

Note that the hypotheses are satisfied if $F$ is left Quillen and $G$ is right Quillen.
The more general hypothesis will be useful in \cref{sec:auto}.



\begin{rmk}\label{rmk:grothendieck}
  Richard Garner has noted that the bigluing construction can be seen as an instance of the \emph{Grothendieck construction} for wfs described in~\cite{roig:bifibred,aes:mdl-smcat,hp:groth-model}.
  Namely, when $\N$ has pullbacks and pushouts, the forgetful functor $\biglue{\al} \to \M$ is a Grothendieck bifibration (i.e.\ both a fibration and an opfibration), whose fiber over $M\in \M$ is the category of factorizations of $\al_M$.
  If \M and \N have wfs, then the wfs of \N also induces a wfs on these fibers, and so the Grothendieck construction yields the biglued wfs on the total category $\biglue{\al}$.
\end{rmk}

\section{Inverse categories, with an introduction to collages}
\label{sec:inverse-categories}

As a warm-up before diving into Reedy categories, we start with \emph{inverse categories}.
These turn out to correspond to \emph{iterated gluing} in the same way that Reedy categories will correspond to iterated bigluing (and their duals, \emph{direct categories}, correspond to iterated cogluing).

As we will do later for Reedy categories, we will not assume the notion of inverse category \emph{a priori}.
Instead, in this section our goal will be to construct the diagram categories $\M^\C$, for some yet-to-be-determined class of categories \C, by iterated gluing.
Specifically, we will repeatedly glue some previously constructed category $\M^\C$ with $\M$ along a functor $G:\M^\C\to\M$.
By \cref{thm:model}, if we want the construction to preserve model structures, it is reasonable to take $G$ to be right Quillen.
A natural class of functors $\M^\C\to\M$, which we may hope will at least sometimes be right Quillen, is provided by \emph{weighted limits}.

Recall that if $X:\C\to\M$ and $W:\C\to\nSet$, the \textbf{$W$-weighted limit of $X$} is defined to be the functor cotensor product
\[ \wlimc W X = \eql\left( \prod_{c} X(c)^{W(c)} \toto \prod_{c,c'} X(c)^{\C(c',c) \times W(c')}\right). \]
If \M is complete, this defines a functor $\wlimc W \blank : \M^\C \to \M$, which has a left adjoint that sends $Y\in\M$ to the copower defined by $(Y\cpw W)(c) = Y \cpw W(c)$.
In particular, $\wlimc W \blank$ is continuous, so gluing along it will at least preserve completeness and cocompleteness.

We would now like to define an \emph{inverse category} to be a category \C such that for any complete category \M, the category $\M^\C$ can be obtained from \M by repeated gluing along weighted limit functors.
However, in order for that to be sensible, we need to know that if we glue along a weighted limit functor $\wlimc W \blank : \M^\C\to \M$, the resulting category is (at least sometimes) again a diagram category $\M^{\C'}$ for some other category $\C'$.
Fortunately, as we will now show, this is \emph{always} the case.

Recall that for categories \C and \D, a \textbf{profunctor} from $\C$ to $\D$, denoted $H:\C \hto \D$, is a functor $H:\D\op\times\C \to \nSet$.
Profunctors are the morphisms of a bicategory \cProf, where composition is the functor tensor product: the composite of $H:\C\hto \D$ and $K:\D\hto \E$ is defined by
\[ (K H)(e,c) = \ncoeq\left( \coprod_{d,d'} K(e,d) \times \D(d,d') \times H(d',c) \toto \coprod_d K(e,d) \times H(d,c) \right). \]

\begin{defn}
  The \textbf{collage} of a profunctor $H:\C\hto\D$, which we will write $\coll{H}$, is the category whose objects are the disjoint union of those of \C and \D, and whose hom-sets are
  \begin{alignat*}{2}
    \coll{H}(c,c') &= \C(c,c') &\qquad
    \coll{H}(d,c) &= H(d,c)\\
    \coll{H}(d,d') &= \D(d,d') &\qquad
    \coll{H}(c,d) &= \emptyset
  \end{alignat*}
  with composition and identities induced from those of \C and \D and the functoriality of $H$.
\end{defn}

In particular, we may consider a weight $W:\C\to\nSet$ to be a profunctor from \C to the terminal category \tc.
Its collage is then a category $\coll{W}$ containing \C as a full subcategory, with one additional object $\star$ added, and with additional morphisms $\coll{W}(\star,c) = W(c)$.

Of central importance for us is that the collage of any profunctor has a universal property in the bicategory of profunctors: it is a \emph{representable lax colimit}.
To explain this, recall that any functor $P:\C\to\D$ induces a \textbf{representable profunctor} $\rep P: \C \hto \D$, defined by 
\[ \rep P(d,c) = \D(d,P(c)). \]
This defines the action on morphisms of an identity-on-objects pseudofunctor $\rep{(\blank)} : \cCat \to \cProf$.
Moreover, this pseudofunctor is locally fully faithful, i.e.\ a natural transformation $P\to Q$ is uniquely determined by its image $\rep P \to \rep Q$.
Finally, every representable profunctor $\rep P$ has a right adjoint profunctor $\corep P$ defined by 
\[ \corep P(c,d) = \D(P(c),d). \]

\begin{defn}
  A \textbf{representable lax cocone with vertex \E} under a profunctor $H:\C\hto \D$ consists of functors $P:\D\to \E$ and $Q:\C\to \E$ and a natural transformation $\rep P H \to \rep Q$.
  A morphism of such cocones (with the same vertex) consists of transformations $P\to P'$ and $Q\to Q'$ such that an evident square commutes.
\end{defn}

By the adjointness $\rep P \adj \corep P$, a transformation $\rep P H \to \rep Q$ can equivalently be regarded as a transformation $H \to \corep P \rep Q$.
Moreover, the ``co-Yoneda lemma'' implies that $(\corep P \rep Q)(d,c) \cong \E(Pd,Qc)$.
So a representable lax cocone under $H:\C\hto \D$ consists of $P:\D\to \E$ and $Q:\C\to \E$ together with ``an interpretation of the elements of $H$ as morphisms from the image of $P$ to the image of $Q$.''
This makes the following fact fairly clear.

\begin{thm}\label{thm:collage}
  Any profunctor $H:\C\hto \D$ admits a representable lax cocone with vertex $\coll H$, where $P$ and $Q$ are the inclusions and the transformation $H(d,c) \to \coll H(Pd,Qc)$ is the identity.
  Moreover, this is the \textbf{universal} representable lax cocone, in the sense that the category of representable lax cocones under $H$ with vertex \E is naturally isomorphic to the category of functors $\coll H \to \E$.\qed
\end{thm}

In fact, $\coll H$ is also a lax colimit in the bicategory \cProf: the universal representable lax cocone is also the universal non-representable lax cocone.
For more on this universal property of collages, see~\cite{street:cauchy-enr,wood:proarrows-ii,gs:freecocomp}.
Our purpose in recalling it is to prove the following theorem.

\begin{thm}\label{thm:gluing=collage}
  For any category \C and weight $W:\C\to\nSet$, if we glue along the $W$-weighted limit functor $\wlimc W\blank : \M^\C \to \M$, we have $\glue{\wlimc W \blank} \eqv \M^{\coll{W}}$.
\end{thm}
\begin{proof}
  In view of \cref{thm:collage}, it will suffice to show that $\glue{\wlimc W \blank}$ is equivalent to the category of representable lax cocones under $W$ with vertex \M.
  By definition, an object of $\glue{\wlimc W \blank}$ consists of $Q\in\M^\C$ and $P\in \M$ and a map $P \to \wlimc W Q$.
  But this is equivalent to
  a map $W \to \M(P,Q)$ in $\nSet^\C$, which we can regard as a map $W \to \wlim{\rep P}{\rep Q}{\tc}$ in $\cProf(\C,\tc)$, if we identify the object $P$ with the corresponding functor $\tc \to \M$.
  Finally, by adjunction this is equivalent to a map $\wcolim{\rep P}{W}{\tc} \to \rep Q$, as required.
  (Note that our map $P \to \wlimc W Q$ is also equivalent by adjunction to a map $P\cpw W \to Q$ in $\M^\C$, but $\rep {(P\cpw W)}$ is not isomorphic to $\wcolim{\rep P}{W}{\tc}$, because the Yoneda embedding does not preserve colimits.)
\end{proof}

Therefore, the categories obtained from \M by repeated gluing along weighted limit functors are all diagram categories.
So we may define the \emph{inverse categories} to be the diagram shapes obtained in this way, e.g.\ we may define an \emph{inverse category of height $0$} to be the empty category, and an \emph{inverse category of height $n+1$} to be one of the form $\coll{W}$, where $W:\C\to\nSet$ and \C is an inverse category of height $n$.

However, in this way we will obtain only inverse categories with finitely many objects, with a unique object being added at each positive integer height.
These are sufficient for some purposes, but often we need to generalize the iteration by increasing either the ``width'' or the ``height'' (or both).

Increasing the \emph{width} means adding more objects at each stage.
This is easy to do by gluing along a functor $\M^\C \to \M^{I}$, where $I$ is a discrete set, instead of a functor $\M^\C\to \M$.
Such a functor is just an $I$-indexed family of functors $\M^\C\to \M$, each of which we may take to be a weighted limit $\wlimc {W_i}\blank$ for some weights $W_i:\C\to\nSet$.

We can regard these weights together as forming a profunctor $W:\C\hto I$.
In general, for any categories \C and \D and any profunctor $W:\C\hto \D$, if \M is complete then we have a functor $\M^\C \to \M^\D$ denoted $\wlimc W \blank$, defined by
\[ \wlimc W X (d) = \wlimc{W(d,\blank)}{X}. \]
\cref{thm:gluing=collage} then has the following easy generalization:

\begin{thm}\label{thm:gluing=collage2}
  For any profunctor $W:\C\hto \D$, if we glue along $\wlimc W\blank : \M^\C \to \M^\D$, we have $\glue{\wlimc W \blank} \eqv \M^{\coll{W}}$.\qed
\end{thm}

For now, however, we restrict attention to the case when $\D=I$ is discrete.

So much for increasing the width.
Increasing the \emph{height} means allowing the iteration to continue into the transfinite.
We take colimits of the inclusions $\C\into \coll{W}$ forming a cocontinuous cotower, and apply~\eqref{eq:cotower} to the categories $\M^\C$.
This leads to the following definition.

\begin{defn}\label{defn:inverse}
  The collection of \textbf{inverse categories of height $\beta$} is defined by transfinite recursion over ordinals $\beta$ as follows.
  \begin{enumerate}
  \item The only inverse category of height $0$ is the empty category.
  \item The inverse categories of height $\be+1$ are those of the form $\coll{W}$, where $W:\C\hto I$ is a profunctor with $I$ a discrete set and $\C$ is an inverse category of height $\be$.\label{item:inverse2}
  \item If $\be$ is a limit ordinal, the inverse categories of height $\be$ are the colimits of cocontinuous \be-cotowers in which each successor morphism $\C_\de\to\C_{\de+1}$ is an inclusion $\C_\de\into \coll{W}$ as in~\ref{item:inverse2}, where $\C_\de$ is an inverse category of height $\de$.
  \end{enumerate}
\end{defn}

We say that an object of an inverse category has \textbf{degree $\de$} if it was added by the step $\C_\de \to \coll{W_\de}= \C_{\de+1}$.
Thus, an inverse category of height $\beta$ has objects of degrees $<\beta$.
It is easy to give a more direct characterization of inverse categories as well (which is usually taken as the definition):

\begin{thm}\label{thm:inverse-char}
  A small category \C is inverse if and only if we can assign an ordinal degree to each of its objects such that every nonidentity morphism strictly decreases degree.
\end{thm}
\begin{proof}
  It is evident that every inverse category has this property.
  For the converse, we induct on the supremum $\beta$ of the degrees of objects of \C.
  For any $\de\le\be$, let $\C_\de$ be the subcategory of objects of degree $<\de$ and $I_\de$ the set of objects of degree $\de$, and let $W_\de:\C_\de\hto I_\de$ be the restriction of the hom-functor of \C.
  If either $\be$ is a successor or $\de<\be$, then the supremum of degrees of objects of $\C_\de$ is strictly less than $\be$, so that by the inductive hypothesis it is inverse.

  Now if $\be$ is a successor, choose $\de$ to be its predecessor; then $\C \cong \biglue{\wlim {W_\de}{\blank}{\C_\de}}$, so that $\C$ is inverse by the successor clause in \cref{defn:inverse}.
  Otherwise, it is the colimit of $\C_\de$ for $\de<\be$, so that $\C$ is inverse by the limit clause in \cref{defn:inverse}.
\end{proof}

Generally when we say that \C is inverse, we mean that we have chosen a particular decomposition of it according to \cref{defn:inverse}, and in particular a degree function on its objects.

\begin{rmk}\label{rmk:transfinite-degrees}
  Just as with cell complexes in a model category, it is always possible to reassign degrees in an inverse category so that there is exactly one object of each degree.
  However, this is unnatural and causes the ordinal height to balloon uncontrollably.
  For instance, when considering the category of simplices of a simplicial set, it is much more natural to say that the degree of each simplex is its dimension, keeping the height at a manageable $\omega$ regardless of how many simplices there are.

  On the other hand, there \emph{are} naturally occurring examples where the generalization to transfinite heights is essential.
  For instance, as observed in~\cite[Examples 1.8(e)]{bm:extn-reedy}, the orbit category $\O_G$ of a compact Lie group $G$ is naturally stratified, with strata that are groupoids.
  But in general, the height of this stratification must be at least $\omega\cdot \dim(G)+|\pi_0(G)|$, with the degree of an orbit $G/H$ being $\omega\cdot \dim(H)+|\pi_0(H)|-1$.
  (We will return to this example in \cref{eg:orbit} and \cref{eg:enriched-orbit}.)
\end{rmk}

If $x\in\C$ has degree $\de$, we define the \textbf{matching object functor} $M_x : \M^\C \to \M$, for any complete category \M, to be the composite
\[ \M^\C \xto{\iota_\de^*} \M^{\C_\de} \xto{\wlim{\C(x,\blank)}{\blank}{\C_\de}} \M. \]
where $\iota_\de : \C_\de \into\C$ is the inclusion of the objects of degree $<\de$.
Note that $\wlim{\C(x,\blank)}{\blank}{\C_\de}$ is (the $x$-part of) the functor along which we glue to get from $\C_\de$ to $\C_{\de+1}$; thus the matching object just extends this to a functor on the larger diagram category $\M^\C$.

By a computation with ends and coends, we have $M_x A \cong \wlimc{\lan_{\iota_\de} \C(x,\blank)}{A}$, where the weight is defined by
\[ (\lan_{\iota_\de} \C(x,\blank))(y) = \wcolim{\C(\blank,y)}{\C(x,\blank)}{\C_\de}. \]
If $y\in\C_\de$, then the co-Yoneda lemma reduces this to $\C(x,y)$.
Otherwise, $\C(z,y)$ is empty for all $z\in\C_\de$, hence so is the colimit.
Thus, we have 
\[ M_x A \cong \wlimc{\partial_\de\C(x,\blank)}{A} \]
where
\begin{align*}
  \partial_\de\C(x,y) &=
  \begin{cases}
    \C(x,y) &\qquad \deg(y)<\de\\
    \emptyset &\qquad \text{otherwise}.
  \end{cases}
\end{align*}

We now have:

\begin{thm}\label{thm:inverse-model}
  For any inverse category \C and any model category \M, the category $\M^\C$ has a model structure in which
  \begin{itemize}
  \item The cofibrations and weak equivalences are objectwise, and
  \item The fibrations are the maps $A\to B$ such that $A_x \to M_x A \times_{M_x B} B_x$ is a fibration in \M for all $x\in \C$.
  \end{itemize}
\end{thm}
\begin{proof}
  By induction on the height of \C.

  For the successor steps, we use \cref{thm:gluing=collage2} and \cref{thm:model}.
  Since cofibrations, fibrations, and weak equivalences in $\M^I$ are all objectwise, to show that $\M^{\coll{W}}$ inherits a model structure from $\M^\C$ it will suffice to show that $\wlimc W \blank : \M^\C\to\M$ is right Quillen whenever \C is inverse.
  The definition of matching objects then implies that the resulting model structure agrees with that stated in the theorem.

  To do this we may equivalently show that its left adjoint $\M \to \M^\C$ is left Quillen.
  But this adjoint takes $X\in\M$ to the diagram $c\mapsto W(c) \cpw X$, the copower of $W(c)$ copies of $X$.
  Thus, since cofibrations and acyclic cofibrations in \M are preserved by coproducts, and cofibrations and acyclic cofibrations in $\M^\C$ are levelwise (by the inductive hypothesis), this functor is left Quillen.

  For the limit steps, we use \cref{thm:clovenlim}, along with the observation in \cref{thm:wfs} that each forgetful functor $\biglue{\al} \to\M$ is strict.
\end{proof}

We have so far only considered profunctors $W:\C\hto \D$ where $\D=I$ is a discrete set, but we could in principle allow \D to be an arbitrary small category.

\begin{defn}\label{defn:stratified}
  The collection of \textbf{stratified categories of height $\beta$} is defined by transfinite recursion over ordinals $\beta$ as follows.
  \begin{itemize}
  \item The only stratified category of height $0$ is the empty category.
  \item The stratified categories of height $\be+1$ are those of the form $\coll{W}$, where $W:\C\hto \D$ and $\C$ is a stratified category of height $\be$.
  \item At limit ordinals, we take colimits as before.
  \end{itemize}
\end{defn}

As before, we say that an object of a stratified category has \textbf{degree $\de$} if it was added by the step $\C_\de \to \coll{W_\de}= \C_{\de+1}$.
We will write $\C_{=\de}$ for the full subcategory of objects of degree $\de$ (which is the category \D added at the $\de^\mathrm{th}$ step), and call it the $\de^{\mathrm{th}}$ \textbf{stratum}.

\begin{thm}
  A small category \C is stratified if and only if we can assign an ordinal degree to each of its objects such that every morphism non-strictly decreases degree.
\end{thm}
\begin{proof}
  Just like \cref{thm:inverse-char}, with the set $I$ replaced by the category $\C_{=\de}$.
\end{proof}

As with inverse categories, when we say that \C is stratified, we mean that we have fixed a degree function.
(Otherwise, there is no content in being stratified, since every category can be stratified in \emph{some} way: just stick everything in degree zero!)
The matching object functors are then defined just as for inverse categories.

Stratified categories are not quite as useful for defining model structures are, because when we glue along the functor $\wlimc{W}{\blank}:\M^\C \to\M^\D$, we require $\M^\D$ to already have a model structure.
However, if we are willing to use a general-purpose model structure on these categories, such as a projective or injective one, we do have an analogue of \cref{thm:inverse-model}.

\begin{thm}\label{thm:stratified-model}
  Let \C be a stratified category, and \M a model category such that $\M^{\C_{=\de}}$ has a projective model structure for each stratum $\C_{=\de}$ of \C.
  Then $\M^\C$ has a model structure in which:
  \begin{itemize}
  \item The weak equivalences are objectwise,
  \item The cofibrations are the maps that are projective-cofibrations when restricted to each stratum, and
  \item The fibrations are the maps $A\to B$ such that $A_x \to M_x A \times_{M_x B} B_x$ is a fibration in \M for all $x\in \C$.
  \end{itemize}
\end{thm}
\begin{proof}
  Just like \cref{thm:inverse-model}.
\end{proof}

\section{Functoriality of bigluing}
\label{sec:functoriality}

In order to extend the theory of \cref{sec:inverse-categories} from gluing to bigluing, and hence from inverse categories to Reedy categories, we need to discuss the functoriality of bigluing a bit first.
Thus, in this section we return to the notation of \cref{sec:bigluing}.
Let \cCell denote the following 2-category:
\begin{itemize}
\item Its objects consist of two categories \M and \N, two functors $F,G:\M\to\N$, and a natural transformation $\al:F\to G$.
  We abbreviate these data as $\kM=(\M,\N,F,G,\al)$.
\item A morphism from $\kM$ to $\kM'$ consists of functors $S:\M\to\M'$ and $T:\N\to\N'$ and natural transformations $\xi:F'S\to TF$ and $\ze:TG \to G'S$ such that the following square commutes:
  \begin{equation}\label{eq:cell-mor}
    \vcenter{\xymatrix{
        F' S \ar[r]^{\xi}\ar[d]_{\al' S} &
        T F \ar[d]^{T\al}\\
        G' S \ar@{<-}[r]_{\ze} &
        T G
      }}
  \end{equation}
\item A 2-cell from $(S,T,\xi,\ze)$ to $(S',T',\xi',\ze')$ consists of natural transformations $S\to S'$ and $T\to T'$ such that the following diagrams commute:
  \begin{equation}
    \vcenter{\xymatrix{
        F' S \ar[r]\ar[d] &
        T F \ar[d]\\
        F' S'\ar[r] &
        T' F
      }}
    \qquad
    \vcenter{\xymatrix{
        T G\ar[r]\ar[d] &
        G' S\ar[d]\\
        T' G\ar[r] &
        G' S'
      }}
  \end{equation}
\end{itemize}

\begin{lem}\label{thm:cell-mor}
  $\biglue{\blank}$ is a 2-functor $\cCell \to \cCat$.
\end{lem}
\begin{proof}
  The action on objects was defined in \cref{sec:bigluing}.
  The image of a morphism $(S,T,\xi,\ze):\kM \to \kM'$ is the functor $\biglue{\al} \to\biglue{\al'}$ sending $(M,N,\phi,\gm)$ to $SM$ and $TN$ equipped with the following factorization:
  \[ F' S M \xto{\xi_M} T F M \xto{T \phi} T N \xto{T \gm} T G M \xto{\ze_M} G' S M. \]
  We leave the rest to the reader.
\end{proof}

In particular, $\biglue{\blank}$ takes adjunctions in \cCell to adjunctions in \cCat.
We now show that this functoriality extends to wfs, hence also model structures and Quillen adjunctions.

\begin{thm}\label{thm:wfs-mor-pres}
  Suppose $(S,T,\xi,\ze):\kM \to \kM'$ in \cCell, where $\kM$ and $\kM'$ satisfy the hypotheses of \cref{thm:wfs} or \cref{thm:prewfs}.
  If $S$ and $T$ are cocontinuous and preserve \L-maps, and \xi is an isomorphism, then the induced functor $\biglue{S,T}:\biglue{\al} \to\biglue{\al'}$ also preserves \L-maps.
\end{thm}
\begin{proof}
  There are two conditions for a morphism in a bigluing category to be an \L-map.
  The first is obviously preserved since $S$ preserves \L-maps.
  For the second, the assumption on \xi implies that in the following diagram:
  \begin{equation}
    \vcenter{\xymatrix{
        F'SM\ar[r]^{\cong}\ar[d] &
        TFM\ar[r]\ar[d] &
        TN\ar[d]\\
        F'SM'\ar[r]_{\cong} &
        TFM'\ar[r] &
        \bullet
      }}
  \end{equation}
  the left-hand square is a pushout.
  Thus, the right-hand square is a pushout if and only if the outer rectangle is.
  The claim now follows from the fact that $T$ preserves pushouts and \L-maps.
\end{proof}


Dually, if $S$ and $T$ are continuous and preserve \R-maps and \ze is an isomorphism, then the induced functor on bigluing categories preserves \R-maps.




We will additionally need a sort of ``parametrized'' functoriality of $\biglue{\blank}$, expressed by the following lemma.

\begin{lem}\label{thm:cell-paramor}
  Suppose bigluing data $\kM$ and $\kM'$ and a category \V, together with functors $S:\V\times \M\to\M'$ and $T:\V\times \N\to\N'$ and natural transformations $\xi:F'S\to T(1\times F)$ and $\ze:T(1\times G) \to G'S$ such that a square like~\eqref{eq:cell-mor} commutes.
  Then there is an induced functor $\biglue{S,T}:\V\times \biglue{\al} \to\biglue{\al'}$.
\end{lem}
\begin{proof}
  We send $V\in\V$ and $(M,N,\phi,\gm)\in\biglue{\al}$ to $S(V,M)$ and $T(V,N)$ equipped with the factorization
  \[ F'S(V,M) \xto{\xi} T(V,FM) \xto{\phi} T(V,N) \xto{\gm} T(V,GM) \xto{\ze} G'S(V,M).\qedhere\]
\end{proof}


The extension of parametrized functoriality to wfs is a bit more subtle (but still an essentially standard argument, see e.g.~\cite[18.4.9]{hirschhorn:modelcats})

\begin{thm}\label{thm:wfs-paramor-pres}
  Suppose given data as in \cref{thm:cell-paramor}, where both \kM and $\kM'$ satisfy the hypotheses of \cref{thm:prewfs}.
  If $S$ and $T$ are left wfs-bimorphisms that are cocontinuous in each variable, and $\xi$ is an isomorphism, then the induced functor $\biglue{S,T}:\V\times \biglue{\al} \to\biglue{\al'}$ is also a left wfs-bimorphism.
\end{thm}
\begin{proof}
  Let $V\to V'$ be an \L-map in \V, and let $(\mu,\nu):(M,N,\phi,\gm) \to (M',N',\phi',\gm')$ be an \L-map in $\biglue{\al}$.
  The first condition we need is that the induced map $S(V,M')\sqcup_{S(V,M)} S(V',M) \to S(V',M')$ is in \L, which is just the assumption that $S$ is a left wfs-bimorphism.
  The second condition is that the induced map
  \[ \Big(T(V,N') \sqcup_{T(V,N)} T(V',N)\Big) \sqcup_{F'\big(S(V,M')\sqcup_{S(V,M)} S(V',M)\big)} F'S(V',M') \to T(V',N') \]
  is in \L.
  Since $F'$ preserves pushouts, we can rearrange the domain of this map as the colimit of the diagram
  \[\xymatrix{ &&F'S(V',M')\\
    & F'S(V,M') \ar[ur] \ar[dl] & F'S(V,M) \ar[l] \ar[r] \ar[d] & F'S(V',M) \ar[ul] \ar[dr] \\
    T(V,N') && T(V,N) \ar[ll] \ar[rr] && T(V',N) 
  }\]
  which in turn can be rewritten as the pushout
  \[ \Big(F'S(V',M') \sqcup_{F'S(V',M)} T(V',N) \Big) \sqcup_{\big(F'S(V,M') \sqcup_{F'S(V,M)} T(V,N)\big)} T(V,N'). \]
  As in \cref{thm:wfs-mor-pres}, the assumption on $\xi$ further identifies this with
  \[ \Big(T(V',FM') \sqcup_{T(V',FM)} T(V',N) \Big) \sqcup_{\big(T(V,FM') \sqcup_{T(V,FM)} T(V,N)\big)} T(V,N'), \]
  and the fact that $T$ preserves pushouts in its second variable identifies it with
  \[ T\Big(V',FM' \sqcup_{FM} N \Big) \sqcup_{T\big(V,FM' \sqcup_{FM} N\big)} T(V,N'). \]
  Modulo these identifications, the map in question is induced from the commutative square
  \begin{equation}
    \vcenter{\xymatrix{
        {T\big(V,FM' \sqcup_{FM} N\big)}\ar[r]\ar[d] &
        T(V,N')\ar[d]\\
        T\Big(V',FM' \sqcup_{FM} N \Big)\ar[r] &
        T(V',N').
      }}
  \end{equation}
  But now the result follows from the assumption that $T$ is a left wfs-bimorphism, since $V\to V'$ and $FM' \sqcup_{FM} N \to N'$ are \L-maps (the latter by definition of \L-maps in $\biglue{\al}$).
\end{proof}

There is an easy dualization for right wfs-bimorphisms $\V\op\times \biglue{\al} \to\biglue{\al'}$.



\section{Iterated bigluing}
\label{sec:iterated-bigluing}

We now intend to parallel the approach of \cref{sec:inverse-categories}, but using bigluing instead of gluing.
The simplest successor step of this sort will be bigluing along a transformation $\al:F\to G$, where $F,G:\M^\C\to\M$ are functors and \C has already been constructed.
Given the requirements on $F$ and $G$, it is natural to take $G$ to be a weighted limit $\wlimc W \blank$ as in \cref{sec:inverse-categories} and $F$ to be a \textbf{weighted colimit} $(\wcolimc U \blank)$ for some $U:\C\op\to\nSet$.
Recall that this is defined by
\[ \wcolimc U X = \ncoeq\left( \coprod_{c,c'} \Big(U({c'}) \times \C(c,c')\Big) \cpw X({c}) \toto \coprod_c U(c) \cpw X(c) \right). \]
As a functor, $(\wcolimc U \blank)$ has a right adjoint $\M\to\M^\C$ whose value on $Y$ is defined by $(Y^U)(c) = Y^{U(c)}$.
In particular, $(\wcolimc U \blank)$ is cocontinuous.

The remaining datum for bigluing is a transformation $\al:(\wcolimc U \blank) \to \wlimc W \blank$.
Since we know the left adjoint of $\wlimc{W}{\blank}$, such a transformation is uniquely determined by a map $H\to \Id_{\M^\C}$, where $H$ is the functor defined by $H(X)(c) = (\wcolimc U X)\cpw W(c)$.
And since colimits commute with colimits, if we interpret $U$ and $W$ as profunctors $\Uchk:\tc\hto \C$ and $\What:\C\hto\tc$ respectively, then we have $H(X)(c) = \wcolimc{(\Uchk\What)(\blank,c)}{X}$.

Now the co-Yoneda lemma tells us that $X(c) \cong \wcolimc{\C(\blank,c)}{X}$.
Thus, we may specify a transformation $\al$ by giving a map of profunctors (i.e.\ a 2-cell in \cProf) $\Uchk\What \to \C(\blank,\blank)$.
Explicitly, this reduces to the following.

\begin{defn}
  Given a category \C, we define \textbf{abstract bigluing data from \C to \tc} to consist of functors $U:\C\op\to\nSet$ and $W:\C\to\nSet$ together with maps
  \[ \al_{c,c'}: W({c'}) \times U(c) \to \C(c,c') \]
  natural in $c$ and $c'$.
\end{defn}

More generally, we may allow profunctors with arbitrary target.

\begin{defn}\label{defn:abd-cd}
  Given categories \C and \D, \textbf{abstract bigluing data from \C to \D} consists of profunctors $U:\D\hto \C$ and $W:\C\hto\D$ together with a map
  \[ \al: \wcolim{W}{U}{\D} \to \C(\blank,\blank) \]
  of profunctors $\C\hto \C$.
\end{defn}

\begin{rmk}
  The category of abstract bigluing data from \C to \tc is known in the literature as the \textbf{Isbell envelope}~\cite{isbell:soc} of \C.
  Abstract bigluing data from \C to \D is then equivalent to a functor from \D into this Isbell envelope.
\end{rmk}

For any complete and cocomplete \M, abstract bigluing data from \C to \D induces a pair of functors $(\wcolimc{U}{\blank}):\M^\C\to\M^\D$ and $\wlimc{W}{\blank}:\M^\C\to\M^\D$ and a natural transformation $\albar:(\wcolimc{U}{\blank}) \to \wlimc{W}{\blank}$, which we can therefore biglue along.
We now hope that $\biglue{\albar}$ will be equivalent to a functor category $\M^{\coll{\al}}$, where $\coll{\al}$ is a sort of ``collage'' of the abstract bigluing data.
To obtain this, we recall the following vast generalization of \cref{thm:collage}.

\begin{defn}\label{defn:gen-replaxcocone}
  Let \J be a small category and $T:\J\to\cProf$ a lax functor, with constraints $1_{T_i}\to T(1_i)$ and $(Tt)(Ts) \to T(ts)$.
  A \textbf{representable lax cocone} under $T$ with vertex \E consists of
  \begin{itemize}
  \item For each object $i\in \J$, a functor $P_i : T i \to \E$.
  \item A colax\footnote{This is called a \emph{lax} cocone despite the presence of a \emph{colax} transformation, because the latter can equivalently be considered a \emph{lax} transformation from the functor $\J\op \to \cCat$ constant at $\tc$ to the functor $\cProf(T(\blank),\E)$.  By replacing this constant functor with a nonconstant one, we obtain the notion of lax \emph{weighted} cocone, which cannot be expressed as a colax transformation.} natural transformation with components $\rep {(P_i)}$ from $T$ to the constant functor at \E.
    Thus, for each morphism $s:i\to j$ in \J we have a natural transformation $\rep {(P_j)} (T s) \to \rep {(P_i)}$, satisfying straightforward axioms.
  \end{itemize}
  A morphism of such cocones consists of transformations $P_i\to P'_i$ such that evident squares commute.
\end{defn}

\begin{defn}\label{defn:gen-collage}
  Given a small category \J and a lax functor $T:\J\to\cProf$, define its \textbf{collage} $\coll{T}$ to be the category described as follows.
  \begin{itemize}
  \item Its set of objects is $\coprod_{i\in \J} \mathrm{ob}(Ti)$.
  \item For $i,j\in \J$, $x\in Ti$, and $y\in Tj$, its hom-set is $\coll{T}(x,y) = \coprod_{s:j\to i} Ts(x,y)$.
  \item The category structure is induced by the lax functoriality constraints of $T$.
  \end{itemize}
\end{defn}

In these definitions and the theorem to follow, one can in fact allow \J to be a small \emph{bicategory}, in which case the coproducts in $\coll{T}(x,y)$ must be replaced by a colimit over $\J(j,i)$; but we will not need this.

\begin{thm}\label{thm:gen-collage}
  Any lax functor $T:\J\to\cProf$ admits a universal representable lax cocone with vertex $\coll T$, so that the category of representable lax cocones under $T$ with any vertex \E is naturally isomorphic to the category of functors $\coll T \to \E$.
\end{thm}
\begin{proof}
  This almost follows from the proof of~\cite[Proposition 2.2(a), (b), and (e)]{street:cauchy-enr}; the only difference is the latter's use of retracts of representable profunctors instead of representable ones (since the former can be characterized as the left adjoints in \cProf).
  However, as remarked in~\cite[Theorem 16.16]{gs:freecocomp}, 
  essentially the same proof yields our version involving actual representables.
  We could also simply observe that the coprojections $P_i : Ti \to \coll T$ are evidently representable, and detect representability since they are jointly surjective on objects.
\end{proof}

As in \cref{sec:inverse-categories}, collages are also lax colimits in \cProf itself.
A different construction of collages with the universal property of \cref{thm:gen-collage} (but not the one for non-representable cocones) can be found in~\cite{gp:double-limits}.

Given \cref{thm:gen-collage}, it is natural to try to represent abstract bigluing data as a lax functor into \cProf.
For this purpose let \J be the ``free-living retraction'': it has two objects $c$ and $d$ and three nonidentity morphisms $w:c\to d$, $u:d\to c$, and $e:d\to d$, with $uw=1_c$, $wu = e$, $ew = s$, $ue = e$, and $ee=e$.

\begin{lem}
  Abstract bigluing data from \C to \D is equivalently a lax functor $T:\J\to\cProf$, for the above \J, such that
  \begin{itemize}
  \item $Tc = \C$ and $Td=\D$, and
  \item The constraints $1_{\C} \to T(1_c)$, $1_{\D} \to T(1_d)$, and $(Tw)(Tu) \to Te$ are identities.
  \end{itemize}
\end{lem}
\begin{proof}
  Of the remaining data, of course $Tu$ and $Tw$ correspond to $U$ and $W$ respectively, while the constraint $(Tu)(Tw) \to T(uw) = T(1_c)$ corresponds to $\al$.
  We leave it to the reader to check that the rest of the constraints are uniquely determined by these.
\end{proof}

In this case we can unwind \cref{defn:gen-collage} to something much simpler, since most hom-categories of \J are singleton sets, except for $\J(d,d)$ which is a doubleton.
Thus, the \textbf{collage of abstract bigluing data} $(U,W,\al)$ is the category $\coll{\al}$ whose objects are the disjoint union of the objects of \C and \D, with
\begin{align*}
  \coll{\al}(c,c') &= \C(c,c')\\
  \coll{\al}(c,d) &= U(c,d)\\
  \coll{\al}(d,c) &= W(d,c)\\
  \coll{\al}(d,d') &= \D(d,d') \sqcup (\wcolimc{U}{W})(d,d').
\end{align*}
(The coproduct in $\coll{\al}(d,d')$ arises from the doubleton $\J(d,d) = \{1_d,e\}$, with $(\wcolimc{U}{W})$ being $Te = (Tw)(Tu)$.)

\begin{thm}\label{thm:bigluing=gencollage}
  Given abstract bigluing data $(U,W,\al):\C \to \D$ and any complete and cocomplete category \M, we have $\M^{\coll{\al}} \simeq \biglue{\albar}$, where $\albar:(\wcolimc U \blank) \to \wlimc W\blank$ is the induced (concrete) bigluing data between functors $\M^\C\to\M^\D$.
\end{thm}
\begin{proof}
  Given \cref{thm:gen-collage}, it remains to show that $\biglue{\albar}$ is equivalent to the category of representable lax cocones with vertex \M under the corresponding lax functor $T:\J\to\cProf$.
  Such a cocone consists of two functors $P_c :\C\to \M$ and $P_d:\D\to \M$, which of course are exactly the objects of $\M^\C$ and $\M^\D$ that appear in an object of $\biglue{\al}$, along with structure maps that we hope will correspond to a factorization of $\albar: (\wcolimc{U}{P_c}) \to \wlimc{W}{P_c}$ through $P_d$.

  Firstly, since our lax functor $T$ is normal (preserves identities strictly), the structure maps of our lax cocone corresponding to identity morphisms are uniquely determined.
  Moreover, since $(Tw)(Tu) \to T(wu) = Te$ is an identity, the axioms of a lax cocone also ensure that the structure map $\rep {(P_d)} WU \to \rep {(P_d)}$ must be the composite of the two structure maps $\rep {(P_d)} W \to \rep {(P_c)}$ and $\rep {(P_c)} U \to \rep {(P_d)}$; thus only the latter two remain as data.
  These must satisfy the axiom that the composite
  \begin{equation}
    \rep {(P_c)} U W \to \rep {(P_d)} W \to \rep {(P_c)}\label{eq:gen-repcocone-fact}
  \end{equation}
  is equal to $\rep {(P_c)} U W \to \rep {(P_c)} 1_\C \cong \rep {(P_c)}$.
  Expressed in terms of tensor products, this means that
  \[  \wcolim{W}{\wcolimc{U}{\rep {(P_c)}}}{\D} \to \wcolim{W}{\rep{(P_d)}}{\D} \to \rep{(P_c)} \]
  is equal to
  \[ \wcolimc{\wcolim{W}{U}{\D}}{\rep {(P_c)}} \to \wcolimc{\C(\blank,\blank)}{\rep {(P_c)}}\cong{\rep {(P_c)}} .\]
  As in the proof of \cref{thm:gluing=collage}, these tensor products do not commute with $\rep{(\blank)}$, but we can pass across adjunctions to get to cotensor products which do.
  Thus we end up comparing two maps from $\rep{(P_c)}$ to
  \[ \wlim{U}{\wlimc{W}{\rep{(P_c)}}}{\D} \cong \rep{(\wlim{U}{\wlimc{W}{P_c}}{\D})},\]
  hence equivalently two maps from $P_c$ to $\wlim{U}{\wlimc{W}{P_c}}{\D}$, and thus also equivalently two maps from $\wcolimc{U}{P_c}$ to $\wlimc{W}{P_c}$.
  It is straightforward to check that the resulting condition is exactly that appearing in $\biglue{\albar}$, and that
  the other axioms of a lax cocone are automatic.
\end{proof}

In conclusion, given abstract bigluing data $(U,W,\al)$ from \C to \D, if we biglue along the resulting natural transformation between functors $\M^\C \to \M^\D$, the result is always again a functor category, and we have a formula for the diagram shape (the collage $\coll{\al}$).
We are now interested in the categories that can be obtained by transfinitely iterating this construction.

\begin{defn}\label{defn:bistratified}
  The collection of \textbf{bistratified categories of height $\beta$} is defined by transfinite recursion over ordinals $\beta$ as follows.
  \begin{itemize}
  \item The only bistratified category of height $0$ is the empty category.
  \item The bistratified categories of height $\be+1$ are those of the form $\coll{\al}$, where $(U,W,\al):\C\hto \D$ is abstract bigluing data and $\C$ is a stratified category of height $\be$.
  \item At limit ordinals, we take colimits as usual.
  \end{itemize}
\end{defn}

As with inverse categories and stratified categories, we say that an object of a stratified category has \textbf{degree $\de$} if it was added by the step $\C_\de \to \coll{\al_\de}= \C_{\de+1}$.
We write $\C_{=\de}$ for the subcategory $\D_\de$ and call it the $\de^{\mathrm{th}}$ \textbf{stratum}; its objects are those of degree $\de$, but unlike before it is \emph{not} a full subcategory of $\C$.

Since we have morphisms that both increase and decrease degree, it is also not as clear how to identify whether a degree function induces a bistratification.
For this, we begin with the following.

\begin{thm}[Recognition principle for collages]\label{thm:recoll}
  Let $\C\subseteq \E$ be a full subcategory, and let $\D\subseteq\E$ be a not-necessarily-full subcategory whose objects are precisely those not in \C.
  Define $U:\D\hto \C$ and $W:\C\hto \D$ by $U(c,d) = \E(c,d)$ and $W(d,c) = \E(d,c)$, and let $\al:\wcolim{W}{U}{\D} \to \C(-,-)$ be composition in \E, yielding abstract bigluing data from \C to \D.
  Then the induced functor $\coll{\al} \to \E$ is an isomorphism if and only if the following condition holds:
  \begin{equation}\label{eq:recoll}
    \parbox{10cm}{A morphism in \E between objects of \D factors through an object of \C if and only if it does \emph{not} lie in \D, and in this case the category of such factorizations is connected.}\tag{$*$}
  \end{equation}
\end{thm}

Here by the \textbf{category of factorizations} of a morphism $f:x\to y$, we mean the category whose objects are pairs of morphisms $(h,g)$ such that $hg = f$, and whose morphisms $(h,g) \to (h',g')$ are morphisms $k$ (which we call \textbf{connecting morphisms}) such that $g' = gk$ and $h = h'k$:
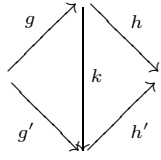
\[ \xymatrix{ & \ar[dr]^{h} \ar[dd]^{k} \\ \ar[ur]^g \ar[dr]_{g'} && \\ & \ar[ur]_{h'} } \]
The statement~\eqref{eq:recoll} is about the full subcategory of this on the objects $(h,g)$ such that $\mathrm{dom}(h) = \mathrm{cod}(g)$ lies in \C.

\begin{proof}
  By definition, the functor $\coll{\al} \to \E$ is bijective on objects, and fully faithful except possibly between objects of \D.
  For $d,d'\in\D$, the map $\coll{\al}(d,d') \to \E(d,d')$ has two components, one of which is the inclusion $\D(d,d') \into \E(d,d')$, and the other of which is the map
  \begin{equation}
    \wcolimc{\E(\blank,d')}{\E(d,\blank)} \to \E(d,d').\label{eq:recoll-map}
  \end{equation}
  Now by definition, $\wcolimc{\E(\blank,d')}{\E(d,\blank)}$ is the quotient of the set $\coprod_{c\in\C} \E(c,d')\times \E(d,c)$ by the equivalence relation generated by setting $(h,k g) \sim (h k,g)$ for any $g:d\to c$, $k:c\to c'$, and $h:c'\to d'$, and~\eqref{eq:recoll-map} is induced by composition.
  Thus, the first clause of condition~\eqref{eq:recoll} ensures that the two components of $\coll{\al}(d,d') \to \E(d,d')$ have disjoint images and are jointly surjective.
  Moreover, the preimage of $f:d\to d'$ under~\eqref{eq:recoll-map} is the quotient of the set of factorizations $(h,g)$ of $f$ by this same relation.
  This is exactly the relation generated by morphisms in the category of factorizations, so the second clause of~\eqref{eq:recoll} ensures that the second component of $\coll{\al}(d,d') \to \E(d,d')$ is injective (the first component being obviously so).
\end{proof}

Another way to state this is: consider the collection of morphisms between objects not in \C that do not factor through any object of \C.
If these morphisms contain the identities and are closed under composition, they form a non-full subcategory \D.
In this case, we can define $U$, $W$, and $\al$ as above, and the only remaining part of~\eqref{eq:recoll} to check is that if a morphism between objects of \D factors through \C then the category of such factorizations is connected.\footnote{According to our terminology, the empty category is not connected.}

This leads to a recognition principle for bistratified categories.
First we state some definitions.

\begin{defn}
  Let \C be a category whose objects are assigned ordinal degrees.
  \begin{itemize}
  \item A morphism is \textbf{level} if its domain and codomain have the same degree.
  \item The \textbf{degree} of a factorization $(h,g)$ of a morphism $f$ is the degree of the intermediate object $\mathrm{dom}(h) = \mathrm{cod}(g)$.
  \item A factorization of a morphism $f$ is \textbf{fundamental} if its degree is strictly less than the degrees of both the domain and codomain of $f$.
  \item A morphism is \textbf{basic} if it does not admit any fundamental factorization.
  \end{itemize}
\end{defn}

\begin{thm}\label{thm:bistrat-char}
  A category \C whose objects are assigned ordinal degrees is bistrat\-ified if and only if the following conditions hold.
  \begin{enumerate}
  \item All identities are basic.\label{item:bistrat1}
  \item Basic level morphisms are closed under composition.\label{item:bistrat2}
  \item The category of fundamental factorizations of any non-basic level morphism is connected.\label{item:bistrat3}
  \end{enumerate}
\end{thm}
\begin{proof}
  This is analogous to \cref{thm:inverse-char}, using \cref{thm:recoll}.
  We induct on the supremum $\beta$ of the degrees of objects of \C.
  For any $\de\le\be$, let $\C_\de$ be the subcategory of objects of degree $<\de$ and $\E_\de$ the subcategory of objects of degree $\le\de$.

  If either $\be$ is a successor or $\de<\be$, then the supremum of degrees of objects of $\C_\de$ is strictly less than $\be$.
  Thus, by the inductive hypothesis, the given conditions for level morphisms of degree $<\de$ are equivalent to $\C_\de$ being bistratified.
  Now \cref{thm:recoll} implies that the given conditions for level morphisms of degree $\de$ are equivalent to $\E_\de$ being the collage of abstract bigluing data on $\C_\de$.

  Now if $\be$ is a successor, let $\de$ be its predecessor; then $\C \cong \E_\de$ and we are done.
  If $\be$ is a limit, then $\C$ is the colimit of $\C_\de$ for $\de<\be$, and we are done again.
\end{proof}

Note that condition~\ref{item:bistrat1} implies in particular that no objects of distinct degrees can be isomorphic.
In other words, all isomorphisms in a bistratified category are level.
Moreover, all isomorphisms are also basic, since composing one morphism in a fundamental factorization of an isomorphism $f$ with the inverse of $f$ would yield a fundamental factorization of an identity.

Unlike for inverse categories and stratified categories, a bistratified category does not automatically induce a model structure on diagram categories; we need additional conditions to ensure that the functors involved in each bigluing are well-behaved.
This is the role of ``Reedy-ness'' in all its forms, which we will introduce in the next section.

We end this section with a few useful observations about bistratified categories.
The first is that the notion is self-dual.

\begin{lem}\label{thm:bistrat-opp}
  If \C is bistratified, so is $\C\op$ with the same degree function.
\end{lem}
\begin{proof}
  We can induct on height, or just observe that the conditions of \cref{thm:bistrat-char} are self-dual.
\end{proof}

The second is that ``objectwise functors between diagram categories over bistratified domains can be constructed inductively''.
More precisely, suppose $S:\M\to\N$ is a functor; we define a functor $S^\C:\M^\C \to\N^\C$ inductively as follows:
\begin{itemize}
\item At successor stages, we apply \cref{thm:cell-mor}, where $\xi$ and $\ze$ are the canonical comparison maps $\wcolimc{U}{S X} \to S(\wcolimc{U}{X})$ and $S(\wlimc{W}{X})\to \wlimc{W}{S X}$.
\item At limit stages, we simply take limits.
\end{itemize}

\begin{lem}\label{thm:bistrat-postcomp}
  If \C is bistratified and $S:\M\to\N$, then $S^\C:\M^\C \to\N^\C$ is isomorphic to applying $S$ objectwise, a.k.a.\ postcomposition with $S$.
\end{lem}
\begin{proof}
  By definition, $S^\C$ acts as $S$ on each stratum.
  Tracing through the construction of its functorial action from $\xi$ and $\ze$, we see that it also acts as $S$ on non-basic level morphisms and morphisms between strata.
\end{proof}

We also have a parametrized version of this.
Given $S:\V\times\M\to\N$, we define $S^\C:\V\times\M^\C \to\N^\C$ in a similar inductive way using \cref{thm:cell-paramor}.

\begin{lem}\label{thm:bistrat-parapostcomp}
  If \C is bistratified and $S:\V\times \M\to\N$, then $S^\C:\V\times \M^\C \to\N^\C$ is naturally isomorphic to the functor which applies $S$ objectwise.\qed
\end{lem}

\section{Reedy categories}
\label{sec:reedy}

In this section, we consider bistratified categories in which each stratum is a discrete set.
The following characterization of such categories follows immediately from \cref{thm:bistrat-char}.

\begin{thm}\label{thm:bistrat-discrete}
  A category \C whose objects are assigned ordinal degrees is bistratified with discrete strata if and only if
  \begin{enumerate}
  \item The basic level morphisms are exactly the identities, and
  \item The category of fundamental factorizations of any non-basic level morphism is connected.\label{item:bistrat-discrete2}
  \end{enumerate}
\end{thm}

Our goal is to identify conditions ensuring that diagrams over such a category inherit an explicit model structure by iterating \cref{thm:model}.
It will be convenient to describe this model structure in terms of matching object functors, as in \cref{sec:inverse-categories}, and their dual \emph{latching} object functors.

If \C is bistratified and $x\in\C$ has degree $\de$, we define the \textbf{matching object functor} $M_x : \M^\C \to \M$ to be the composite
\[ \M^\C \xto{\iota_\de^*} \M^{\C_\de} \xto{\wlim{\C(x,\blank)}{\blank}{\C_\de}} \M, \]
where $\C_\de$ is the full subcategory of objects of degree $<\de$.
As in \cref{sec:inverse-categories}, $\wlim{\C(x,\blank)}{\blank}{\C_\de}$ is the $x$-component of the weighted limit functor $\M^\C \to \M^I$ used in the $\de^{\mathrm{th}}$ step of bigluing; the matching object extends this to a functor on $\M^\C$.
This is again isomorphic to a weighted limit over all of \C:
\[ M_x A \cong \wlimc{\partial_\de\C(x,\blank)}{A},\]
with weight defined by
\[ \partial_\de\C(x,y) = \wcolim{\C(\blank,y)}{\C(x,\blank)}{\C_\de}. \]
As in \cref{sec:inverse-categories}, if $y\in\C_\de$, the co-Yoneda lemma gives $\partial_\de\C(x,y)\cong\C(x,y)$.
Now, however, in contrast to the previous situation, if $\deg(y)\ge\de$ there can still be nontrivial morphisms into it from objects of $\C_\de$.

Concretely, the elements of $\partial_\de\C(x,y)$ are equivalence classes of composable pairs $x\to z\to y$ where $\deg(z)<\de$, with the equivalence relation being generated by morphisms $z\to z'$ making the two triangles commute.
In particular, we have a natural map $\partial_\de\C(x,y) \to \C(x,y)$ given by composition, and the fiber of this map over $f:x\to y$ is the set of connected components of the category of factorizations of $f$ through objects of degree $<\de$.

Dually, we define the \textbf{latching object functor} $L_x : \M^\C \to \M$ to be the composite
\[ \M^\C \xto{\iota_\de^*} \M^{\C_\de} \xto{\wcolim{\C(\blank,x)}{\blank}{\C_\de}} \M, \]
which is an extension of the $x$-component of the weighted colimit functor used in the $\de^{\mathrm{th}}$ step of bigluing.
This is isomorphic to a weighted colimit over all of \C:
\[ L_x A \cong \wcolimc{\partial_\de\C(\blank,x)}{A},\]
where $\partial_\de\C(\blank,\blank)$ is the \emph{same} bifunctor as for the matching object, only now regarded as a weight for colimits with the second variable fixed.

We now observe that wfs are automatically inherited by diagram categories over any bistratified \C.

\begin{thm}\label{thm:bistrat-disc-wfs}
  If \C is bistratified with discrete strata, then for any complete and cocomplete category \M with a wfs (or pre-wfs), the category $\M^\C$ inherits a \textbf{Reedy wfs} (or pre-wfs) in which
  \begin{itemize}
  \item $A\to B$ is in \R iff the induced map $A_x \to M_x A \times_{M_x B} B_x$ in \M is in \R for all $x\in \C$.
  \item $A\to B$ is in \L iff the induced map $L_x B \sqcup_{L_x A} A_x \to B_x$ in \M is in \L for all $x\in \C$.
  \end{itemize}
\end{thm}
\begin{proof}
  By induction.
  Successor stages follow from \cref{thm:wfs} (or \cref{thm:prewfs}), using the fact that \L-maps and \R-maps in $\M^I$ are objectwise.
  Limit stages follow from \cref{thm:clovenlim} and strictness of the forgetful functor in \cref{thm:wfs}.
\end{proof}

The following observation will also be useful.

\begin{lem}\label{thm:reedy-wfs-opp}
  If \C is bistratified with discrete strata, then the Reedy (pre-)wfs on $\M^{\C\op}$ is the opposite of the Reedy (pre-)wfs on $(\M\op)^\C$.
\end{lem}
\begin{proof}
  We induct on height.
  Limit stages are easy as usual.
  For successor stages, if $W:\C\to\nSet$ is a weight, then modulo the isomorphism $(\M\op)^\C \cong (\M^{\C\op})\op$, the $W$-weighted limit functor of $\M\op$
  \[ \wlimc W \blank : (M^{\C\op})\op \to \M\op \]
  can be identified with the opposite of the $W$-weighted \emph{colimit} functor of \M.
  Thus, if $(U,W,\al)$ is abstract bigluing data from \C to $I$, then $(W,U,\al)$ is also such from $\C\op$ to $I$.
  The matching and latching objects get interchanged, and so the wfs of \cref{thm:bistrat-disc-wfs} gets dualized.
\end{proof}

In particular, \cref{thm:bistrat-disc-wfs} implies that if \C is bistratified with discrete strata and \M is a model category, then $\M^\C$ inherits two wfs.
The question is now under what conditions they fit together into a model structure.

By \cref{thm:model}, this will happen if each functor $({\wcolim{\C(\blank,x)}{\blank}{\C_\de}})$ takes acyclic cofibrations to couniversal weak equivalences and each functor $\wlim{\C(x,\blank)}{\blank}{\C_\de}$ takes acyclic fibrations to universal weak equivalences.
The following lemma yields a sufficient condition for this.
Recall from \cref{thm:copower} that the wfs (injections, surjections) on $\nSet$ is ``universally enriching''; our present observation is that this enrichment lifts to weighted limits and colimits, when all diagram categories are given their Reedy (pre-)wfs.

\begin{lem}\label{thm:reedy-wfs}
  For any bistratified category \C with discrete strata and any complete and cocomplete category \M with a wfs,
  \begin{enumerate}
  \item The weighted limit $(\nSet^{\C})\op \times \M^\C \to \M$ is a right wfs-bimorphism.\label{item:reedywfs1}
  \item The weighted colimit $\nSet^{\C\op} \times \M^\C \to \M$ is a left wfs-bimorphism.\label{item:reedywfs2}
  \end{enumerate}
  (All diagram categories are equipped with their Reedy wfs.)
\end{lem}
\begin{proof}
  We prove~\ref{item:reedywfs1}; the proof of~\ref{item:reedywfs2} is dual.
  It will suffice to show that the left adjoint $\nSet^\C \times \M \to \M^\C$ is a left wfs-bimorphism.
  Note that this left adjoint is simply the objectwise copower, which by \cref{thm:bistrat-parapostcomp} can be constructed inductively using \cref{thm:cell-paramor}.
  Thus, we may use induction on the height of \C.

  The limit step follows from \cref{thm:clovenlimmor}.
  For the successor step, we use \cref{thm:wfs-paramor-pres}.
  The inductive hypothesis says that $S$ is a left wfs-bimorphism, while $T$ is such by \cref{thm:copower}, and $\xi$ is an isomorphism since the copower functor is cocontinuous in each variable.
\end{proof}

Recall from \cref{sec:model} that in the context of a wfs $(\L,\R)$, we say $X$ is an \emph{\L-object} if the map $\emptyset\to X$ is in \L.
Thus, \cref{thm:reedy-wfs} implies that if $U$ and $W$ are \L-objects in $\nSet^{\C\op}$ and $\nSet^\C$ respectively, and \M is a model category, then $(\wcolimc{U}{\blank})$ and $\wlimc{W}{\blank}$ are left and right Quillen respectively, so that \cref{thm:model} can be applied.
(Note that this condition treats $U$ and $W$ essentially symmetrically, corresponding to the fact that the notion of Reedy category is self-dual; this will no longer be the case for generalized Reedy categories in \cref{sec:auto}.)

In fact, this sufficient condition is also \emph{necessary}.
To show this, we begin with the following.

\begin{lem}\label{thm:quillen-almostreedy}
  If \C is bistratified with discrete strata and $U:\C\op\to\nSet$, the following are equivalent.
  \begin{enumerate}
  \item For any \M with a wfs, $(\wcolimc{U}{\blank}):\M^\C\to\M$ preserves \L-maps.\label{item:qar1}
  \item $U$ is an \L-object in $\nSet^{\C\op}$.\label{item:qar2}
  \end{enumerate}
  Similarly, $W$ is an \L-object in $\nSet^{\C}$ if and only if $\wlimc{W}{\blank}$ preserves \R-maps for any such \M.
\end{lem}
\begin{proof}
  \cref{thm:reedy-wfs} shows that~\ref{item:qar2} implies~\ref{item:qar1}.
  To show~\ref{item:qar1} implies~\ref{item:qar2}, we induct on the height of \C.

  As usual, let $\C_\de$ be the full subcategory of objects of degree $<\de$, with $\iota_\de:\C_\de \to \C$ the inclusion.
  Note that $\iota_\de^* : \M^\C \to \M^{\C_\de}$ preserves matching and latching objects at objects of degree $<\de$.
  In particular, $\iota_\de^* : \M^\C \to \M^{\C_\de}$ preserves \R-maps, so its left adjoint $\lan_{\iota_\de}: \M^{\C_\de} \to \M^\C$ preserves \L-maps.
  Therefore, if~\ref{item:qar1} holds for \C and $U$, the composite
  \[ \M^{\C_\de} \xto{\lan_{\iota_\de}} \M^\C \xto{\wcolimc{U}{\blank}} \M \]
  also preserves \L-maps.
  But this composite is isomorphic to $(\wcolim{\iota_\de^* U}{\blank}{\C_\de})$.

  Thus, if $\C_\de$ has height strictly less than that of $\C$, the inductive hypothesis implies that $\iota_\de^* U$ is an \L-object in $\nSet^{\C\op}$, and so $L_x U \to U_x$ is an \L-map for any $x$ of degree $<\de$.
  If the height of \C is a limit ordinal, then as $\de$ varies this includes all objects $x\in \C$, so we are done.
  Otherwise, $\C$ has height $\be+1$, and it remains only to show that $L_x U \to U_x$ is an \L-map (i.e.\ an injection) when $x$ has degree $\be$.

  Take $\M = \nSet$ with its usual (injections, surjections) wfs.
  Then we have the representable functor $\C(x,\blank)\in \nSet^\C$, and $\wcolimc U {\C(x,\blank)} \cong U_x$ by the co-Yoneda lemma.
  We also have $\partial_\be\C(x,\blank) \in \nSet^\C$, with $\wcolimc U {\partial_\be\C(x,\blank)} \cong L_x U$ by definition.
  Thus, it will suffice to show that $\partial_\be\C(x,\blank) \too \C(x,\blank)$ is an \L-map in $\nSet^\C$.
  In other words, we must show that the induced map
  \begin{equation}
    \partial_\be\C(x,y) \sqcup_{L_y \partial_\be\C(x,\blank)} L_y \C(x,\blank) \too \C(x,y)\label{eq:reedy-char-po}
  \end{equation}
  is an injection for any $y\in \C$.

  There are two cases: either $\deg(y)<\be$ or $\deg(y) = \be$.
  If $\deg(y)<\be$, then $\partial_\be\C(x,y) \cong \C(x,y)$, and likewise $L_y \partial_\be\C(x,\blank) \cong L_y \C(x,\blank)$.
  Thus, in this case~\eqref{eq:reedy-char-po} is an isomorphism.

  If $\deg(y) = \be$, then the map $\partial_\be\C(x,y) \to \C(x,y)$ is an injection by \cref{thm:bistrat-discrete} (in fact, it is an isomorphism unless $x=y$).
  We also have $L_y \C(x,\blank) \cong \partial_\be\C(x,y)$.
  Finally, we have
  \[ L_y \partial_\be\C(x,\blank) \cong \wcolim{\partial_\be\C(\blank,y)}{\partial_\be\C(x,\blank)}{\C_\be}. \]
  But since $\partial_\be\C(x,z) \cong \C(x,z)$ and $\partial_\be\C(z,y) \cong \C(z,y)$ for any $z\in\C_\be$, this is just $\partial_\be\C(x,y)$ again.
  So in this case,~\eqref{eq:reedy-char-po} is just the inclusion $\partial_\be\C(x,y) \into \C(x,y)$.

  The dual statement about $W$ follows formally by \cref{thm:reedy-wfs-opp}.
\end{proof}

\cref{thm:quillen-almostreedy} implies that if $(\wcolim{U}{\blank}{\C}):\M^\C\to\M$ takes acyclic cofibrations to acyclic cofibrations for any model category \M, then $U$ is an \L-object, and dually for $W$.
We would like to weaken the hypotheses to require only that $(\wcolim{U}{\blank}{\C})$ takes acyclic cofibrations to couniversal weak equivalences.
In general, not every couniversal weak equivalence is an acyclic cofibration, but this is true in a few model categories, such as the following.

Recall that the category \nCat has a canonical model structure, in which the weak equivalences are the equivalences of categories and the cofibrations are the functors that are injective on objects.
Let $\codisc:\nSet\to\nCat$ denote the functor which sends a set to the contractible groupoid on that set.

\begin{lem}\label{thm:acof-cat}
  Every couniversal weak equivalence in \nCat is an acyclic cofibration.
\end{lem}
\begin{proof}
  I am indebted to Karol Szumi\l{}o for pointing out that the proof of~\cite[A Somewhat Less Trivial Lemma]{sp:canonical} can be adapted to prove this lemma.
  
  Suppose $f:A\to B$ is a couniversal weak equivalence, and suppose for contradiction there were objects $x,y\in A$ with $x\neq y$ but $f(x)=f(y)$.
  Since $f$ is an equivalence, there is a unique isomorphism $e:x\cong y$ with $f(e) = \id$.
  Let $E=\codisc 2$ be the contractible groupoid with two objects $a$ and $b$; then there is some functor $g:A\to \codisc 2$ with $g(x) = a$ and $g(y)=b$.
  Let $C$ be a category containing an object $c$ with a nonidentity automorphism $p:c\cong c$; then there is a unique functor $h:E\to C$ taking the unique isomorphism $a\cong b$ to $p$.
  The composite $hg:A\to C$ then satisfies $hg(e) = p$.

  Let $D$ be the pushout of $f$ and $hg$.
  The induced functor $C\to D$ must map the nonidentity automorphism $p$ to an identity, since $f$ maps $e$ to an identity.
  But it must also be an equivalence of categories, since $f$ is a couniversal weak equivalence, giving a contradiction.
\end{proof}

Putting this all together, we have:

\begin{thm}\label{thm:almost-reedy}
  Suppose \C is bistratified with discrete strata, and that for any model category \M, the two wfs on $\M^\C$ form a model structure with the objectwise weak equivalences.
  Let $U:\C\op\to\nSet$ and $W:\C\to\nSet$, and let $\nSet^{\C\op}$ and $\nSet^{\C}$ have their Reedy wfs induced from the (injections, surjections) wfs on \nSet.
  \begin{enumerate}
  \item $({\wcolimc{U}{\blank}}):\M^\C\to\M$ takes Reedy acyclic cofibrations to couniversal weak equivalences for all model categories \M if and only if $U$ is an \L-object in $\nSet^{\C\op}$.\label{item:ar1}
  \item $\wlimc{W}{\blank}:\M^\C\to\M$ takes Reedy acyclic fibrations to universal weak equivalences for all model categories \M if and only if $W$ is an \L-object in $\nSet^\C$.\label{item:ar2}
  \end{enumerate}
\end{thm}
\begin{proof}
  We prove~\ref{item:ar1}; the proof of~\ref{item:ar2} is dual.
  \cref{thm:reedy-wfs} implies directly that if $U$ is an \L-object in $\nSet^{\C\op}$. then $({\wcolimc{U}{\blank}}):\M^\C\to\M$ preserves acyclic cofibrations, hence takes them to couniversal weak equivalences.
  Thus it remains to prove the converse.

  By the proof of \cref{thm:quillen-almostreedy}, it suffices to show that $({\wcolimc{U}{\blank}}):\nSet^\C\to\nSet$ preserves \L-maps.
  Let $i:A\to B$ be an \L-map in $\nSet^\C$, and let $j= i\sqcup \id:A\sqcup 1 \to B\sqcup 1$.
  Since weighted colimits preserve coproducts, we have $({\wcolimc{U}{j}})\cong (\wcolimc{U}{i}) \sqcup \id$.
  Thus, $\wcolimc{U}{i}$ is an \L-map (i.e.\ an injection) if and only if ${\wcolimc{U}{j}}$ is.
  
  Consider now the map $\codisc j$ in $\nCat^\C$.
  Since the set-of-objects functor $\nCat\to\nSet$ is cocontinuous and reflects cofibrations, $\codisc j$ is a Reedy cofibration.
  Since the domain and codomain of $j$ are both nonempty (this is why we passed from $i$ to $j$), $\codisc j$ is an objectwise weak equivalence.
  And since $\nCat^\C$ is, by assumption, a model category with the Reedy cofibrations and objectwise weak equivalences, $\codisc j$ is a Reedy acyclic cofibration, i.e.\ an \L-map for the Reedy wfs induced by the (acyclic cofibrations, fibrations) wfs on \nCat.

  By assumption on $U$, therefore, ${\wcolimc{U}{\nabla j}}$ is a couniversal weak equivalence in \nCat.
  But by \cref{thm:acof-cat}, it is therefore an acyclic cofibration, and in particular injective on objects.
  Since the set-of-objects functor $\nCat\to\nSet$ preserves colimits, ${\wcolimc{U}{j}}$ is also an injection, as desired.
\end{proof}

In conclusion, once we have \cref{thm:model} and the idea of iterating it, the following definition is really inevitable.

\begin{defn}\label{defn:almost-reedy}
  A bistratified category \C with discrete strata is \textbf{almost-Reedy} if for any object $x$ of degree $\be$, the hom-functors $\C(x,\blank) \in \nSet^{\C_\be}$ and $\C(\blank,x) \in \nSet^{\C_\be\op}$ are \L-objects.
\end{defn}

\begin{thm}\label{thm:reedy-model}
  Let \C be almost-Reedy.
  Then for any model category \M, the category $\M^\C$ inherits a model structure such that
  \begin{itemize}
  \item The weak equivalences are objectwise.
  \item $A\to B$ is a fibration (resp.\ acyclic fibration) iff for all $x\in \C$, the induced map $A_x \to M_x A \times_{M_x B} B_x$ is a fibration (resp.\ acyclic fibration) in \M.
  \item $A\to B$ is a cofibration (resp.\ acyclic cofibration) iff  for all $x\in \C$, the induced map $L_x B \sqcup_{L_x A} A_x \to B_x$ is a cofibration (resp.\ acyclic cofibration) in \M.
  \end{itemize}
\end{thm}
\begin{proof}
  By induction on height, with \cref{thm:almost-reedy} for the successor steps (limit steps are easy).
\end{proof}

Our next goals are to give a more concrete description of almost-Reedy categories, and to compare them to Kan's notion of Reedy category.

First of all, saying that $\C(x,\blank)$ is an \L-object means that for each $y\in\C$ of degree $\de$, the induced map
\[ \partial_\de\C(x,y) \too \C(x,y) \]
is an injection.
Since $y$ has degree $\de$ and $x$ has greater degree, $\partial_\de\C(x,y)$ is the set of connected components of the category of fundamental factorizations of morphisms $x\to y$.
In particular, the image of this map is the set of non-basic arrows $x\to y$; so injectivity of this map says that every non-basic arrow that strictly decreases degree has a connected category of fundamental factorizations.

Dually, $\C(\blank,x)$ being an \L-object asks that
\[ \partial_\de\C(y,x) \too \C(y,x) \]
is an injection whenever $\deg(x) >\deg(y)=\de$.
Thus, now we are asking the same condition except for morphisms which strictly \emph{increase} degree.
Recall that this same condition for \emph{level} morphisms appeared in \cref{thm:bistrat-discrete}\ref{item:bistrat-discrete2}.

We have just proven the following.

\begin{thm}\label{thm:reedychar}
  For a category to be almost-Reedy, it is necessary and sufficient that it be equipped with a degree function on its objects such that
  \begin{enumerate}
  \item The basic level morphisms are exactly the identities, and\label{item:reedychar1}
  \item The category of fundamental factorizations of any non-basic morphism (level or not) is connected.\qed\label{item:reedychar2}
  \end{enumerate}
\end{thm}

Perhaps surprisingly, almost-Reedy categories share most of the structure of ordinary Reedy categories.
Let $\rup\C$ be the set of basic morphisms that non-strictly raise degree, i.e.\ such that the degree of their codomain is at least the degree of their domain.
Dually, let $\rdn\C$ be the set of basic morphisms that non-strictly lower degree.
By \cref{thm:reedychar}\ref{item:reedychar1}, the only level morphisms in $\rup\C$ and $\rdn\C$ are identities.

\begin{lem}\label{thm:reedy-fact}
  Every morphism $f:c\to c'$ in an almost-Reedy category factors as $\rup f \rdn f$, where $\rup f \in \rup \C$ and $\rdn f \in\rdn \C$.
\end{lem}
\begin{proof}
  We induct on $\min(\deg(c),\deg(c'))$.
  If $f$ is basic, then it is either in $\rup\C$ or $\rdn\C$, in which case it has a trivial such factorization with one factor an identity.
  Otherwise, we have a fundamental factorization $f = h_0 g_0$.

  By the inductive hypothesis, we can write $g_0 = \rup {g_0} \rdn {g_0}$ and $h_0 = \rup {h_0} \rdn {h_0}$.
  If both $\rup{g_0}$ and $\rdn{h_0}$ are identities, then setting $\rdn f = \rdn{g_0}$ and $\rup f = \rup{h_0}$ yields the desired factorization.
  Otherwise, either the domain of $\rup{g_0}$ or the codomain of $\rdn{h_0}$ must have degree strictly less than that of $(h_0,g_0)$.
  If the former, set $g_1 = \rdn{g_0}$ and $h_1 = h\rup{g_0}$, while if the latter, set $g_1 = \rdn{h_0} g$ and $h_1 = \rup{h_0}$.
  In either case, we have another fundamental factorization $f = h_1 g_1$ of degree less than that of $(h_0,g_0)$.

  We now iterate, obtaining fundamental factorizations $f = h_n g_n$ with strictly decreasing degrees.
  Since ordinals are well-founded, the process must stop.
  Hence we must eventually have both $\rup{g_n}$ and $\rdn{h_n}$ identities, giving the desired factorization.
\end{proof}

\begin{lem}\label{thm:reedy-zigzag}
  In an almost-Reedy category, two fundamental factorizations of the same morphism with degrees $\delta$ and $\delta'$ are connected by a zigzag of fundamental factorizations passing only through intermediate factorizations of degree $<\max(\de,\de')$.
\end{lem}
\begin{proof}
  By \cref{thm:reedychar}\ref{item:reedychar2}, any two fundamental factorizations of the same morphism $f$ are connected by \emph{some} zigzag of fundamental factorizations.
  Thus, suppose we have such a zigzag, which we may assume contains no identity connecting morphisms, and let $\be$ be the maximum degree of all the intermediate factorizations occurring in it.
  We will show that if $\be\ge\de$ and $\be\ge\de'$, then we can replace this zigzag by another one in which there occurs one fewer factorization of degree $\be$ and none of degree $>\be$.
  By iterating this procedure, therefore, we will eventually reach a zigzag of the desired sort.
  (Technically, we are performing well-founded induction on $\omega\cdot \theta$, where $\theta$ is the height of \C.)

  To start with, choose an intermediate factorization $f=h g$ of degree $\be$ in the given zigzag.
  By assumption, the two factorizations $f=h' g'$ and $f = h'' g''$ on either side of it are each of degree $\le\be$.
  If either of them has degree $\be$, then the connecting morphism relates objects of the same degree and hence has a fundamental factorization.
  Thus, by inserting one or two additional objects in the zigzag (of degree $<\be$) we may assume that $(h',g')$ and $(h'',g'')$ each have degree $<\be$.

  Now there are four possibilities for the direction of the morphisms connecting these three factorizations in a zigzag.
  \begin{enumerate}
  \item $(h',g') \to (h,g) \to (h'',g'')$.
    In this case, we can simply omit $(h,g)$ from the zigzag entirely.
  \item $(h',g') \ot (h,g) \ot (h'',g'')$.
    We can again omit $(h,g)$.
  \item $(h',g') \xot{k} (h,g) \xto{\ell} (h'',g'')$.
    In this case we have $h = h' k$ and $h = h''\ell$, both of which are fundamental factorizations.
    Thus, they are connected by a zigzag of fundamental factorizations of $h$, which we can compose with $g$ and splice into our given zigzag in place of $(h,g)$.
  \item $(h',g') \xto{k} (h,g) \xot{\ell} (h'',g'')$.
    This is similar to the previous case, only now $g = kg'$ and $g = \ell g''$ are fundamental factorizations.
  \end{enumerate}
  Thus, in all four cases we can produce another zigzag with one fewer object of degree $\be$ and none of degree $>\be$, as claimed.
\end{proof}

\begin{thm}\label{thm:reedy-fact-uniq}
  The factorization in \cref{thm:reedy-fact} is unique.
\end{thm}
\begin{proof}
  Suppose $f = h g$ and $f = h' g'$ are two factorizations with $h,h'\in\rup \C$ and $g,g'\in\rdn \C$.
  Without loss of generality, let the degree of $(h,g)$ be greater than or equal to that of $(h',g')$.

  If $f$ is basic, then either $h=h'=\id$ or $g=g'=\id$, whence the factorizations agree.
  Otherwise, both of these factorizations must be fundamental, and hence they are connected by a zigzag of fundamental factorizations.
  By \cref{thm:reedy-zigzag}, therefore, they are connected by such a zigzag passing only through factorizations of degree strictly less than that of $(h,g)$.
  But any map of factorizations relating $(h,g)$ to one of lesser degree exhibits a fundamental factorization of either $h$ or $g$, which is impossible as they are both basic.
  Thus, the only possible such zigzag has length one or zero.
  If it has length one, then the connecting morphism is level and hence has a fundamental factorization, yielding again a fundamental factorization of either $h$ or $g$, a contradiction.
  Thus, it has length zero, so $(h,g) = (h',g')$.
\end{proof}

We can now show that almost-Reedy categories really are almost Reedy categories.
Recall the usual definition:

\begin{defn}
  A \textbf{Reedy category} is a category \C equipped with a ordinal degree function on its objects, and subcategories $\rup\C$ and $\rdn \C$ containing all the objects, such that
  \begin{itemize}
  \item Every nonidentity morphism in $\rup\C$ strictly raises degree and every nonidentity morphism in $\rdn\C$ strictly lowers degree.
  \item Every morphism $f$ factors uniquely as $\rup f \rdn f$, where $\rup f \in \rup \C$ and $\rdn f \in\rdn \C$.
  \end{itemize}
\end{defn}

\begin{thm}\label{thm:reedy}
  The following are equivalent for a category \C with an ordinal degree function on its objects.
  \begin{enumerate}
  \item \C is a Reedy category.\label{item:reedy1}
  \item \C is an almost-Reedy category and $\rup\C$ and $\rdn \C$ are closed under composition.\label{item:reedy2}
  \end{enumerate}
\end{thm}
\begin{proof}
  We have already shown that an almost-Reedy category has all the properties of a Reedy category except for closure of $\rup\C$ and $\rdn \C$ under composition, so~\ref{item:reedy2} implies~\ref{item:reedy1}.

  For the converse, let \C be Reedy.
  If a morphism $f\in\rdn\C$ has a factorization $f = h g$, then we can write
  \[ f = h g = \rup h \rdn h \rup g \rdn g = \rup h \rup k \rdn k \rdn g \]
  where $k = \rdn h \rup g$.
  Since $\rup \C$ and $\rdn \C$ are closed under composition, and Reedy factorizations are unique, we have $\rup h \rup k = \rup f$ and $\rdn k \rdn g = \rdn f$.

  In particular, since $\rup k$ non-strictly raises degree and $\rdn h$ non-strictly lowers it, the factorization $f = h g$ has degree greater than or equal to $f = \rup f \rdn f$.
  Thus, if $f = h g$ is a fundamental factorization, then so is $f = \rup f \rdn f$; hence neither $\rup f$ nor $\rdn f$ is an identity, and so by uniqueness of Reedy factorizations, $f\notin\rdn \C$ and $f\notin\rup\C$.
  On the other hand, if $f\notin\rdn \C$ and $f\notin\rup\C$, then the factorization $f = \rup f \rdn f$ must be fundamental since neither factor is an identity.
  Thus, $f$ has a fundamental factorization if and only if $f\notin\rdn \C$ and $f\notin\rup\C$, and thus (contrapositively) $\rup\C\cup\rdn\C$ is precisely the class of basic morphisms.
  In particular, the only basic level morphisms are identities, giving \cref{thm:reedychar}\ref{item:reedychar1}.
  Moreover, the subcategories $\rup\C$ and $\rdn \C$ given as \emph{data} in the definition of Reedy category in fact agree with those we have \emph{defined} for almost-Reedy categories.
  (The fact that $\rup\C$ and $\rdn\C$ in a Reedy category admit this characterization also appears in~\cite[Observation 3.18]{rv:reedy}.)

  We now observe, as in~\cite[Lemma 2.9]{rv:reedy}, that an arbitrary factorization $f = h g$ as above is connected by a two-step zigzag to the Reedy factorization $\rup f \rdn f$:
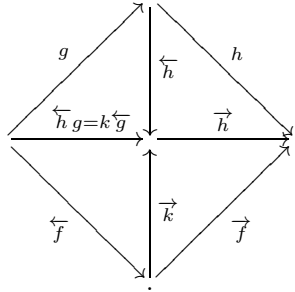
  \begin{equation}
    \xymatrix@+2pc{ & \ar[dr]^{h} \ar[d]^{\rdn h} \\
      \ar[ur]^{g} \ar[r]^(.6){\rdn h g = k \rdn g} \ar[dr]_{\rdn f} & \ar[r]^{\rup h} & \\
      & \ar[ur]_{\rup f} \ar[u]_{\rup k}. }\label{eq:reedyfact-zigzag}
  \end{equation}
  In particular, if $f = h g$ is a fundamental factorization, then it is connected to the canonical fundamental factorization $\rup f \rdn f$ by a zigzag of fundamental factorizations.
  Thus, the category of fundamental factorizations of any non-basic morphism is connected, which is \cref{thm:reedychar}\ref{item:reedychar2}.
  So~\ref{item:reedy2} implies~\ref{item:reedy1}.
\end{proof}

\begin{rmk}
  As mentioned in the introduction, \cref{thm:reedy} (or more precisely, its generalization \cref{thm:creedy}) has been formally verified in the computer proof assistant Coq; see \cref{sec:formalization}.
\end{rmk}

Thus, by starting from the desired form of a Reedy model structure, we have arrived almost inexorably at the definition of Reedy category.
The only short gap left is between almost-Reedy and Reedy categories.
The following example shows that this really is a gap.

\begin{eg}\label{eg:almost-reedy}
  Consider the commutative square category, with degrees assigned as shown below:
  
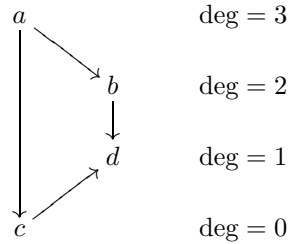
\[ \xymatrix@R=1pc{
    a \ar[ddd] \ar[dr] && \deg=3\\
    & b \ar[d] & \deg=2\\
    & d  & \deg=1\\
    c \ar[ur] && \deg=0}
  \]
  It is easy to check that it is almost-Reedy: the only non-basic arrow is the common composite $a\to d$, and it has only one fundamental factorization $a\to c\to d$.
  But this category (with this degree function) is not Reedy, since the basic arrows $a\to b$ and $b\to d$ have a non-basic composite.
  (This same category does admit a different degree function for which it is Reedy.)

  In the model structure on diagrams over this almost-Reedy category, a diagram $X$ is Reedy fibrant if $X_c$ and $X_d$ are fibrant, $X_b\to X_d$ is a fibration, and the induced map $X_a \to X_c \times_{X_d} X_b$ is a fibration.
  It is Reedy cofibrant if $X_a$, $X_b$, and $X_c$ are cofibrant and $X_c \to X_d$ is a cofibration.
\end{eg}

Is there a reason to prefer Reedy categories to almost-Reedy ones?
One obvious one is that the definition is easier to verify in practice for non-toy examples.
Another is that in a Reedy category, the matching object $M_x A$ can be computed using only the maps out of $x$ that are in $\rdn\C$, rather than all of them.

More precisely, note that the matching object $M_x A = \wlim{\C(x,-)}{A}{\C_\de}$ in an almost-Reedy category can equivalently be defined as an ordinary (``conical'') limit over the category $\strsl{x}{\C}$, whose objects are morphisms $f:x\to y$ with $\deg(x)<\deg(y)$ and whose morphisms are commutative triangles under $x$.
In a Reedy category, we also have the subcategory $\strsl{x}{\rdn\C}$, whose objects are nonidentity morphisms $f:x\to y$ in $\rdn\C$ and whose morphisms are commutative triangles in $\rdn\C$ under $x$.
It can be shown (see \cref{thm:reedy-initial} below) that the inclusion $\strsl{x}{\rdn\C} \into \strsl{x}{\C}$ is an \emph{initial functor}~\cite[\S IX.3]{maclane}, so that the matching object can equivalently be computed as a limit over $\strsl{x}{\rdn\C}$.

This is not true in an almost-Reedy category, as can be seen from \cref{eg:almost-reedy}: the definition of the matching object $M_a X = X_c \times_{X_d} X_b$ requires the map $X_c \to X_d$, which is not in $\rdn\C$.
In fact, in the almost-Reedy case, $\strsl{x}{\rdn\C}$ may not even be a category!
However, with a bit of care we can still make precise what it would mean for this to be true in an almost-Reedy category, and it turns out to give another equivalent characterization of Reedy categories.

Recall that a category can be considered as a (directed) \emph{graph} equipped with composition and identities.
If $\M$ is a category, $\D$ is a graph, and $F:\D\to\M$ is a graph morphism, we can define the \textbf{limit} of $F$ just as we would for a functor between categories: it is an object $M\in \M$ equipped with morphisms $M\to F(d)$ for all vertices (i.e.\ ``objects'') $d\in \D$, such that the evident triangle commutes for all edges (i.e.\ ``morphisms'') in $\D$.

If $\D$ is a graph, $\E$ is a category, $G:\D\to\E$ is a graph morphism, and $e\in \E$ is an object, let $G/e$ denote the graph whose vertices are morphisms $G(d) \to e$ in $\E$, for some vertex $d\in\D$, and whose edges are edges $d\to d'$ in $\D$ forming a commutative triangle in $\E$.
We say that $G$ is \textbf{initial} if for any $e\in\E$ the graph $G/e$ is connected, i.e.\ it is nonempty and any two vertices are related by some zigzag of edges.
If $\D$ is a category and $G$ a functor, then $G/e$ is the underlying graph of the usual comma category, and this reduces to the usual notion of initial functor.
It is straightforward to adapt the usual proof for categories to show that a graph morphism $G:\D\to\E$ is initial if and only if for any functor $F:\E\to\M$ to a complete category, the induced comparison map $\lim F \to \lim F G$ is an isomorphism.
There is a dual notion of \textbf{final} graph morphism which induces an isomorphism between colimits.

Now if $\C$ is almost-Reedy and $x\in\C$, we have a graph $\strsl{x}{\rdn\C}$ whose vertices are nonidentity morphisms $f:x\to y$ in $\rdn \C$, and whose edges are morphisms $y\to y'$ in $\rdn \C$ forming a commutative triangle.
This graph is not necessarily a category since $\rdn\C$ may not be closed under composition.
However, we have an obvious graph morphism $\strsl{x}{\rdn\C} \into \strsl{x}{\C}$.
Similarly, we have a graph $\strsl{\rup\C}{x}$ with a graph morphism $\strsl{\rup\C}{x} \into \strsl{\C}{x}$.

\begin{thm}\label{thm:reedy-initial}
  Let $\C$ be almost-Reedy; then the following are equivalent.
  \begin{enumerate}
  \item $\C$ is Reedy.\label{item:ri1}
  \item For any $x\in\C$, the graph morphism $\strsl{x}{\rdn\C} \into \strsl{x}{\C}$ is initial and the graph morphism $\strsl{\rup\C}{x} \into \strsl{\C}{x}$ is final.\label{item:ri2}
  \end{enumerate}
\end{thm}
\begin{proof}
  If $\C$ is Reedy, then $\strsl{x}{\C}$ is a category and $\iota:\strsl{x}{\rdn\C} \into \strsl{x}{\C}$ is a functor.
  Let $f:x\to y$ be an object of $\strsl{x}{\C}$; we want to show that $(\iota/f)$ is connected.
  Factoring $f$ as $\rup f \rdn f$ shows that this category is nonempty.
  Another object of this category is a factorization $f = h g$ with $g\in\rdn\C$; thus $f = \rup h \rdn h g$.
  Since $\rdn\C$ is closed under composition, $\rdn h g \in \rdn \C$, and so by unique factorization we must have $\rup h = \rup f$ and $\rdn h g = \rdn f$.
  Thus the factorization $h g$ is connected to $\rup f \rdn f$ by a 1-step zigzag consisting of $\rdn h$, which also lies in $\strsl{x}{\rdn\C}$.
  The argument for $\rup\C$ is dual, so~\ref{item:ri1} implies~\ref{item:ri2}.

  Conversely, suppose~\ref{item:ri2} holds; we will show that $\rdn\C$ and $\rup\C$ are closed under composition, so that by \cref{thm:reedy} $\C$ is Reedy.
  Suppose $f:x\to y$ with $f = h g$ where $h,g\in\rdn\C$; we want to show $f\in\rdn\C$ (the argument for $\rup\C$ is dual).
  By induction, we may suppose that this is true for all composable pairs whose domain has degree $<\deg(x)$.
  We may also suppose neither $g$ nor $h$ is an identity.
  Then this factorization and the unique factorization $f=\rup f \rdn f$ from \cref{thm:reedy-fact} are both vertices of $(\iota/f)$, where $\iota:\strsl{x}{\rdn\C} \into \strsl{x}{\C}$ is the above graph morphism.
  Thus, they are connected by a zigzag of factorizations with comparison maps in $\rdn \C$.

  Now if $f\notin\rdn\C$, then the factorization $(\rup f, \rdn f)$ has degree less than $\deg(y)$, while since $h\in\rdn\C$ is not an identity the factorization $(h, g)$ has degree greater than that of $y$.
  Thus, as we move along the zigzag from $(h,g)$ to $(\rup f, \rdn f)$, there must be a first factorization that has degree $\le\deg(y)$.

  I claim all the factorizations occurring \emph{prior} to this point have \emph{both} factors in $\rdn \C$.
  Their first factors are all in $\rdn\C$ by definition.
  Thus it suffices to show that given arrows as in \cref{fig:fact-zag},
  \begin{figure}
    \centering
    \begin{subfigure}[b]{0.4\textwidth}
      \centering
      $\xymatrix{ &x \ar[dl]_{g'} \ar[ddr]^{g''}\\
        \ar[ddr]_{h'} \ar[drr]^{k}\\
        && \ar[dl]^{h''}\\
        & y }$
      \caption{}
      \label{fig:fact-zag}
    \end{subfigure}
    \begin{subfigure}[b]{0.4\textwidth}
      \centering
      $\xymatrix{ & x \ar[dl]_{g'} \ar[ddr]^{g''}\\
        \ar[d]_{h'} \ar[drr]^{k}\\
        y && z \ar[ll]^{h''}
      }$
      \caption{}
      \label{fig:fact-zag2}
    \end{subfigure}
    \caption{Configurations for the proof of \cref{thm:reedy-initial}}
    \label{fig:reedy-initial}
  \end{figure}
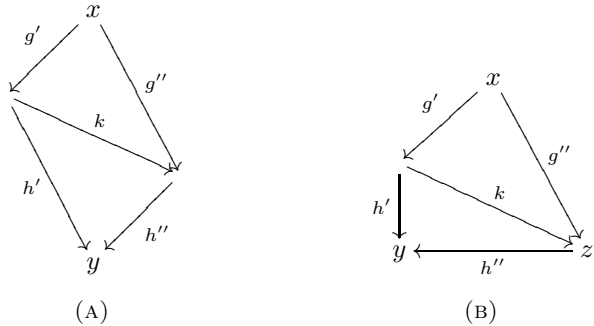
  where $g', g'', k\in\rdn\C$ and degrees strictly decrease downwards, we have $h'\in\rdn\C$ if and only if $h''\in\rdn\C$.
  However, any fundamental factorization of $h''$ would yield, by composition with $k$, a fundamental factorization of $h'$; so if $h'$ is basic then so is $h''$.
  On the other hand, if $h''\in\rdn\C$, then since the domain of $k$ has degree $<\deg(x)$ (as $g'$ is not an identity, being a vertex of $\strsl{x}{\rdn\C}$), the inductive hypothesis implies that $h' = h'' k$ is also in $\rdn\C$.
  This proves the claim.

  Now at the first factorization with degree $\le\deg(y)$, we have a diagram as shown in \cref{fig:fact-zag2},
  with $g', g'', k, h'\in\rdn\C$ and $\deg(z)\le\deg(y)$.
  If $h''$ is not an identity, then the factorization $\rup{h''}\rdn{h''}$ has degree strictly less than $\deg(y)$, and so $h' = \rup{h''}\left(\rdn{h''} k\right)$ is a fundamental factorization, contradicting $h'\in\rdn\C$.
  Thus $h''$ must be an identity, so that $f = g'' \in\rdn \C$ as desired.
\end{proof}

\section{Generalized Reedy categories}
\label{sec:auto}

In \cref{sec:reedy}, we restricted attention to bistratified categories with discrete strata.
In addition to being the route to the classical notion of Reedy category, this is the most natural way to obtain a fully explicit Reedy-like model structure on $\M^\C$, since $\M^I$ for discrete $I$ has a model structure that is completely objectwise.

If we allow abstract bigluing data from \C to some non-discrete category \D, then in order to apply \cref{thm:model} we would need a given model structure on $\M^\D$.
But such model structures do exist under certain hypotheses, such as the projective and injective ones.
In this section, we will consider what happens if we use a projective model structure on each $\M^\D$.
We will call the resulting generalized notions \textbf{c-Reedy} (for ``categorified''), and if necessary we will refer to the corresponding notions from \cref{sec:reedy} as \textbf{s-Reedy} (for ``set'' or ``strict'').
Using the projective rather than the injective (or some other) model structure is a choice; the injective case can be deduced by duality, but the resulting notion of c-Reedy category would be different.
If necessary for disambiguation, one might speak of \emph{projectively} or \emph{injectively} c-Reedy categories; but since we consider only the projective case we will avoid such cumbersome terminology.

An interesting intermediate case is when the categories \D are all required to be groupoids.
This exhibits certain simplifying features, has many important applications, and leads to a comparison with the theory of~\cite{cisinski:presheaves} and~\cite{bm:extn-reedy}.
We will refer to this case as \textbf{g-Reedy} (for ``groupoidal'' or ``generalized'', the latter following the terminology of the cited references).

One of the simplifications in the groupoidal case is that since all isomorphisms in a bistratified category are basic and level, we don't need to assert separately that the latter are closed under composition.
This is expressed by the following analogue of \cref{thm:bistrat-discrete}, which also follows immediately from \cref{thm:bistrat-char}.

\begin{thm}\label{thm:bistrat-gpd}
  A category \C whose objects are assigned ordinal degrees is bistratified with groupoidal strata if and only if
  \begin{enumerate}
  \item The basic level morphisms are exactly the isomorphisms, and
  \item The category of fundamental factorizations of any non-basic level morphism is connected.
  \end{enumerate}
\end{thm}

We now introduce a more general sort of \textbf{matching object functor} for arbitrary bistratified categories.
For any degree $\de$, let us write $M_\delta$ for the composite
\[ \M^\C \xto{\iota_\de^*} \M^{\C_\de} \xto{\wlim{\C(\blank,\blank)}{\blank}{\C_\de}} \M^{\C_{=\de}} \]
in which the second functor is the weighted limit used in the $\de^{\mathrm{th}}$ stage of bigluing.
(Recall that $\C_{=\de}$ denotes the $\de^{\mathrm{th}}$ stratum of \C, i.e.\ the \emph{non-full} subcategory consisting of the objects of degree $\de$ and the \emph{basic} morphisms between them.)
The difference with the matching object functors from \cref{sec:reedy} is now that we don't want to separate the components at different objects of $\C_{=\de}$.

This is again isomorphic to a weighted limit
\[ (M_\de A)_x \cong \wlimc{\partial_\de\C(x,\blank)}{A},\]
with weight defined by
\[ \partial_\de\C(x,y) = \wcolim{\C(\blank,y)}{\C(x,\blank)}{\C_\de}. \]
Note that here we are regarding $\partial_\de\C$ as a profunctor from $\C$ to $\C_{=\de}$, so that this weighted limit takes values in $\M^{\C_{=\de}}$.
Similarly, we have the \textbf{latching object functor} defined by
\[ (L_\de A)_x \cong \wcolimc{\partial_\de\C(\blank,x)}{A} \]
in which we regard $\partial_\de \C$ as a profunctor from $\C_{=\de}$ to $\C$.
As in \cref{sec:reedy}, we will also continue to write $M_x A$ for the value $(M_\de A)_x$ of the functor $M_\de A \in \M^{\C_{=\de}}$ at the object $x\in\C_{=\de}$, and similarly for $L_x$.
We may also consider $M_x$ and $L_x$ to take values in $\M^{\aut(x)}$, where $\aut(x)$ is the automorphism group of $x$ in $\C_{=\de}$ (or equivalently in \C, since all isomorphisms are basic level); this is mainly important when the strata are groupoidal.
Finally, we will write $A_\de$ for the restriction of $A\in\M^\C$ to a functor on $\C_{=\de}$.

\begin{thm}\label{thm:bistrat-wfs}
  If \C is any bistratified category, then for any complete and cocomplete category \M with a pre-wfs, the category $\M^\C$ inherits a \textbf{c-Reedy pre-wfs} in which
  \begin{enumerate}
  \item $A\to B$ is in \R iff the induced map $A_x \to M_x A \times_{M_x B} B_x$ in \M is in \R for all $x\in \C$.\label{item:bswfsR}
  \item $A\to B$ is in \L iff the induced map $L_\de B \sqcup_{L_\de A} A_\de \to B_\de$ in $\M^{\C_{=\de}}$ is a projective \L-map for all degrees $\de$.\label{item:bswfsL}
  \end{enumerate}
  If \C has groupoidal strata, then condition~\ref{item:bswfsL} is equivalent to
  \begin{enumerate}
  \item[\ref{item:bswfsL}$'$] The induced map $L_x B \sqcup_{L_x A} A_x \to B_x$ in $\M^{\aut(x)}$ is a projective \L-map for all $x\in\C$.
  \end{enumerate}
  Moreover, if the projective pre-wfs on each $\M^{\C_{=\de}}$ is a wfs, then the c-Reedy pre-wfs is also a wfs.
\end{thm}
\begin{proof}
  Just like \cref{thm:bistrat-disc-wfs}, but bigluing with the projective pre-wfs on $\M^{\C_{=\de}}$ instead of the objectwise one on $\M^I$.
  The claim about the groupoidal case follows because the projective pre-wfs on diagrams over a groupoid is equivalent to the product of the projective pre-wfs on diagrams over each of its connected components.
\end{proof}

As in \cref{sec:reedy}, we now want to characterize when the two resulting wfs on $\M^\C$, for \M a model category, fit together into a model structure with objectwise weak equivalences.

We begin with an analogue of \cref{thm:reedy-wfs}, which is unsurprisingly somewhat more complicated.
For one thing, because in \cref{sec:reedy} the strata were discrete, we could restrict attention to single hom-functors $\C(x,\blank)$ and $\C(\blank,x)$; thus in \cref{thm:reedy-wfs} it was sufficient to consider limits and colimits weighted by functors $W\in \nSet^\C$ and $U\in\nSet^{\C\op}$ respectively.
We could instead have considered the ``parametrized'' weighted limit functor $(\nSet^{\C\op\times I})\op \times \M^\C \to \M^I$, where $I$ a set (eventually the set of objects of a given degree), but the discreteness of $I$ would imply that the wfs on $\nSet^{\C\op\times I}$ and $\M^I$ are all objectwise, so there would be no real change.

Now, however, our strata are no longer discrete, so we must consider instead weighted limit and colimit functors such as $(\nSet^{\C\times\D\op})\op \times \M^\C \to \M^\D$ and $\nSet^{\C\op\times \D} \times \M^\C \to \M^\D$, with $\D$ a category.
In particular, this means that categories such as $\nSet^{\C\op\times\D}$, $\nSet^{\C\times\D\op}$, and $\M^\D$ inherit more than one (pre-)wfs, and we have to choose the correct ones to make the lemma true.
We want to use the c-Reedy wfs for diagrams over $\C$ and $\C\op$, but diagrams over $\D$ and $\D\op$ could be given either an injective or a projective wfs.
Furthermore, we could write $\nSet^{\C\times\D\op}$ (say) as either $(\nSet^\C)^{\D\op}$ or $(\nSet^{\D\op})^\C$, thus applying either the c-Reedy construction first and then an injective or projective one, or vice versa.
Finally, because of our choice to use \emph{projective} structures on strata rather than injective ones, the analogue of \cref{thm:reedy-wfs-opp} is not true for c-Reedy wfs; thus we would obtain different wfs on $\nSet^{\C\op}$ by applying the c-Reedy construction with the bistratified category $\C\op$, or by writing it as $((\nSet\op)^\C)\op$ and applying the c-Reedy construction with $\C$ itself.

It turns out that the correct thing to do is to write $\nSet^{\C\op\times \D} \cong (((\nSet\op)^\C)\op)^\D$ and $\nSet^{\C\times\D\op} \cong (\nSet^\C)^{\D\op}$ --- thus first applying the c-Reedy construction with $\C$ itself in both cases --- and then use the \emph{injective} wfs with respect to $\D$ and $\D\op$.
The resulting wfs will be the only ones we consider on the categories $\nSet^{\C\op\times \D}$ and $\nSet^{\C\times\D\op}$, so we will simply denote them by $(\L,\R)$ as usual.
For the category $\M^\D$, it turns out that we need to use the projective pre-wfs for limits, but the injective one for colimits.
In addition to making the lemma true, these choices have the advantage of making as many things as possible objectwise and thus simple to analyze.

\begin{lem}\label{thm:creedy-wfs}
  Suppose \C is bistratified, \M is a complete and cocomplete category with a pre-wfs, and \D is a small category.
  \begin{itemize}
  \item Let $\M^\C$ have its c-Reedy pre-wfs and $\nSet$ have its (injection, surjection) wfs.
  \item Let $\nSet^\C$ and $\nSet^{\C\op} \cong ((\nSet\op)^\C)\op$ have their c-Reedy wfs.
  \item Let $\nSet^{\C\op\times \D} \cong (((\nSet\op)^\C)\op)^\D$ and $\nSet^{\C\times\D\op} \cong (\nSet^\C)^{\D\op}$ have their induced injective pre-wfs.
  \end{itemize}
  Then:
  \begin{enumerate}
  \item The weighted limit $(\nSet^{\C\times\D\op})\op \times \M^\C \to \M^\D$ is a right wfs-bimorphism when $\M^\D$ has its projective pre-wfs.\label{item:creedywfs1}
  \item The weighted colimit $\nSet^{\C\op\times \D} \times \M^\C \to \M^\D$ is a left wfs-bimorphism when $\M^\D$ has its injective pre-wfs.\label{item:creedywfs2}
  \end{enumerate}
\end{lem}
\begin{proof}
  We prove~\ref{item:creedywfs1}; the proof of~\ref{item:creedywfs2} is dual.
  Since $\nSet^{\C\times\D\op} \cong (\nSet^\C)^{\D\op}$ and its \L-maps are objectwise with respect to $\D\op$, and the \R-maps in $\M^\D$ are objectwise with respect to \D, it suffices to show that the weighted limit $(\nSet^\C)\op \times \M^\C \to \M$ is a right wfs-bimorphism.
  (This is a simple version of~\cite[Theorem 3.2]{gambino:wgtlim}.)
  Now it suffices to show by adjunction that the objectwise copower $\nSet^\C \times \M \to \M^\C$ is a left wfs-bimorphism.
  For this we can apply the same argument as in \cref{thm:reedy-wfs}, using \cref{thm:wfs-paramor-pres}, \cref{thm:bistrat-parapostcomp}, and induction, once we have as a starting point that the copower functor $\nSet \times \M \to \M$ is a left wfs-bimorphism.
  But this is exactly \cref{thm:copower}.
\end{proof}

Recall from \cref{thm:objwise-cwe} that if \M is a model category, then the projective acyclic fibrations in $\M^\D$ are universal weak equivalences and the injective acyclic cofibrations in $\M^\D$ are couniversal weak equivalences.
Thus, if $W$ is an \L-object in $\nSet^{\C\times\D\op}$, then $\wlimc W \blank:\M^\C \to \M^\D$ takes acyclic fibrations to universal weak equivalences, and if $U$ is an \L-object in $\nSet^{\C\op\times \D}$, then $(\wcolimc U \blank) :\M^\C \to\M^\D$ takes acyclic cofibrations to couniversal weak equivalences.

In constrast to the situation in \cref{sec:reedy}, however, these conditions on $U$ and $W$ are \emph{not} dual to each other; this is due again to our choice to use the projective structure on strata in the c-Reedy wfs.
On the one hand, $U$ being an \L-object in $\nSet^{\C\op\times \D}$ means that for each $d\in\D$, $U(\blank,d)$ is an \L-object in $((\nSet\op)^\C)\op$.
As remarked above, this is not the same as being an \L-object in $\nSet^\C$, since the c-Reedy wfs uses the projective and not the injective model structure; instead it means being an \R-object in $(\nSet\op)^\C$, which by \cref{thm:bistrat-wfs} means that for each $y\in \C$, the map ``$U(y,d) \to M_y U(\blank,d)$'' is an \R-map in $\nSet\op$.
In other words, its opposite $L_y U(\blank,d) \to U(y,d)$ must be an \L-map in $\nSet$, i.e.\ an injection.

On the other hand, $W$ being an \L-object in $\nSet^{\C\times\D\op}$ says, by definition, that each $W(d,\blank)$ is an \L-object in $\nSet^\C$.
In other words, we require that each map $L_\de W(d,\blank) \to W(d,\blank)$ is a \emph{projective} \L-map in $\nSet^{\C_{=\de}}$.
Thus, the condition on $W$ involves projective-cofibrancy, while the condition on $U$ does not.
To express this condition more concretely, we need the following.

\begin{lem}\label{thm:projcof}
  For any small category \D, the projective \L-maps in $\nSet^\D$ are the complemented injections whose complement is a coproduct of retracts of representable functors.
  In particular, for any group $G$, the projective \L-maps in $\nSet^G$ are the injections whose complement has a free $G$-action.
\end{lem}
\begin{proof}
  Since the wfs on \nSet is cofibrantly generated, its projective pre-wfs is in fact a wfs.
  The generating projective \L-maps in any projective wfs are the underlying generating \L-maps tensored with representable functors.
  The only generating cofibration of $\nSet$ is $\emptyset \to 1$, so the only generating cofibrations of $\nSet^\D$ are $\emptyset \to \D(d,\blank)$ for some $d\in\D$.
  A relative cell complex built out of these is exactly a complemented injection whose complement is a coproduct of representables.
  Thus, the \L-maps are the retracts of these, which are as claimed.

  The case for groups follows because the only representable in $\nSet^G$ is the free orbit $G/e$, and it has no nontrivial retracts.
\end{proof}

We can now prove analogues of \cref{thm:quillen-almostreedy}, \cref{thm:acof-cat}, and \cref{thm:almost-reedy}.

\begin{lem}\label{thm:quillen-creedy}
  Let \C be bistratified, \D a small category, and $U\in \nSet^{\C\op\times \D}$ and $W\in\nSet^{\C\times \D\op}$, equipped with their pre-wfs from \cref{thm:creedy-wfs}.
  Then the following are equivalent:
  \begin{enumerate}
  \item $U$ is an \L-object.
  \item For any \M with a wfs, $(\wcolimc{U}{\blank}):\M^\C\to\M^\D$ takes c-Reedy \L-maps to injective \L-maps.
  \item For any \M with a wfs and any $d\in \D$, $(\wcolimc{U(\blank,d)}{\blank}):\M^\C\to\M$ takes c-Reedy \L-maps to \L-maps.
  \end{enumerate}
  Similarly, the following are equivalent:
  \begin{enumerate}
  \item $W$ is an \L-object.
  \item For any \M with a wfs, $\wlimc{W}{\blank}:\M^\C\to\M^\D$ takes c-Reedy \R-maps to projective \R-maps.
  \item For any \M with a wfs and any $d\in \D$, $\wlimc{W(d,\blank)}{\blank}:\M^\C\to\M$ takes c-Reedy \R-maps to \R-maps.
  \end{enumerate}
\end{lem}
\begin{proof}
  In both cases, the second and third conditions are equivalent by the definition of injective \L-maps and projective \R-maps.
  Moreover, in both cases the first condition implies the second by \cref{thm:creedy-wfs}.
  Thus it suffices to show that the third condition implies the first.

  In the case of $U$, by the remarks after \cref{thm:creedy-wfs}, it suffices to show that for each $d$, the map $L_x U(\blank,d)\to U(x,d)$ is an injection for each $x\in \C$.
  We can largely repeat the proof of \cref{thm:quillen-almostreedy}; the only hitch is that instead of showing that~\eqref{eq:reedy-char-po} is an injection, we need to show that 
  \begin{equation}
    \partial_\be\C(x,\blank) \sqcup_{L_\de \partial_\be\C(x,\blank)} L_\de \C(x,\blank) \too \C(x,\blank)\label{eq:creedy-char-po}
  \end{equation}
  is a projective \L-map in $\nSet^{\C_{\de}}$, for any degree $\de$.
  If $\de<\be$, then this is again an isomorphism.
  If $\de=\be$, the same argument reduces it to $\partial_\be\C(x,\blank) \to \C(x,\blank)$.
  However, as remarked above, the complement of $\partial_\be\C(x,y) \to \C(x,y)$ is the set of basic morphisms $x\to y$, which since $x$ and $y$ both have degree $\be$ is just $\C_{=\be}(x,y)$.
  Thus, the map $\partial_\be\C(x,\blank) \to \C(x,\blank)$ is complemented and its complement is the representable $\C_{=\be}(x,\blank)$.
  By \cref{thm:projcof}, therefore, it is a projective \L-map.

  In the case of $W$, it similarly suffices to show that each $W(d,\blank)$ is an \L-object in $\nSet^\C$, i.e.\ that the map $L_\de W(d,\blank)\to W(d,\blank)$ is a projective \L-map in $\nSet^{\C_{=\de}}$ for each degree $\de$.
  In this case, we repeat the proof of \cref{thm:quillen-almostreedy} with $\nSet^{\C_{=\be}\op}$, with its projective wfs, in place of $\nSet$.
  (We also have to dualize, so that we consider weighted limits in $(\nSet^{\C_{=\be}\op})\op$ to get weighted colimits in $\nSet^{\C_{=\be}\op}$.)
  Now we end up having to show that $\partial_\be\C(\blank,y) \to \C(\blank,y)$ is a projective \L-map in $\nSet^{\C_{=\be}\op}$ for any $y$ of degree $\be$; but this follows from \cref{thm:projcof} by the same reasoning.
\end{proof}

\begin{lem}\label{thm:acof-catinj}
  Every couniversal weak equivalence in $\nCat^\D$ is an injective (i.e.\ objectwise) acyclic cofibration.
\end{lem}
\begin{proof}
  Suppose $f:A\to B$ is a couniversal weak equivalence, and suppose for contradiction there were $d\in \D$ and objects $x,y\in A_d$ with $x\neq y$ but $f(x)=f(y)$ in $B_d$.
  Since $f$ is an equivalence, there is a unique isomorphism $e:x\cong y$ in $A_d$ with $f(e) = \id$.
  Let $E$, $g$, $C$, and $h$ be as in the proof of \cref{thm:acof-cat} with $A_d$ replacing $A$.
  Note that $hg:A_d \to C$ sends \emph{all} objects of $A_d$ to the single object $c$.

  Now define $\Chat\in\nCat^\D$ by $\Chat_{d'} = C^{\D(d',d)}$, i.e.\ $\Chat$ is the image of $C$ under the right adjoint to the ``evaluate at $d$'' functor $\nCat^\D \to \nCat$.
  Thus, the functor $hg:A_d \to C$ corresponds to a map $k:A\to \Chat$, where for $\psi:d'\to d'$ we have $k_{d'}(\blank)(\psi) = hg(A_\psi(\blank))$.
  In particular, $k_{d'}$ sends \emph{all} objects of $A_{d'}$ to the constant function $\psi\mapsto c$.
  Thus, $k_d(e)$ is a nonidentity automorphism of this object.
  But now if we let $D$ be the pushout of $f$ and $k$, we find that $k_d(e)$ must go to an identity in $D$, so that $\Chat \to D$ is not an objectwise equivalence.
\end{proof}

\begin{thm}\label{thm:almost-creedy}
  Suppose \C is bistratified, and that for any model category \M admitting projective model structures on each $\M^{\C_{=\de}}$, the two wfs on $\M^\C$ form a model structure with the objectwise weak equivalences.
  Let \D be a small category and $U\in\nSet^{\C\op\times \D}$ and $W\in\nSet^{\C\times \D\op}$, equipped with their pre-wfs from \cref{thm:creedy-wfs}.
  \begin{enumerate}
  \item $({\wcolimc{U}{\blank}}):\M^\C\to\M^\D$ takes Reedy acyclic cofibrations to couniversal weak equivalences for all such model categories \M if and only if $U$ is an \L-object.\label{item:acr1}
  \item $\wlimc{W}{\blank}:\M^\C\to\M^\D$ takes Reedy acyclic fibrations to universal weak equivalences for all such model categories \M if and only if $W$ is an \L-object.\label{item:acr2}
  \end{enumerate}
\end{thm}
\begin{proof}
  In both cases ``only if'' follows from \cref{thm:creedy-wfs}.
  For the converse, we take $\M=\nCat$ in~\ref{item:acr1} and $\M=\nCat\op$ in~\ref{item:acr2}, so that by \cref{thm:acof-catinj}, both functors take values in objectwise acyclic cofibrations.
  We can now ignore the extra parameter object $d\in\D$, reduce from acyclic cofibrations of categories to injections of sets as in the proof of \cref{thm:almost-reedy}, and apply \cref{thm:quillen-creedy}.
\end{proof}

This motivates the following definition.

\begin{defn}\label{defn:almost-creedy}
  A bistratified category is \textbf{almost c-Reedy} if for all objects $x$ of degree \be,
  \begin{itemize}
  \item $\C(\blank,x)$ is an \L-object in the c-Reedy wfs on $\nSet^{\C_{\be}\op} \cong ((\nSet\op)^{\C_{\be}})\op$, and
  \item $\C(x,\blank)$ is an \L-object in the c-Reedy wfs on $\nSet^{\C_\be}$.
  \end{itemize}
  It is \textbf{almost g-Reedy} if in addition all its strata are groupoidal.
\end{defn}

Thus, we have the following theorem.

\begin{thm}\label{thm:creedy-model}
  Let \C be almost c-Reedy.
  Then for any model category \M such that each $\M^{\C_{=\de}}$ has a projective model structure, $\M^\C$ has a model structure such that
  \begin{itemize}
  \item The weak equivalences are objectwise.
  \item $A\to B$ is a fibration (resp.\ acyclic fibration) iff  for all $x\in \C$, the induced map $A_x \to M_x A \times_{M_x B} B_x$ is a fibration (resp.\ acyclic fibration) in \M.
  \item $A\to B$ is a cofibration (resp.\ acyclic cofibration) iff for all degrees $\de$, the induced map $L_\de B \sqcup_{L_\de A} A_\de \to B_\de$ is a projective-cofibration (resp.\ acyclic projective-cofibration) in $\M^{\C_{=\de}}$.
  \end{itemize}
  If \C is almost g-Reedy, then the condition to be a cofibration can be simplified to
  \begin{itemize}
  \item $A\to B$ is a cofibration iff the induced map $L_x B \sqcup_{L_x A} A_x \to B_x$ is a projective-cofibration in $\M^{\aut(x)}$ for all $x\in\C$.
  \end{itemize}
  and similarly in the acyclic case.\qed
\end{thm}

\begin{rmk}
  As observed in~\cite{bm:extn-reedy}, the non-self-duality of (almost) c- and g-Reedy categories intertwines with the duality between projective and injective model structures.
  Specifically, if $\C\op$ is almost c-Reedy and \M admits \emph{injective} model structures, then $\M^\C$ has an induced model structure which can be obtained by regarding it as $((\M\op)^{\C\op})\op$.
\end{rmk}

We emphasize again that because the c-Reedy wfs is not self-dual, neither is the notion of almost c-Reedy (or almost g-Reedy) category.
Specifically, the first condition in \cref{defn:almost-creedy} reduces to the same condition as in \cref{sec:reedy} that
\begin{equation}
  \partial_\de\C(y,x) \to \C(y,x)\label{eq:cofUmap}
\end{equation}
is an injection whenever $\de = \deg(y)<\deg(x)$; while the second requires instead that for each $x$ and each degree $\de<\deg(x)$, the map
\begin{equation}
  \partial_\de\C(x,\blank) \to \C(x,\blank)\label{eq:cofWmap}
\end{equation}
is a projective \L-map in $\nSet^{\C_{=\de}}$.
By \cref{thm:projcof}, this latter means a complemented injection whose complement is a coproduct of retracts of representables.
(Note that the complement of~\eqref{eq:cofWmap} at $y\in\C_{=\de}$ is just the set of basic morphisms $x\to y$.)

In the g-Reedy case, the latter reduces to asking that for each $y$ with $\deg(y)<\deg(x)$, the map
\begin{equation}
  \partial_\de\C(x,y) \to \C(x,y)
\end{equation}
is a projective \L-map in $\nSet^{\aut(y)}$, which means an injection whose complement has a free $\aut(y)$-action.

This enables us to characterize the almost c-Reedy categories more concretely.

\begin{thm}\label{thm:creedy-char}
  For a category with a degree function on its objects to be almost c-Reedy, it is necessary and sufficient that
  \begin{enumerate}
  \item All identities are basic.\label{item:creedychar0}
  \item Basic level morphisms are closed under composition.\label{item:creedychar1}
  \item The category of fundamental factorizations of any non-basic morphism is connected.\label{item:creedychar2}
  \item If $f$ is basic and strictly decreases degree, and $g$ is basic level, then $g f$ is again basic.
    Thus, for any $x$ and any $\de<\deg(x)$, the basic morphisms from $x$ constitute a functor $\C_{=\de}\to\nSet$.\label{item:creedychar3}
  \item Each of the functors in~\ref{item:creedychar3} is a coproduct of retracts of representables.\label{item:creedychar4}
  \end{enumerate}
\end{thm}
\begin{proof}
  By \cref{thm:bistrat-char}, conditions~\ref{item:creedychar0},~\ref{item:creedychar1}, and~\ref{item:creedychar2} are equivalent to \C being bistratified and~\eqref{eq:cofUmap} and~\eqref{eq:cofWmap} being objectwise injections.
  By \cref{thm:projcof}, conditions~\ref{item:creedychar3} and~\ref{item:creedychar4} are then equivalent to saying that~\eqref{eq:cofWmap} is a projective \L-map.
\end{proof}

In the groupoidal case, the characterization simplifies a bit.

\begin{cor}\label{thm:greedy-char}
  For a category with a degree function on its objects to be almost g-Reedy, it is necessary and sufficient that
  \begin{enumerate}
  \item The basic level morphisms are exactly the isomorphisms.\label{item:greedychar1}
  \item The category of fundamental factorizations of any non-basic morphism is connected.\label{item:greedychar2}
  \item The automorphisms of any object $y$ act freely (by composition) on the basic morphisms with codomain $y$ that strictly decrease degree.\label{item:greedychar3}
  \end{enumerate}
\end{cor}
\begin{proof}
  Replace \cref{thm:bistrat-char} with \cref{thm:bistrat-gpd} and use the second part of \cref{thm:projcof} (note that \cref{thm:creedy-char}\ref{item:creedychar3} is automatic in this case).
\end{proof}

As we did in \cref{sec:reedy}, let us define $\rup\C$ and $\rdn\C$ in an almost c-Reedy category to consist of the basic morphisms that non-strictly raise and lower degree, respectively.
Thus $\rup\C\cap\rdn\C$ is the category of basic level morphisms.
As in \cref{sec:reedy}, $\rup\C$ and $\rdn\C$ are not necessarily subcategories, but we do have the following:

\begin{lem}\label{thm:creedy-compose}
  Let \C be almost c-Reedy.
  \begin{enumerate}
  \item If $f$ and $g$ are basic level, then so is $g f$.\label{item:crc1}
  \item If $f\in \rdn\C$ and $g$ is basic level, then $g f \in\rdn\C$.\label{item:crc2}
  \item If $g f\in \rdn\C$ and one of $f$ or $g$ is basic level, then the other is in $\rdn\C$.\label{item:crc4}
  \item If $g f\in \rup\C$ and one of $f$ or $g$ is basic level, then the other is in $\rup\C$.\label{item:crc5}
  \end{enumerate}
\end{lem}
\begin{proof}
  \ref{item:crc1} is just \cref{thm:creedy-char}\ref{item:creedychar1}, while~\ref{item:crc2} is that combined with \cref{thm:creedy-char}\ref{item:creedychar3}.
  For~\ref{item:crc4}--\ref{item:crc5}, if $g$ is basic level, then composing a fundamental factorization of $f$ with $g$ would yield a fundamental factorization of $g f$, and similarly if $f$ is basic level.
\end{proof}

\begin{lem}\label{thm:creedy-cone}
  In an almost c-Reedy category, if $g_1 f_1 = f_2 = g_3 f_3$, where $f_1,f_2,f_3\in\rdn\C$ and strictly decrease degree,
  and $g_1$ and $g_2$ are basic level, then there exists $f_0\in \rdn\C$ and basic level $h_1$ and $h_3$ such that $g_1 h_1 = g_3 h_3$ and $f_1 = h_1 f_0$ and $f_3 = h_3 f_0$.
\end{lem}
\begin{proof}
  \cref{thm:creedy-char}\ref{item:creedychar4} says that the functor $\rdn\C(x,\blank):\C_{=\de}\to\nSet$ is a coproduct of retracts of representables.
  This is equivalent to saying that its category of elements is a disjoint union of categories whose identity functor admits a cone.
  We have morphisms $f_1 \xto{g_1} f_2 \xot{g_3} f_3$ in this category of elements, so all three lie in the same summand of the coproduct.
  A cone over the identity functor of this summand, with vertex $f_0$ and projections $f_0 \xto{h_i} f_i$, gives the desired data.
\end{proof}

We now have versions of \cref{thm:reedy-fact} and \cref{thm:reedy-zigzag}.

\begin{lem}\label{thm:creedy-fact}
  Every morphism $f:c\to c'$ in an almost c-Reedy category factors as $\rup f \rdn f$, where $\rup f \in \rup \C$ and $\rdn f \in\rdn \C$.
\end{lem}
\begin{proof}
  Just like the proof of \cref{thm:reedy-fact}, except that the stopping condition is when both $\rup{g_n}$ and $\rdn{h_n}$ are level (hence basic level).
  In this case, the desired factorization is $\rdn f = \rdn{h_n}\rup{g_n}\rdn{g_n}$ and $\rup f = \rup{h_n}$, where $\rdn f\in\rdn\C$ by \cref{thm:creedy-compose}\ref{item:crc2}.
\end{proof}

\begin{lem}\label{thm:creedy-zigzag}
  In an almost c-Reedy category, two fundamental factorizations of the same morphism with degrees $\de$ and $\de'$ are connected by a zigzag of fundamental factorizations that passes only through intermediate factorizations of degree $\le\max(\de,\de')$.
\end{lem}
\begin{proof}
  This is similar to the proof of \cref{thm:reedy-zigzag}, but complicated by the presence of nonidentity basic level morphisms (which causes the weakening of the concluding inequality).
  As there, let $\be$ be the maximum degree of intermediate factorizations; we will show that if $\be>\de$ and $\be>\de'$ then we can reduce the number of factorizations of degree $\be$ by one.

  Let $f = h_1 g_1$ be the \emph{first} factorization of degree $\be$ as we traverse the zigzag in one direction.
  Then the factorization preceding it is definitely of degree $<\be$.
  If the factorization following it is also of degree $<\be$, or if its connecting morphism is not basic, then we can proceed as in \cref{thm:reedy-zigzag}.
  Thus, we have two remaining situations to consider.
  \begin{enumerate}
  \item \label{item:crz1} $(h_0,g_0) \xto{k} (h_1,g_1) \xot{\ell} (h_2,g_2)$ where $k$ strictly increases degree and $\ell$ is basic level.
    Then $g_1 = k g_0$ is not basic.
    Since $g_1 = \ell g_2$ as well, by \cref{thm:creedy-compose}\ref{item:crc2}, $g_2$ is not basic either, so we have a fundamental factorization $g_2 = p q$.
    Now we can replace $(h_2,g_2)$ in the zigzag by $(h_2 p,q)$, which is of degree $<\be$, and proceed as in \cref{thm:reedy-zigzag}.
  \item $(h_0,g_0) \xot{k} (h_1,g_1) \xto{\ell} (h_2,g_2)$, where $k$ strictly decreases degree and $\ell$ is basic level.
    Now consider what happens after $(h_2,g_2)$ in the zigzag (it cannot be the end, since it has degree $\be>\max(\de,\de')$): we must have $(h_2,g_2) \xot{p} (h_3,g_3)$.
    If $p$ strictly increases degree, then we can argue as in~\ref{item:crz1}, swapping $(h_1,g_1)$ with $(h_2,g_2)$.
    Thus, $p$ is also level, and we may assume it is basic (otherwise, a fundamental factorization of it could be spliced in).

    Now since $\ell$ and $p$ are basic level, \cref{thm:creedy-compose}\ref{item:crc2} and~\ref{item:crc4} imply that $g_1$, $g_2$, and $g_3$ are either all basic or all not basic.
    If they are not, then we can use a fundamental factorization of $g_1$ to replace $(h_1,g_1)$ by a factorization of smaller degree.
    Thus, we may assume they are all basic.
    But now, by \cref{thm:creedy-cone} applied to $\ell g_1 = g_2 = p g_3$, we can replace the fragment of zigzag $(h_1,g_1) \xto{\ell} (h_2,g_2) \xot{p} (h_3,g_3)$ with one of the form $(h_1,g_1) \ot (h_4,g_4) \to (h_3,g_3)$.
    Composing the connecting morphisms $(h_4,g_4) \to (h_1,g_1) \xto{k} (h_0,g_0)$, we may omit $(h_1,g_1)$, ending up with one fewer factorization of degree $\be$ as desired.\qedhere
  \end{enumerate}
\end{proof}

Let us call a factorization $f = h g$ a \textbf{Reedy factorization} if $g\in\rdn\C$ and $h\in\rup\C$.
We may now expect by analogy with \cref{thm:reedy-fact-uniq} some sort of uniqueness for Reedy factorizations.
Unfortunately, in general this fails.

\begin{eg}
  Consider the commutative square category, with degrees assigned as shown:
  
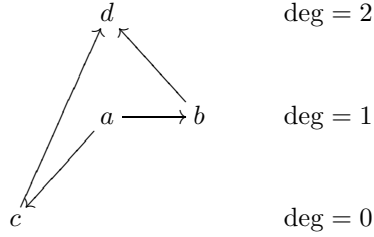
\[ \xymatrix{ & d && \deg = 2\\
    & a \ar[dl] \ar[r] & b \ar[ul] & \deg=1\\
    c \ar[uur] &&& \deg=0 } \]
  This is almost c-Reedy, but the composite $a \to d$ has two Reedy factorizations that are not even related by a zigzag of factorizations.
  (This does not contradict \cref{thm:creedy-char}\ref{item:creedychar2}, since the Reedy factorization $a\to b\to d$ is not fundamental.)
\end{eg}

The best we can do seems to be the following.

\begin{thm}\label{thm:creedy-fact-uniq}
  Let \C be an almost c-Reedy category and suppose it satisfies the additional property that if $f \in\rup\C$ and $g$ is basic level, then $f g \in\rup\C$.
  Then any two Reedy factorizations of the same morphism are connected by a zigzag of Reedy factorizations whose connecting maps are basic level (and thus, in particular, they have the same degree).
\end{thm}
\begin{proof}
  If $f$ is basic, then either $(f,\id)$ or $(\id,f)$ is a Reedy factorization, and any Reedy factorization is connected to this one by a zigzag of length one.
  Thus, suppose $f$ is not basic, and let $(h,g)$ and $(h',g')$ be two Reedy factorizations of it.
  Without loss of generality, let $\deg(h,g) \ge \deg(h',g')$.
  Note that if $h$ or $g$ were level, then by \cref{thm:creedy-compose}\ref{item:crc2} and the assumption on \C, $f$ would be basic.
  Thus, $(h,g)$ is a fundamental factorization, and likewise so is $(h',g')$.

  Therefore, by \cref{thm:creedy-zigzag}, $(h,g)$ and $(h',g')$ are connected by a zigzag of fundamental factorizations all of degree $\le\deg(h,g)$.
  Starting from $(h,g)$, if the next map in the zigzag were $(h,g) \xto{k} (h_1,g_1)$ where $k$ strictly decreases degree, then $(h_1,k)$ would be a fundamental factorization of $h$, which is impossible since $h$ is basic.
  Similarly, if the next map were $(h,g) \xot{k} (h_1,g_1)$ where $k$ strictly increases degree, then $(k,g_1)$ would be a fundamental factorization of $g$.
  Thus, whichever direction $k$ goes in, it must be level.
  Similarly, if $k$ had a fundamental factorization $(\ell,m)$, then $(h_1\ell,m)$ or $(\ell,mg_1)$ would be a fundamental factorization of $h$ or $g$ respectively; thus $k$ is basic.
  Now using \cref{thm:creedy-compose}\ref{item:crc2},~\ref{item:crc4}, and~\ref{item:crc5} and the assumption on \C, we have $g_1\in \rdn\C$ and $h_1\in\rup\C$, so $(h_1,g_1)$ is also Reedy.
  Proceeding inductively in this way along the zigzag, we conclude that all the maps occurring in it are basic level and all the factorizations occurring in it are Reedy.
\end{proof}

Note that if all basic level morphisms are invertible, then the condition in \cref{thm:creedy-fact-uniq} is automatic.
Moreover, a zigzag consisting of isomorphisms can be reduced to a single isomorphism; thus in the g-Reedy case we have a nicer conclusion.

\begin{cor}
  If \C is almost g-Reedy, then the Reedy factorization of every morphism $f$ is unique up to isomorphism.\qed
\end{cor}

Let us now recall the notion of g-Reedy category from~\cite{bm:extn-reedy} (there called a \emph{generalized Reedy category}), which is strictly more general than that of~\cite{cisinski:presheaves}.

\begin{defn}[\cite{bm:extn-reedy}]\label{defn:bm-greedy}
  A \textbf{g-Reedy category} is a category \C equipped with an ordinal degree function on its objects, and subcategories $\rup \C$ and $\rdn \C$ containing all the objects, such that
  \begin{itemize}
  \item Every non-invertible morphism in $\rup\C$ strictly raises degree and every non-invertible morphism in $\rdn\C$ strictly lowers degree.
  \item A morphism is an isomorphism if and only if it lies in both $\rup\C$ and $\rdn \C$, and in this case it is level.
  \item Every morphism $f$ factors uniquely up to isomorphism as $\rup f \rdn f$, where $\rup f \in \rup \C$ and $\rdn f\in\rdn \C$.
  \item If $\theta f = f$ for $f\in\rdn \C$ and $\theta$ an isomorphism, then $\theta=\id$.
  \end{itemize}
\end{defn}

\begin{thm}\label{thm:greedy}
  The following are equivalent for a category \C with an ordinal degree function on its objects.
  \begin{enumerate}
  \item \C is a g-Reedy category.\label{item:greedy1}
  \item \C is a almost g-Reedy category and $\rup\C$ and $\rdn \C$ are closed under composition.\label{item:greedy2}
  \end{enumerate}
\end{thm}
\begin{proof}
  Note that the last clause in \cref{defn:bm-greedy} is just another way of stating \cref{thm:greedy-char}\ref{item:greedychar3}.
  Thus, we have already shown that an almost g-Reedy category has all the properties of a g-Reedy category except for closure of $\rup\C$ and $\rdn \C$ under composition, so~\ref{item:greedy2} implies~\ref{item:greedy1}.
  For the converse, \cref{thm:greedy-char}\ref{item:greedychar1} and~\ref{item:greedychar2} follow essentially as in \cref{thm:reedy}, with isomorphisms replacing identities.
\end{proof}

\begin{eg}\label{eg:orbit}
  A number of examples of g-Reedy categories can be found in~\cite{bm:extn-reedy}; here we recall only one (which we will return to in \cref{sec:enriched-generalized}).
  Let $G$ be a finite discrete group and $\O_G$ its \emph{orbit category}, whose objects are the transitive $G$-sets $G/H$ and whose morphisms are $G$-equivariant maps.
  Then $\O_G$ is g-Reedy in two different ways:
  \begin{enumerate}
  \item with $\deg(G/H) = |H|$, where $\rup{\O_G} = \O_G$ and $\rdn{\O_G}$ contains only the isomorphisms;
  \item with $\deg(G/H) = |G/H| = [G:H]$, where $\rdn{\O_G} = \O_G$ and $\rup{\O_G}$ contains only the isomorphisms.
    (This requires checking the freeness condition for the action of automorphisms on $\rdn{\O_G}$.)
  \end{enumerate}
  Its opposite category $\O_G\op$ is g-Reedy in only one way, with $\deg(G/H) = |G/H|$ and $\rup{\O_G\op} = \O_G\op$; the other putative structure would fail the freeness condition.

  Now in equivariant homotopy theory, we are often interested in categories of the form $\M^{\O_G\op}$.
  Thus, if $\M$ is a model category admitting both projective and injective model structures, there are \emph{three} ways to put a g-Reedy model structure on $\M^{\O_G\op}$: we can write $\M^{\O_G\op} \cong ((\M\op)^{\O_G})\op$ and apply either of the above g-Reedy structures on $\O_G$, or we can directly apply the g-Reedy structure of $\O_G\op$.
  Two of these model structures end up coinciding with the injective and projective model structures on $\M^{\O_G\op}$ itself, so that the g-Reedy machinery gives a more explicit description of the projective cofibrations and injective fibrations.
  In the other model structure, the fibrations are the $\aut(x)$-injective-fibrations for each object $x$ separately, while the cofibrations look like traditional s-Reedy cofibrations.

  If $G$ is infinite or topological, $\O_G$ will not generally be g-Reedy (although the above g-Reedy structures work if we restrict to subgroups of finite cardinality or finite index, respectively).
  But as noted in \cref{rmk:transfinite-degrees}, if $G$ is \emph{compact Lie}, and we restrict $\O_G$ (as is usual in the topological case) by requiring the subgroups $H$ to be closed (and the maps continuous), then $\O_G$ does have a g-Reedy structure with $\deg(G/H)$ being the ordinal $\omega\cdot \dim(H)+|\pi_0(H)|-1$, where $\rup{\O_G} = \O_G$ and $\rdn{\O_G}$ contains only the isomorphisms.
  The other two structures do not work in this case, since the opposite ordering of (even closed) subgroups may not be well-founded.
  Thus, if $\M$ admits projective model structures, $\M^{\O_G\op}$ admits a g-Reedy model structure, which ends up coinciding with its injective model structure, so we get a more explicit description of the injective fibrations.

  However, this is not as useful as it might be, since when $G$ is topological (including compact Lie), one generally wants to consider $\O_G$ as a \emph{topologically enriched} category.
  We will return to this question in \cref{sec:enriched-generalized}.
\end{eg}

We can define a notion of \emph{c-Reedy category} analogous to \cref{defn:bm-greedy} without the restriction to groupoidal strata, but it is a bit more difficult to state.

\begin{defn}\label{defn:creedy}
  A \textbf{c-Reedy category} is a category \C equipped with an ordinal degree function on its objects, and subcategories $\Cbar$, $\rup \C$, and $\rdn \C$ containing all the objects, such that
  \begin{enumerate}
  \item $\Cbar \subseteq \rup\C\cap \rdn\C$.\label{item:crd1}
  \item Every morphism in $\Cbar$ is level.\label{item:crd2}
  \item Every morphism in $\rup\C\setminus\Cbar$ strictly raises degree, and every morphism in $\rdn\C\setminus\Cbar$ strictly lowers degree.\label{item:crd3}
  \item Every morphism $f$ factors as $\rup f \rdn f$, where $\rup f \in \rup \C$ and $\rdn f\in\rdn \C$, and the category of such factorizations with connecting maps in \Cbar is connected.\label{item:crd4}
  \item For any $x$ and any degree $\de<\deg(x)$, the functor $\rdn\C(x,\blank):\Cbar_{=\de} \to \nSet$ is a coproduct of retracts of representables.\label{item:crd5}
  \end{enumerate}
\end{defn}

Note that~\ref{item:crd3} implies the other containment in~\ref{item:crd1}: if $f\in (\rup\C\cap\rdn\C)\setminus \Cbar$, it would have to both strictly raise and strictly lower degree, which is impossible.
Thus, in fact $\Cbar = \rup\C\cap \rdn\C$.

\begin{thm}\label{thm:creedy}
  The following are equivalent for a category \C with an ordinal degree function on its objects.
  \begin{enumerate}
  \item \C is a c-Reedy category.\label{item:creedy1}
  \item \C is a almost c-Reedy category and $\rup\C$ and $\rdn \C$ are closed under composition.\label{item:creedy2}
  \end{enumerate}
\end{thm}
\begin{proof}
  Note that closure of $\rup\C$ under composition implies the extra hypothesis of \cref{thm:creedy-fact-uniq}.
  Thus, our preceding lemmas show that~\ref{item:creedy2} implies~\ref{item:creedy1}.

  For the converse, we again argue roughly as in \cref{thm:reedy}, but we must take more care to deal with non-uniqueness of factorizations.
  Let us call a factorization $(\rup f,\rdn f)$ in a c-Reedy category with $\rup f \in \rup \C$ and $\rdn f\in\rdn\C$ a \emph{Reedy factorization}.
  Now the argument $f = h g = \rup h \rdn h \rup g \rdn g = \rup h \rup k \rdn k \rdn g$ shows that any factorization $(g,h)$ of $f$ is connected by a zigzag to a Reedy factorization $(\rup h \rup k,\rdn k \rdn g)$ of no greater degree.  
  Therefore, $f$ has a fundamental factorization iff it has a fundamental Reedy factorization.

  Now, if $f\in\rup\C$, then $(f,\id)$ is a Reedy factorization having the same degree as the domain of $f$.
  But by \cref{defn:creedy}\ref{item:crd4}, any two Reedy factorizations have the same degree; thus $f$ has no fundamental Reedy factorizations, hence no fundamental factorizations at all, and so it is basic.
  Similarly, every morphism in $\rup\C$ is basic.

  Conversely, however, if $f$ is basic, then a Reedy factorization $f = \rup f \rdn f$ cannot be fundamental, so either $\rdn f$ or $\rup f$ must be level.
  If $\rup f$ is level, then by \cref{defn:creedy}\ref{item:crd3} it is also in $\rdn f$, so $f\in\rdn C$, and similarly.
  Thus, the basic morphisms are precisely $\rup \C \cup \rdn \C$, and so once again these subcategories given as data agree with those that we have defined.
  It also follows that $\Cbar = \rup \C \cap \rdn \C$ is exactly the class of basic level morphisms.

  With that out of the way, we can essentially repeat the argument of \cref{thm:reedy}.
 \cref{thm:creedy-char}\ref{item:creedychar0} and~\ref{item:creedychar1} follow because \Cbar is a subcategory containing all the objects.
  For \cref{thm:creedy-char}\ref{item:creedychar2}, we connect any two fundamental factorizations to Reedy factorizations, and then use \cref{defn:creedy}\ref{item:crd4} to connect those Reedy factorizations.
  Finally, \cref{thm:creedy-char}\ref{item:creedychar3} is immediate since $\Cbar\subseteq \rdn\C$ and $\rdn\C$ is a subcategory, while \cref{thm:creedy-char}\ref{item:creedychar4} is identical to \cref{defn:creedy}\ref{item:crd5}.
\end{proof}

\begin{rmk}
  As mentioned in the introduction, \cref{thm:creedy} has been formally verified in the computer proof assistant Coq; see \cref{sec:formalization}.
\end{rmk}

\begin{rmk}
  In the terminology of~\cite[Definition 4.10]{cheng:dl-lawvere}, the subcategories $\rdn\C$ and $\rup\C$ of a c-Reedy category form a \emph{factorization system over \Cbar}.
  The special cases when \Cbar consists of isomorphisms or identities, as in a g-Reedy or s-Reedy category, correspond to ordinary orthogonal factorization systems and to ``strict factorization systems'', respectively.
\end{rmk}

I do not know any interesting examples of c-Reedy categories that are not g-Reedy.
However, when we come to the enriched context in \cref{sec:enriched-generalized}, this extra generality will be useful, as it enables us to avoid addressing the question of what an ``enriched groupoid'' is.

\section{Enriched Reedy categories}
\label{sec:enriched}

The yoga of Reedy-ness we have presented works in any sort of ``category theory'' where we have profunctors and collages.
In this section we show this by example in the case of enriched categories, which is probably the most familiar and most useful case, and also permits a comparison with the existing theory of~\cite{angeltveit:enr-reedy}.
We will omit many proofs, which are straightforward enrichments of those in \cref{sec:reedy}.

To start with, let \V be a complete and cocomplete closed symmetric monoidal category, so that we can consider \V-enriched categories, functors, profunctors, and so on.
The following definition looks exactly the same as \cref{defn:abd-cd}, except that the weights and hom-objects take values in \V rather than \nSet.

\begin{defn}\label{defn:v-abd-cd}
  Given small \V-categories \C and \D, \textbf{abstract bigluing data from \C to \D} consists of \V-profunctors $U:\D\hto \C$ and $W:\C\hto\D$ together with a \V-natural transformation
  \[ \al: \wcolim{W}{U}{\D} \to \C(\blank,\blank) \]
  The \textbf{collage} of such data is a \V-category with objects being the disjoint union of the objects of \C and \D, with
  \begin{align*}
    \coll{\al}(c,c') &= \C(c,c')\\
    \coll{\al}(c,d) &= U(c,d)\\
    \coll{\al}(d,c) &= W(d,c)\\
    \coll{\al}(d,d') &= \D(d,d') \sqcup (\wcolimc{U}{W})(d,d').
  \end{align*}
\end{defn}

If \M is a complete and cocomplete \V-category, we have weighted limit and colimit functors $(\wcolimc{U}{\blank}):\M^\C \to\M^\D$ and $\wlimc{W}{\blank}:\M^\C \to\M^\D$, and $\al$ induces as before a transformation $\al:(\wcolimc U \blank) \to \wlimc W\blank$.

\begin{thm}\label{thm:enriched-bicollage}
  In the above situation, we have an equivalence of \V-categories $\M^{\coll{\al}} \simeq \biglue{\al}$.\qed
\end{thm}

Note additionally that by \cref{thm:bistrat-parapostcomp}, the tensor and cotensor of $\biglue{\al}$ over \V can be constructed inductively using \cref{thm:cell-paramor}.

We define \textbf{bistratified \V-categories} exactly as in \cref{defn:bistratified}.
There is an extra wrinkle in the characterization of these, due to the fact that not every subobject in \V may have a complement.

\begin{defn}\label{defn:v-fundfact}
  Let \C be a \V-category whose objects are assigned ordinal degrees, let $x$ and $y$ be objects, let $\de$ be the lesser of their two degrees, and let $\C_\de$ as usual be the full subcategory on objects of degree $<\de$.
  We write $\partial_\de\C(x,y)$ for the coequalizer
  \[ \ncoeq\left( \coprod_{z,w\in\C_\de} \C(w,y) \otimes \C(z,w) \otimes \C(x,z) \toto
    \coprod_{z\in \C_\de} \C(z,y) \otimes \C(x,z)  \right).
  \]
  It comes with a canonical map
  \begin{equation}
    \partial_\de\C(x,y) \to \C(x,y)\label{eq:v-fundfact}
  \end{equation}
\end{defn}

Analogously to \cref{thm:bistrat-char}, we have

\begin{thm}\label{thm:v-bistrat-char}
  A \V-category \C whose objects are assigned ordinal degrees is bistratified if and only if the following conditions hold.
  \begin{enumerate}
  \item If $x$ and $y$ have the same degree, then the map~\eqref{eq:v-fundfact} is the coprojection of a coproduct, so that $\C(x,y) \cong \partial_\de\C(x,y) \sqcup \Cbar(x,y)$ for some object $\Cbar(x,y)$, which we call the \textbf{object of basic level morphisms}.\label{item:vbistrat1}
  \item The identity map $I \to \C(x,x)$ factors through $\Cbar(x,x)$.\label{item:vbistrat2}
  \item If $x$, $y$, and $z$ all have the same degree, then the composition map $\C(y,z) \otimes \C(x,y) \to \C(x,z)$ maps $\Cbar(y,z) \otimes \Cbar(x,y)$ into $\Cbar(x,z)$.\label{item:vbistrat3}
  \item The induced structure is associative and unital, making \Cbar a \V-category and the inclusion $\Cbar \to \C$ a \V-functor.\qed\label{item:vbistrat4}
  \end{enumerate}
\end{thm}

If coprojections of coproducts in \V are monic, as is often the case\footnote{In particular, this is the case if \V is cartesian monoidal or has zero morphisms.  However, there do exist counterexamples; I learned the following simple one from Michal Przybylek.  Let $\V = \nSet\times \nSet\op$ with the closed symmetric monoidal structure $(A,B) \otimes (C,D) = (A\times C, B^C \times D^A)$; this is a special case of the ``Chu construction''~\cite{chu:construction}.  Here the coproduct is $(A,B)\sqcup(C,D) = (A\sqcup C, B\times D)$, and the coprojection $(\emptyset,\ast) \to (\emptyset,\ast) \sqcup (\ast,\emptyset) \cong (\ast, \emptyset)$ is not monic.\label{fn:chu}}, then condition~\ref{item:vbistrat4} is unnecessary.
Otherwise, conditions~\ref{item:vbistrat2} and~\ref{item:vbistrat3} must be viewed as the \emph{choice} of a lift rather than its mere existence.
If complements are unique (up to unique isomorphism) when they exist, as is sometimes the case, then~\ref{item:vbistrat1} can also be viewed as a mere condition, but in general it too must be regarded as a choice of data.

We say that a bistratified \V-category \C has \textbf{discrete strata} if the \V-category \Cbar is discrete, i.e.\ we have
\[ \Cbar(x,y) =
\begin{cases}
  I &\qquad x=y\\
  \emptyset &\qquad x\neq y.
\end{cases}
\]
We will write $\C_{=\de}$ for the full subcategory of $\Cbar$ on the objects of degree $\de$, and call it the $\de^{\mathrm{th}}$ \textbf{stratum}.
As usual, we emphasize that $\C_{=\de}$ is a \emph{non-full} subcategory of \C (where ``subcategory'' is meant in the sense that its hom-objects are summands of those of \C; as before, this doesn't necessarily imply their ``inclusions'' are monic).

We can define enriched matching and latching object functors as 
\begin{align}
  M_x A &= \wlimc{\partial_\de\C(x,\blank)}{A}\\
  L_x A &= \wcolimc{\partial_\de\C(\blank,x)}{A}
\end{align}
where now $\partial_\de\C(\blank,\blank)$ denotes the above \V-enriched weight.

\begin{thm}
  For any bistratified \V-category \C with discrete strata, and any complete and cocomplete \V-category \M with a wfs (or pre-wfs), there is an induced \textbf{Reedy wfs} (or pre-wfs) on $\M^\C$ that is characterized by:
  \begin{itemize}
  \item $A\to B$ is an \R-map iff $A_x \to M_x A \times_{M_x B} B_x$ is an \R-map in \M for all $x\in \C$.
  \item $A\to B$ is an \L-map iff $L_x B \sqcup_{L_x A} A_x \to B_x$ is an \L-map in \M for all $x\in \C$.\qed
  \end{itemize}
\end{thm}

Now suppose \V is a symmetric monoidal model category; we need a version of \cref{thm:reedy-wfs}.
If \M is a \V-category with a wfs (or pre-wfs), we will say it is a \textbf{\V-wfs} (or \textbf{\V-pre-wfs}) if the hom-functor $\M\op\times\M\to\V$ is a right wfs-bimorphism for the (cofibration, acyclic fibration) wfs on \V.
As usual, this can equivalently be expressed in terms of the tensor or cotensor.
Both wfs of a \V-model category are \V-wfs (but this statement is only $\frac{2}{3}$ of being a \V-model category).

\begin{thm}\label{thm:vreedy-wfs}
  Let \C be a bistratified \V-category with discrete strata and let \M be a complete and cocomplete \V-category with a \V-wfs.
  Let $(\V^{\C})$ and $\V^{\C\op}$ have their Reedy wfs induced by the (cofibration, acyclic fibration) wfs on \V.
  Then:
  \begin{enumerate}
  \item The weighted limit $(\V^{\C})\op \times \M^\C \to \M$ is a right wfs-bimorphism.
  \item The weighted colimit $\V^{\C\op} \times \M^\C \to \M$ is a left wfs-bimorphism.
  \end{enumerate}
\end{thm}
\begin{proof}
  Just like \cref{thm:reedy-wfs}, starting from the assumption that \M has a \V-wfs.
\end{proof}

There is not much hope of strengthening this to an enriched analogue of \cref{thm:almost-reedy}, but we can at least use it to give a sufficient condition for the existence of \V-Reedy model structures.

\begin{defn}
  An \textbf{almost-Reedy \V-category} is a bistratified \V-category with discrete strata such that for any object $x$ of degree $\de$, the hom-functors $\C(x,\blank) \in \V^{\C_\de}$ and $\C(\blank,x) \in \V^{\C_\de\op}$ are \L-objects in the Reedy wfs.
\end{defn}

\begin{thm}\label{thm:vreedy-model}
  If \C is an almost-Reedy \V-category and \M is a \V-model category, then the \V-functor category $\M^\C$ has a model structure characterized by:
  \begin{itemize}
  \item The weak equivalences are objectwise,
  \item $A\to B$ is a fibration iff $A_x \to M_x A \times_{M_x B} B_x$ is a fibration in \M for all $x\in \C$.
  \item $A\to B$ is a cofibration iff $L_x B \sqcup_{L_x A} A_x \to B_x$ is a cofibration in \M for all $x\in \C$.\qed
  \end{itemize}
\end{thm}

More explicitly, for \C to be almost-Reedy we require the canonical maps
\begin{equation}
  \partial_\de\C(x,y) \too \C(x,y)\label{eq:vcofmap}
\end{equation}
to be cofibrations whenever $x$ and $y$ have different degrees (and, as before, $\de$ is the lesser of their degrees).
Note that unlike in the unenriched case, where the cofibrancy condition on non-level morphisms had the same form as the bistratification condition on level morphisms (both required these maps to be injections), in the enriched situation the two generalize to different notions of ``injection'': the first must be the coprojection of a coproduct, while the second must be a cofibration.
However, if we want a notion of ``Reedy \V-category'' that looks more like the unenriched version, we can notice that there is a common refinement of these two kinds of ``injection'': a coprojection is a cofibration if its complement is cofibrant.

There seems little hope of enriching Lemmas~\ref{thm:reedy-fact}--\ref{thm:reedy-fact-uniq}, but we can jump right to a definition of Reedy category in terms of unique factorizations.

\begin{defn}\label{defn:vreedy}
  A \textbf{Reedy \V-category} is a small \V-category \C with ordinal degrees assigned to its objects, along with:
  \begin{enumerate}
  \item \V-categories $\rup \C$ and $\rdn\C$, with the same objects as \C, and whose hom-objects are cofibrant.\label{item:vreedy1}
  \item \V-functors $\rup \C \to \C$ and $\rdn\C \to\C$ that are the identity on objects.\label{item:vreedy2}
  \item The unit maps $I\to \rup\C(x,x)$ and $I\to \rdn\C(x,x)$ are isomorphisms.\label{item:vreedy4}
  \item $\rup\C(x,y) = \emptyset$ if $\deg(x)\geq\deg(y)$ and $x\neq y$, and similarly $\rdn\C(x,y) = \emptyset$ if $\deg(x)\leq\deg(y)$ and $x\neq y$.\label{item:vreedy3}
  \item For any $x,y$, the map\label{item:vreedy5}
    \[ \coprod_{z} \rup\C(z,y) \otimes \rdn\C(x,z) \to \C(x,y) \]
    is an isomorphism.
  \end{enumerate}
\end{defn}

Note that conditions~\ref{item:vreedy1} and~\ref{item:vreedy4} imply in particular that the unit object $I$ of \V must be cofibrant.
Condition~\ref{item:vreedy5} says in an enriched way that ``every morphism $f:x\to y$ factors uniquely as $f = \rup f \rdn f$ where $\rup f \in\rup\C$ and $\rdn f \in\rdn\C$''.
By~\ref{item:vreedy3}, the coproduct in~\ref{item:vreedy5} may as well run only over $z$ with $\deg(z)\leq \min(\deg(x),\deg(y))$.
There is also at most one nontrivial summand where $\deg(z)= \min(\deg(x),\deg(y))$:
\begin{itemize}
\item If $x=y$, then when $z=x$ we have the summand $I\otimes I \cong I$, and all other summands with $\deg(z) = \deg(x)$ are $\emptyset$.
\item If $x\neq y$ and $\deg(x)=\deg(y)$, then all summands with $\deg(z) = \deg(x)$ are $\emptyset$.
\item If $\deg(x) <\deg(y)$, then when $z=x$ we have the summand $\rup\C(x,y) \otimes I \cong \rup\C(x,y)$, and all other summands with $\deg(z) = \deg(x)$ are $\emptyset$.
\item If $\deg(x) > \deg(y)$, then when $z=y$ we have the summand $I\otimes \rdn\C(x,y) \cong \rdn\C(x,y)$, and all other summands with $\deg(z) = \deg(y)$ are $\emptyset$.
\end{itemize}
We regard this unique summand, whatever it is, as the \emph{basic hom-object} from $x$ to $y$, and its complement (the coproduct of all the summands with $\deg(z) < \min(\deg(x),\deg(y))$) as the \emph{non-basic hom-object} from $x$ to $y$.
(We will show in a moment that they deserve these names.)
By assumption, the basic hom-object is always cofibrant.
It also follows from these considerations that $\rup\C$ and $\rdn\C$ are ``subcategories'' of $\C$ in the sense that  the maps $\rup\C(x,y) \to \C(x,y)$ and $\rdn\C(x,y) \to \C(x,y)$ are coprojections of a coproduct (though depending on \V, this may not imply that they are monomorphisms).

\begin{thm}\label{thm:v-reedy}
  A Reedy \V-category is an almost-Reedy \V-category.
\end{thm}
\begin{proof}
  Let $\min(\deg(x),\deg(y))=\de$; it will suffice to show that the map $\partial_\de \C(x,y) \to \C(x,y)$ is an isomorphism onto the non-basic hom-object.
  When $\deg(x)=\deg(y)$, by \cref{thm:v-bistrat-char} this will show that \C is bistratified with discrete strata, while when $\deg(x)\neq\deg(y)$ it will show that the maps~\eqref{eq:vcofmap} are cofibrations (since the basic hom-object is cofibrant).

  Reinterpreted, what we want to show is that we have a coequalizer diagram:
  \begin{equation}
    \resizebox{\textwidth-1.5cm}{!}{$\displaystyle\coprod_{z,w\in\C_\de} \C(w,y) \otimes \C(z,w) \otimes \C(x,z) \toto
    \coprod_{z\in \C_\de} \C(z,y) \otimes \C(x,z) \too
    \coprod_{w\in \C_\de} \rup\C(w,y) \otimes \rdn\C(x,w)$}.\label{eq:vreedy-coeq}
  \end{equation}
  We will do this essentially by ``enriching'' the proof of \cref{thm:reedy}, including~\cite[Lemma 2.9]{rv:reedy}.

  We begin by observing that the following diagram commutes for any $z\in\C_{\de}$, where $u,v,w$ in the coproducts range over all objects of degree $<\de$.
  \begingroup
  \def\C{\mathcal{C}}
  \begin{equation}
    \vcenter{\xymatrix{
      \coprod_{u,v} \rup\C(v,y) \otimes \rdn\C(z,v) \otimes \rup\C(u,z) \otimes \rdn\C(x,u) \ar[d] \ar[r]^-\cong &
      \C(z,y) \otimes \C(x,z) \ar[ddd] \\
      \coprod_{u,v} \rup\C(v,y) \otimes \C(u,v) \otimes \rdn\C(x,u) \ar@/^2pc/[ddr]  & \\
      \coprod_{u,v,w} \rup\C(v,y) \otimes \rup\C(w,v) \otimes \rdn\C(u,w) \otimes \rdn\C(x,u) \ar[d]\ar[u]^\cong \\
      \coprod_{w} \rup\C(w,y) \otimes \rdn\C(x,w) \ar[r] &
      \C(x,y)
    }}\label{eq:vreedy-compose}
  \end{equation}
  \endgroup
  The two internal quadrilaterals commute by associativity of composition in $\C$.
  Therefore, inverting the two isomorphisms, we see that the composition map $\C(z,y) \otimes \C(x,z) \to \C(x,y)$ on the right is equal to the composite around the left-hand side.

  (The bottom map in~\eqref{eq:vreedy-compose} is not an isomorphism because we omit the summands in its domain with $\deg(w)=\de$.
  But the two displayed isomorphisms are so, because since $\deg(z)<\de$ we only need to factor morphisms from or to it through objects also of degree $<\de$, and similarly for $u$ and $v$.)

  Now the composite around the top and left of~\eqref{eq:vreedy-compose} yields the $z$-component of the final map in~\eqref{eq:vreedy-coeq}, the one we intend to be the coequalizer of the other two.
  Note that it is a factorization of the composition map $\coprod_{z\in \C_\de} \C(z,y) \otimes \C(x,z) \to \C(x,z)$ through the ``non-basic hom-object'', which we have defined to be the complement of the hom-object in $\rup\C$ or $\rdn\C$ (as appropriate by degree).
  Thus, the existence and epimorphy of this map shows that the non-basic hom-object deserves its name, and is an enrichment of the statement ``$\rup\C\cup\rdn\C$ is precisely the class of basic morphisms'' from \cref{thm:reedy}.

  It remains to show that~\eqref{eq:vreedy-coeq} is a coequalizer diagram.
  In fact, we will show that it is an \emph{absolute} coequalizer diagram.
  Recall that a \emph{split} coequalizer is a diagram like
  \[ \xymatrix{ A \ar@<1mm>[r]^p \ar@<-1mm>[r]_q & B \ar[r]^e \ar@/^2pc/[l]^t & C \ar@/^2pc/[l]^{s} } \]
  such that $ep=eq$, $es = \id$, $se=qt$, and $pt = \id$.
  This automatically makes $e$ a coequalizer of $p$ and $q$, which is moreover preserved by any functor.
  An \emph{absolute} coequalizer generalizes this by replacing the single morphism $t$ with a finite list of morphisms relating $se$ to $\id$ by a finite chain of ``zigzags''.
  See~\cite{pare:abscoeq} for the general theory; we will only need a 2-step version which looks like this:
  \[ \xymatrix{ A \ar@<1mm>[r]^p \ar@<-1mm>[r]_q & B \ar[r]^e \ar@/^1.5pc/[l]^{t_1} \ar@/^2.5pc/[l]^{t_2} & C \ar@/^2pc/[l]^{s} } \]
  such that $ep=eq$, $es=\id$, $se=qt_1$, $pt_1 = pt_2$, and $q t_2 = \id$.
  The reader can easily check that this suffices to make $e$ a coequalizer of $p$ and $q$.

  Let $p$ and $q$ be the parallel morphisms in~\eqref{eq:vreedy-coeq}, where $p$ composes the right two hom-objects and $q$ composes the left two.
  We will define $s$, $t_1$, and $t_2$ essentially by ``enriching'' the zigzag~\eqref{eq:reedyfact-zigzag}.
  
  We define $s$ by simply setting $z=w$ and including $\rdn\C(x,z)$ into $\C(x,z)$ and $\rup\C(z,y)$ into $\C(z,y)$.
  The fact that $es=\id$ follows from the construction of $e$ above; each component of $s$ factors through the upper-left corner of~\eqref{eq:vreedy-compose} where $u=v=z$.

  We define $t_1$ to be the following composite.
  \begin{equation}
    \C\C \otiso \rup\C\rdn\C\rup\C\rdn\C \to \rup\C \C \rdn \C \otiso \rup\C \rup\C \rdn\C \rdn\C \to \rup\C \rup\C \rdn\C \to \C\C\C.\label{eq:t1}
  \end{equation}
  Here and from now on, we use a schematic notation for objects of \V such as those appearing in~\eqref{eq:vreedy-compose}, in which we omit the names of the objects of \C and the symbols $\coprod$ and $\otimes$.
  It is to be understood that the coproducts range over all ``intermediate'' objects having degree $<\de$.
  If $\V=\nSet$, then $t_1$ takes a pair of morphisms $(h,g)$ (regarded as a factorization of $f = hg$) to the triple $(\rup h, \rup k, \rdn k\rdn g)$, where $k=\rdn h \rup g$ as in \cref{thm:reedy}.

  The first three morphisms in~\eqref{eq:t1} are part of $e$, proceeding not quite all the way around the upper-left composite of~\eqref{eq:vreedy-compose}.
  The fourth just composes in $\rdn\C$, and the last includes $\rup \C$ and $\rdn\C$ into $\C$.
  The definitions of $t_1$ and $e$ make it clear that $se = qt_1$.
  When $\V=\nSet$, this equality means that if $f=hg$, then $(\rup f, \rdn f) = (\rup h \rup k, \rdn k \rdn g)$.

  We define the map $t_2$ more simply, as
  \[ \C\C \otiso \rup\C \rdn \C \C \to \C\C\C. \]
  If $\V=\nSet$, then $t_2$ takes $(h,g)$ to $(\rup h, \rdn h, g)$.
  Evidently $q t_2 = \id$, while the following diagram (in which the long top-left-bottom composite $\C\C\to\C\C\C$ is $t_1$) shows that $pt_1 = pt_2$ (which in the case $\V=\nSet$ means $(\rup h, \rup k \rdn k \rdn g) = (\rup h, \rdn h g)$).
  \def\C{\mathcal{C}}
  \begin{equation}
    \vcenter{\xymatrix@R=1.5pc{
        \rup\C \rdn\C \rup\C \rdn\C \ar[d] \ar[r]^-\cong &
        \rup\C \rdn\C \C \ar[rr]^-\cong \ar[d] \ar[dr] &&
        \C\C \ar[d]^{t_2} \\
      \rup\C \C \rdn\C \ar[r] &
      \rup\C \C \ar[r] &
      \C\C &
      \C\C\C \ar[l]_p \\
      \rup\C \rup\C \rdn\C \rdn\C \ar[u]^\cong \ar[r] &
      \rup\C \rup\C \rdn\C \ar[u] \ar[r] &
      \C\C\C\ar[u]_p
      }}
  \end{equation}

  One thing remains: we have to show that the two composites in~\eqref{eq:vreedy-coeq} are actually equal, i.e.\ that $e p = e q$!
  Note that this is automatic if coproduct injections in \V are monic; since this is the case for \nSet, the remainder of the proof has no analogue in the proof of \cref{thm:reedy}.
  We begin by extending~\eqref{eq:vreedy-compose} as shown in \cref{fig:vreedy-fork}.
  \begin{figure}
    \centering
    \begingroup
    \def\C{\mathcal{C}}
    \begin{equation}
      \label{eq:vreedy-fork}
      \vcenter{\xymatrix@R=1.5pc{
          &\rup\C \rdn\C \rup\C \rdn\C \rup\C \rdn\C \ar@<5mm>[d] \ar@<-5mm>[d] \ar[rr]^-\cong &&
          \C\C\C \ar@<2mm>[ddd]^p \ar@<-2mm>[ddd]_q \\
          & \rup\C \C \rdn\C \rup\C \rdn\C \quad \rup\C \rdn\C \rup\C \C \rdn\C & \\
          & \rup\C \rup\C \rdn\C \rdn\C \rup\C \rdn\C \quad
          \rup\C \rdn\C \rup\C \rup\C \rdn\C \rdn\C
          \ar@<5mm>[u]^\cong \ar@<-5mm>[u]_\cong
          \ar@<5mm>[d] \ar@<-5mm>[d] & \\
          &\rup\C \rdn\C \rup\C \rdn\C \ar[d] \ar[rr]^-\cong &&
          \C\C \ar[ddd] \ar[dddll]_e \\
          &\rup\C \C \rdn\C \\
          &\rup\C \rup\C \rdn\C \rdn\C \ar[u]^\cong \ar[d] \\
          &\rup\C \rdn\C \ar[rr] &&
          \C
        }}
    \end{equation}
    \endgroup
    \caption{}\label{fig:vreedy-fork}
  \end{figure}
  The lower rectangle is just~\eqref{eq:vreedy-compose} in schematic notation.
  At the top, we have two rectangles with vertices $\rup\C \rdn\C \rup\C \rdn\C \rup\C \rdn\C$, $\rup\C \rdn\C \rup\C \rdn\C$, $\C\C\C$, and $\C\C$, in which the vertical maps either both take the left path or both the right.
  Each of these rectangles commutes by the same argument as in~\eqref{eq:vreedy-compose}.
  Thus, to show that $e p = e q$, it will suffice to show that the two composites $\rup\C \rdn\C \rup\C \rdn\C \rup\C \rdn\C \to \rup\C\rdn\C$ along the left-hand side are equal.
  But these are the outer composites in \autoref{fig:vreedyisafork} on page~\pageref{fig:vreedyisafork}, in which all the inner quadrilaterals commute.
  \begin{figure}
    \centering
    \def\C{\mathcal{C}}
    \begin{equation}
    \xymatrix@R=1.5pc@C=1.5pc{
      & &\rup\C \rdn\C \rup\C \rdn\C \rup\C \rdn\C \ar[dr] \ar[dl] \\ &
      \rup\C \C \rdn\C \rup\C \rdn\C \ar[dr] && \rup\C \rdn\C \rup\C \C \rdn\C \ar[dl] \\
      \rup\C \rdn\C \rup\C \rdn\C \ar[dd] 
      & \rup\C \rup\C \rdn\C \rdn\C \rup\C \rdn\C \ar[u]^\cong \ar[d] \ar[l] &
      \rup\C\C\C\rdn\C  &
      \rup\C \rdn\C \rup\C \rup\C \rdn\C \rdn\C \ar[u]_\cong \ar[d] \ar[r] &
      \rup\C \rdn\C \rup\C \rdn\C \ar[dd] \\ &
      \rup\C \rup\C \rdn\C  \C \rdn\C  \ar[ur]^\cong \ar[d]
      & & \rup\C \C \rup\C \rdn\C \rdn\C \ar[ul]_\cong \ar[d] \\
      \rup\C \C \rdn \C
      & \rup\C \rup\C \C \rdn\C \ar[l] 
      & \rup\C \rup\C\rdn\C \rup\C \rdn\C \rdn \C \ar[ul]_\cong \ar[ur]^\cong \ar[d]
      & \rup\C \C \rdn\C \rdn\C \ar[r]
      & \rup\C \C \rdn\C \\
      \rup\C \rup\C \rdn\C \rdn\C \ar[u]^\cong \ar[ddrr]
      & \rup\C \rup\C \rup\C \rdn\C \rdn\C \ar[u]^\cong \ar[l]
      & \rup\C \rup\C \C \rdn\C \rdn\C \ar[ul] \ar[ur]
      & \rup\C \rup\C \rdn\C \rdn\C \rdn\C \ar[u]_\cong  \ar[r] &
      \rup\C \rup\C \rdn\C \rdn\C \ar[u]_\cong \ar[ddll] \\ &
      & \rup\C \rup\C \rup\C \rdn\C \rdn\C \rdn\C \ar[u]^\cong \ar[ur] \ar[ul] \ar[d] \\ &
      & \rup\C \rdn\C 
      }
  \end{equation}
    \caption{}
    \label{fig:vreedyisafork}
  \end{figure}
\end{proof}

The definition of Reedy \V-category in~\cite{angeltveit:enr-reedy} (as corrected in the arXiv v2) is a bit more restrictive.
In more categorical language, it consists of:
\begin{enumerate}
\item An unenriched Reedy category \B.
\item A normal lax functor from \B, regarded as a bicategory with only identity 2-cells, to \V, regarded as a bicategory with one object.  This consists of:
  \begin{enumerate}
  \item For each morphism $f\in\B(x,y)$, an object $\C(x,y)_f \in \V$,
  \item Composition maps $\C(y,z)_g \otimes \C(x,y)_f \to \C(x,z)_{g f}$,
  \item An identity isomorphism $I\toiso \C(x,x)_{\id}$, and
  \item Appropriate associativity and unitality axioms.
  \end{enumerate}
\item For any $g\in\B(x,y)$, the composition map $\C(z,y)_{\rup g} \otimes \C(x,z)_{\rdn g} \to \C(x,y)_g$ is an isomorphism.
\item For any $g\in\B(x,y)$, the objects $\coprod_{g\in \rup \B} \C(x,y)_g$ and $\coprod_{g\in \rdn \B} \C(x,y)_g$ are cofibrant in \V.
\end{enumerate}
We can restrict this lax functor along the inclusions $\rup\B\into\B$ and $\rdn\B\into\B$, obtaining two more lax functors.
Moreover, the bicategory-with-one-object \V can be embedded into the bicategory \vprof of \V-categories and \V-functors, by sending the unique object to the unit \V-category.
This gives us three lax functors into \vprof.
It is straightforward to verify that their collages $\C$, $\rup\C$, and $\rdn \C$ (meaning an enriched generalization of \cref{defn:gen-collage}) form a Reedy \V-category in the sense of \cref{defn:vreedy}.

\section{Enriched generalized Reedy categories}
\label{sec:enriched-generalized}

As an example of the sort of new theory enabled by this approach to Reedy-ness, let us now combine the enriched Reedy categories of \cref{sec:enriched} with the generalized ones of \cref{sec:auto}.
With the general theory of c-Reedy categories under our belt, there is no reason to restrict attention to those with groupoidal strata;
this is convenient because it means we can avoid deciding what to mean by a ``\V-groupoid''.

We will require projective model structures on \V-diagram categories.
It is easy to see that if \M has a cofibrantly generated \V-wfs, then $\M^\D$ has a projective wfs for any small \V-category \D, in which the \R-maps are objectwise.
The generating \L-maps are those of the form $\D(d,\blank)\otimes i$, as $i$ ranges over the generating \L-maps in \M and $d$ over the objects of \D.
To obtain a projective \emph{model structure}, we need an additional hypothesis, such as that acyclic cofibrations in \M are preserved under tensoring by objects of \V.


Let \C be a bistratified \V-category.
Recall that we characterized these in \cref{thm:v-bistrat-char}, and that we write $\C_{=\de}$ for the $\de^{\mathrm{th}}$ stratum.
As in \cref{sec:auto}, we define generalized \textbf{matching} and \textbf{latching} functors for any degree $\de$ by
\begin{align}
  M_\de &: \M^\C \to \M^{\C_{=\de}}\\
  (M_\de A)_x &= \wlimc{\partial_\de\C(x,\blank)}{A}
\end{align}
and similarly
\[ (L_\de A)_x = \wcolimc{\partial_\de\C(\blank,x)}{A}. \]
Here $\partial_\de\C(\blank,\blank)$ is the enriched functor defined in \cref{defn:v-fundfact}.
We again may write $M_x A = (M_\de A)_x$ and so on, and $A_\de$ for the restriction of $A$ to $\C_{=\de}$.
We have an enriched version of \cref{thm:bistrat-wfs}, proven in the same way:

\begin{thm}\label{thm:v-bistrat-wfs}
  For any bistratified \V-category \C and any complete and cocomplete \V-category \M with a \V-pre-wfs, there is an induced \textbf{c-Reedy pre-wfs} on $\M^\C$ that is characterized by:
  \begin{itemize}
  \item $A\to B$ is an \R-map iff $A_x \to M_x A \times_{M_x B} B_x$ is an \R-map in \M for all $x\in\C$.
  \item $A\to B$ is an \L-map iff $L_\de B \sqcup_{L_\de A} A_\de \to B_\de$ is a projective \L-map in $\M^{\C_{=\de}}$ for all degrees $\de$.
  \end{itemize}
  If the projective pre-wfs on each $\M^{\C_{=\de}}$ is a wfs, then the c-Reedy pre-wfs on $\M^\C$ is also a wfs.\qed
\end{thm}

And an enriched version of \cref{thm:creedy-wfs}, also proven in the same way.

\begin{thm}\label{thm:v-creedy-wfs}
  Let \C be a bistratified \V-category, \M a complete and cocomplete \V-category with a \V-pre-wfs, and \D a small \V-category.
  \begin{itemize}
  \item Let $\M^\C$ have its c-Reedy pre-wfs and $\V$ have its (cofibration, acyclic fibration) wfs.
  \item Let $\V^\C$ and $\V^{\C\op} \cong ((\V\op)^\C)\op$ have their c-Reedy wfs.
  \item Let $\V^{\C\op\times \D} \cong (((\V\op)^\C)\op)^\D$ and $\V^{\C\times\D\op} \cong (\V^\C)^{\D\op}$ have their induced injective pre-wfs.
  \end{itemize}
  Then:
  \begin{enumerate}
  \item The weighted colimit $\V^{\C\op\times \D} \times \M^\C \to \M^\D$ is a left wfs-bimorphism when $\M^\D$ has its injective pre-wfs.\label{item:vcreedywfs2}
  \item The weighted limit $(\V^{\C\times\D\op})\op \times \M^\C \to \M^\D$ is a right wfs-bimorphism when $\M^\D$ has its projective pre-wfs.\qed\label{item:vcreedywfs1}
  \end{enumerate}
\end{thm}

\begin{defn}
  A bistratified \V-category is \textbf{almost c-Reedy} if for all objects $x$ of degree $\be$,
  \begin{itemize}
  \item The hom-functor $\C(\blank,x)$ is an \L-object in the c-Reedy wfs on $\V^{\C_\be\op} \cong ((\V\op)^{\C_\be})\op$, and
  \item The hom-functor $\C(x,\blank)$ is an \L-object in the c-Reedy wfs on $\V^{\C_\be}$.
  \end{itemize}
\end{defn}

\begin{thm}\label{thm:v-creedy-model}
  Let \C be an almost c-Reedy \V-category and \M a \V-model category such that each $\M^{\C_{=\de}}$ has a projective model structure.
  Then $\M^\C$ has a model structure defined by:
  \begin{itemize}
  \item The weak equivalences are objectwise,
  \item $A\to B$ is a fibration iff the induced map $A_x \to M_x A \times_{M_x B} B_x$ is a fibration in \M for all $x\in \C$.
  \item $A\to B$ is a cofibration iff the induced map $L_\de B \sqcup_{L_\de A} A_\de \to B_\de$ is a projective-cofibration in $\M^{\C_{=\de}}$ for all degrees $\de$.\qed
  \end{itemize}
\end{thm}

As in \cref{sec:auto}, to say that each $\C(\blank,x)$ is an \L-object in $\V^{\C_{\be}\op}$ is to say that each map
\begin{equation}
  \partial_\de\C(y,x) \to \C(y,x)\label{eq:Vccofmap}
\end{equation}
is a cofibration in \V whenever $\de=\deg(y)<\deg(x)$.
And to say that each $\C(x,\blank)$ is an \L-object in $\V^{\C_\be}$ means that for each $\de<\deg(x)$, the map
\begin{equation}
  \partial_\de\C(x,\blank) \to \C(x,\blank)\label{eq:Vcprcofmap}
\end{equation}
is a projective-cofibration in $\V^{\C_{=\de}}$.
As in \cref{sec:enriched}, we can obtain these conditions together with bistratification from a definition in terms of ``enriched factorizations''.

\begin{defn}\label{defn:v-creedy}
  A \textbf{c-Reedy \V-category} is a small \V-category \C with ordinal degrees assigned to its objects, along with:
  \begin{enumerate}
  \item A commutative square of \V-functors that are the identity on objects:\label{item:vcr1}
    \begingroup
    \def\C{\mathcal{C}}
    \begin{equation}
      \vcenter{\xymatrix@-.5pc{
          \Cbar\ar[r]\ar[d] &
          \rup\C\ar[d]\\
          \rdn\C\ar[r] &
          \C.
        }}
    \end{equation}
    \endgroup
  \item If $\deg(x)=\deg(y)$, then the maps $\Cbar(x,y) \to \rup\C(x,y)$ and $\Cbar(x,y) \to \rdn\C(x,y)$ are isomorphisms.\label{item:vcr2}
  \item If $\deg(x)\neq\deg(y)$, then $\Cbar(x,y)=\emptyset$.\label{item:vcr3}
  \item If $\deg(x)>\deg(y)$, then $\rup\C(x,y) = \emptyset$ and $\rdn\C(y,x) = \emptyset$.\label{item:vcr4}
  \item The hom-objects of $\rup\C$ (hence also those of $\Cbar$) are cofibrant.\label{item:vcr7}
  \item For any $x$ and $\de<\deg(x)$, the functor $\rdn\C(x,\blank):\Cbar_{=\de} \to \V$ is projectively cofibrant, where $\Cbar_{=\de}$ is the full subcategory of $\Cbar$ on objects of degree $\de$.\label{item:vcr6}
  \item For any $x,y$, the following diagram is a coequalizer:\label{item:vcr5}
    \begin{equation}
      \coprod_{z,w} \rup\C(w,y)\otimes \Cbar(z,w) \otimes \rdn\C(x,z) \toto
      \coprod_z\rup\C(z,y)\otimes \rdn\C(x,z) \too
      \C(x,y).\label{eq:v-creedy-coeq}
    \end{equation}
  \end{enumerate}
\end{defn}

Note that~\ref{item:vcr5} says in an enriched way that ``every morphism $f:x\to y$ factors as $f = \rup f \rdn f$ where $\rup f \in\rup\C$ and $\rdn f \in\rdn\C$, uniquely up to zigzags of such factorizations whose connecting maps are in $\Cbar$.''

Similarly to \cref{defn:vreedy}, the coproducts in~\ref{item:vcr5} may as well run only over $z,w$ with $\deg(z)=\deg(w)\leq \min(\deg(x),\deg(y))$.
Moreover, the two parallel maps preserve degrees of $z$ and $w$, so their coequalizer decomposes as a coproduct over $\de\leq\min(\deg(x),\deg(y))$ of the coequalizers of the parallel pairs
\begin{equation}
  \coprod_{\deg(z)=\deg(w)=\de} \rup\C(w,y)\otimes \Cbar(z,w) \otimes \rdn\C(x,z) \toto
  \coprod_{\deg(z)=\de}\rup\C(z,y)\otimes \rdn\C(x,z).\label{eq:v-creedy-coeqpart}
\end{equation}
For the summand when $\de=\min(\deg(x),\deg(y))$, we have three cases:
\begin{itemize}
\item If $\de=\deg(x)=\deg(y)$, then all hom-objects are equivalently those of $\Cbar$, so by the co-Yoneda lemma the coequalizer of~\eqref{eq:v-creedy-coeqpart} reduces to $\Cbar(x,y)$.
\item If $\de=\deg(x)<\deg(y)$, then all hom-objects of $\rdn\C$ are those of $\Cbar$, hence also those of $\rup\C$, so the coequalizer of~\eqref{eq:v-creedy-coeqpart} reduces to $\rup\C(x,y)$.
\item Similarly, if $\de=\deg(y)<\deg(x)$, the coequalizer reduces to $\rdn\C(x,y)$.
\end{itemize}
This unique summand of $\C(x,y)$ is the \emph{basic hom-object} from $x$ to $y$, and its complement is the \emph{non-basic hom-object}.
The latter is the coproduct of all the coequalizers of~\eqref{eq:v-creedy-coeqpart} for $\de<\min(\deg(x),\deg(y))$, which can be written as the coequalizer of
\begin{equation}
  \coprod_{z,w\in\C_\de} \rup\C(w,y)\otimes \Cbar(z,w) \otimes \rdn\C(x,z) \toto
  \coprod_{z\in\C_\de}\rup\C(z,y)\otimes \rdn\C(x,z).
\end{equation}
where as usual $\C_\de$ denotes the full subcategory on objects of degree $<\de$.
It will be convenient to regard this as a tensor product of functors
\begin{equation}
  \wcolim{\rup\C(\blank,y)}{\rdn\C(x,\blank)}{\Cbarde}
\end{equation}
where $\Cbarde$ denotes the full subcategory of $\Cbar$ on objects of degree $<\de$.

As before, these considerations imply that the inclusions $\rup\C(x,y) \to \C(x,y)$ and $\rdn\C(x,y)\to \C(x,y)$ are always coproduct coprojections, which for many \V{}s implies that they are monic.
If they \emph{are} monic, then we can conclude that the square in~\ref{item:vcr1} is actually a pullback, so we have an enriched version of the equality $\Cbar = \rup\C\cap\rdn\C$ from \cref{defn:creedy}.

\begin{thm}\label{thm:v-creedy}
  A c-Reedy \V-category is an almost c-Reedy \V-category.
\end{thm}
\begin{proof}
  Let $\de=\min(\deg(x),\deg(y))$.
  As in \cref{thm:v-reedy}, it will suffice to show that the map $\partial_\de \C(x,y) \to \C(x,y)$ is an isomorphism onto the non-basic hom-object.
  When $\deg(x)=\deg(y)$, this will show that \C is bistratified with strata $\Cbar_{=\de}$.
  When $\deg(x)<\deg(y)$, the assumed cofibrancy of the basic hom-object $\rup\C(x,y)$ will show that~\eqref{eq:Vccofmap} is a cofibration.
  And when $\deg(x)>\deg(y)$, the assumed projective-cofibrancy of the basic hom-functor $\rdn\C(x,\blank)$ will show that~\eqref{eq:Vcprcofmap} is a projective-cofibration.

  Thus, analogously to~\eqref{eq:vreedy-coeq}, what we want is a coequalizer diagram
  \begin{equation}
    \small \coprod_{z,w\in\C_\de} \C(w,y) \otimes \C(z,w) \otimes \C(x,z) \toto
    \coprod_{z\in \C_\de} \C(z,y) \otimes \C(x,z) \too
    \left(\wcolim{\rup\C(\blank,y)}{\rdn\C(x,\blank)}{\Cbarde}\right).
  \end{equation}
  Now note that because $\Cbarde$ is a subcategory of $\C_\de$, in computing the coequalizer of the two parallel maps above, it is would be equivalent to first tensor over $\Cbarde$.
  In other words, we may as well look for a coequalizer diagram
  \begin{equation}
    \small \left(\C(\blank,y) \otimes_{\Cbarde} \C(\blank,\blank) \otimes_{\Cbarde} \C(x,\blank)\right) \toto
    \left(\C(\blank,y) \otimes_{\Cbarde} \C(x,\blank)\right)  \too
    \left(\wcolim{\rup\C(\blank,y)}{\rdn\C(x,\blank)}{\Cbarde}\right).
  \end{equation}
  Now, however, we can essentially repeat the proof of \cref{thm:v-reedy}, simply changing the notational conventions so that $\C\C$, for instance, denotes $\C(\blank,y) \otimes_{\Cbarde} \C(x,\blank)$ rather than $\coprod_{z\in \C_\de} \C(z,y) \otimes \C(x,z)$.
\end{proof}

We end with a few of examples of c-Reedy \V-categories.
The first is obtained by modifying the main example of~\cite[\S6]{angeltveit:enr-reedy} to include automorphisms.

\begin{eg}\label{eg:prop}
  Let \O be an operad in \V, \emph{with} symmetric group action.
  We regard \O as a functor from the groupoid of finite sets and bijections to \V, with extra structure.
  Assume that $\O(0)$ and $\O(1)$ are cofibrant in \V, and that for $n\ge 1$, $\O(n)$ is a (projectively) cofibrant $\O(1)$-module (via the usual action).

  If $f:X\to Y$ is a map of finite sets, we write
  \[\O[f] = \bigotimes_{y\in Y} \O(f\inv(y)) \]
  Note that $\O[f]$ is canonically a module over $\O(1)^{\otimes|Y|}$; the assumption ensures that it is a cofibrant module whenever $f$ is surjective.

  Define a \V-category $\FO$ whose objects are (some canonical set of) finite sets and whose hom-objects are
  \begin{equation}
    \FO(X,Y) = \coprod_{f:X\to Y} \O[f].\label{eq:FOhom}
  \end{equation}
  The operad structure on \O makes this into a \V-category.
  (In fact, it is none other than the \textbf{PROP} associated to the operad \O.)
  We will show that it is a c-Reedy \V-category.

  We define the degree of a finite set to be its cardinality, and we define the hom-objects of $\rup\FO$, $\rdn\FO$, and $\FObar$ by restricting the coproducts in~\eqref{eq:FOhom} to run over injections, surjections, and bijections $X\to Y$ respectively.
  This yields \cref{defn:v-creedy}\ref{item:vcr1}.
  Condition~\ref{item:vcr2} follows since any injection or surjection between finite sets of the same cardinality is a bijection, while~\ref{item:vcr3} and~\ref{item:vcr4} are obvious.
  Condition~\ref{item:vcr7} holds because $\O(0)$ and $\O(1)$ are cofibrant, since $\rup\FO(X,Y)$ is a coproduct of copies of $\O(1)^{\otimes|X|} \otimes \O(0)^{\otimes |Y|-|X|}$.

  For condition~\ref{item:vcr6}, note that $\FObar(Y,Y) \cong \aut(Y) \cpw \O(1)^{\otimes |Y|}$.
  Thus, we are asking for projective-cofibrancy in $\V^{\aut(Y)\cpw\O(1)^{\otimes |Y|}}$ of $\coprod_{f:X\epi Y} \O[f]$, the coproduct being over surjections.
  But $\aut(Y)$ acts freely on the set of surjections $X\epi Y$, so this object is freely generated by a coproduct of $\O(1)^{\otimes |Y|}$-modules of the form $\O[f]$ for a surjection $f$.
  Since each of these is a cofibrant module, the resulting $\FObar(Y,Y)$-module is also cofibrant.

  Finally, the diagram in~\ref{item:vcr5} splits up as a coproduct over functions $f:X\to Y$:
  \begin{equation}
    \coprod_{f = h k g} \O[h] \otimes \O[k] \otimes \O[g] \toto
    \coprod_{f = h g}\O[h]\otimes \O[g] \too
    \O[f]
  \end{equation}
  where $g:X\to Z$ is surjective, $h:Z'\to Y$ is injective, and $k:Z\to Z'$ is bijective.
  We therefore have
  \begin{align}
    \O[g] &\cong \bigotimes_{y\in Y, f\inv(y)\neq\emptyset} \O(f\inv(y))\\
    \O[h] &\cong \O(1)^{\otimes|\im(f)|} \otimes \O(0)^{\otimes|Y\setminus\im(f)|}\\
    \O[k] &\cong \O(1)^{\otimes|\im(f)|}.
  \end{align}
  Thus, for given $h,k,g$ we have
  \begin{align}
    \O[h] \otimes_{\O[k]} \O[g] &\cong \O(0)^{\otimes|Y\setminus\im(f)|}\otimes \bigotimes_{y\in Y, f\inv(y)\neq\emptyset} \O(f\inv(y))\\
    &\cong \O[f].
  \end{align}
  The coproducts over $h,k,g$ disappear in the quotient because $f$ factors uniquely up to isomorphism as a surjection followed by an injection; thus we have a coequalizer as desired.
\end{eg}

\begin{eg}\label{eg:operators}
  In addition to the assumptions in \cref{eg:prop}, suppose that \V is \emph{cartesian} monoidal.
  (Technically, it suffices for it to be \emph{semicartesian}, i.e.\ for the unit of the monoidal structure to be the terminal object.)
  Then in the definition of \FO we can use \emph{partial} functions $X\rightharpoonup Y$ instead of total ones.
  In defining composition, we may have to ``project away'' objects such as $\O(n)$; this is where (semi)cartesianness comes in.
  Other than this, the rest of the argument in \cref{eg:prop} goes through without any trouble.
  The resulting \V-category is the \textbf{category of operators} associated to \O as in~\cite[Construction 4.1]{mt:uniq-inf-loop}, which is therefore also a c-Reedy \V-category.

  If we allow arbitrary \emph{relations} from $X$ to $Y$, then the resulting category is the (\V-enriched) \textbf{Lawvere theory} associated to \O.
  Unfortunately, in this case \cref{defn:v-creedy}\ref{item:vcr6} fails, so we no longer get a c-Reedy \V-category.
\end{eg}

\begin{eg}\label{eg:enriched-orbit}
  As in \cref{eg:orbit}, let $G$ be a compact Lie group and $\O_G$ its orbit category, consisting of the transitive $G$-spaces $G/H$, with $H$ a closed subgroup, and the continuous $G$-maps between them.
  Then $\O_G$ is a topologically enriched c-Reedy category, where as before we take $\omega\cdot \dim(H)+|\pi_0(H)|-1$ and $\rup{\O_G} = \O_G$, while $\OGbar = \rdn{\O_G}$ contains only the isomorphisms.

  Thus, if $\M$ is a topological model category admitting projective model structures, then $\M^{\O_G}$ has a c-Reedy model structure, and if $\M$ admits injective model structures, then $\M^{\O_G\op}$ has a c-Reedy model structure.
  However, because topological spaces are not a combinatorial model category, we do not generally know how to construct injective model structures on topological model categories, which is unfortunate since $\M^{\O_G\op}$ is more often the category of interest in equivariant homotopy theory.
  But if we pass to the homotopy-equivalent world of simplicial sets, making $\O_G$ into a simplicially enriched category by taking singular complexes of its hom-spaces, then it remains enriched c-Reedy, and injective model structures do exist.
  Thus, we have a Reedy-style model structure for equivariant homotopy theory over a compact Lie group, which as in \cref{eg:orbit} turns out to be an explicit version of the injective model structure.
\end{eg}

\begin{rmk}
  As mentioned in the introduction, we can also generalize to ``category theories'' other than ordinary and enriched ones.
  Other well-known category theories are internal categories and indexed categories.
  more generally, we could use the enriched indexed categories of~\cite{shulman:eicats}.
  And more generally still, we could use a framework for formal category theory such as the \emph{cosmoi} of~\cite{street:cauchy-enr} or the \emph{proarrow equipments} of~\cite{wood:proarrows-ii}.
  The fundamental idea of ``Reedy-like categories'' as iterated collages of abstract bigluing data makes sense quite generally.
  Other potential generalizations include allowing different model category structures at different stages of the bigluing, such as in~\cite{johnson:modreedy}, or bigluing along more general functors, such as in the Reedy structure on oplax limits from~\cite{shulman:invdia}.
\end{rmk}

\appendix

\section{Functorially-Reedy categories}
\label{sec:freedy}

In some places such as~\cite{dhks:holim,barwick:reedy}, one finds a slightly different definition of Reedy category in which the factorizations $f = \rup f \rdn f$ are not required to be unique, only \emph{functorial}.
We will call these fs-Reedy categories (``s'' for ``strict'' or ``set'', meaning that they are analogous to those of \cref{sec:reedy} rather than \cref{sec:auto}).

\begin{defn}
  An \textbf{fs-Reedy category} is a category \C equipped with a ordinal degree function on its objects, and subcategories $\rup\C$ and $\rdn \C$ containing all the objects, such that
  \begin{itemize}
  \item Every nonidentity morphism in $\rup\C$ strictly raises degree and every nonidentity morphism in $\rdn\C$ strictly lowers degree.
  \item Every morphism $f$ factors \emph{functorially} as $\rup f \rdn f$, where $\rup f \in \rup \C$ and $\rdn f \in\rdn \C$.
  \end{itemize}
\end{defn}

Perhaps surprisingly, this definition is \emph{not} equivalent to the more common one with unique factorizations.
There are roughly two ways in which it can fail.

\begin{eg}\label{eg:fsreedy1}
  The following example is due to Charles Rezk.
  Let \C be the poset $(1\le 0 \le 2)$, with degrees assigned as follows:
  
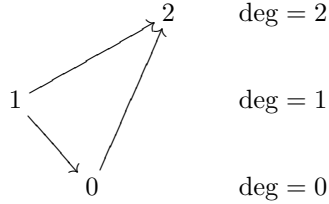
\[ \xymatrix@-.5pc{ && 2 & \deg = 2\\
    1 \ar[urr] \ar[dr] &&& \deg=1\\
    & 0 \ar[uur] && \deg=0 } \]
  This is almost-Reedy, with $1\to 0$ and $0\to 2$ being the only nonidentity morphisms in $\rdn \C$ and $\rup \C$ respectively; hence by \cref{thm:reedy} it is Reedy.
  
  However, if we enlarge $\rup\C$ by adding the morphism $1\to 2$ to it (which also strictly raises degree), then the previous unique factorizations are still \emph{functorial}, but no longer unique, since $1\to 2$ can now also be factored as $1 \to 1 \to 2$.
  Thus, with this larger $\rup\C$, \C is fs-Reedy but not Reedy.
\end{eg}

\begin{eg}\label{eg:fsreedy2}
  Let \C be the category with two objects $0$ and $1$, having degrees $0$ and $1$ respectively, and a single nonidentity isomorphism $0\cong 1$.
  Then \C is not even bistratified, and since it contains a nonidentity isomorphism it does not even admit some other degree function for which it is almost-Reedy.
  However, if we let $\rup\C$ and $\rdn\C$ be the subcategories containing only the nonidentity morphisms $0\to 1$ and $1\to 0$, respectively, then \C is fs-Reedy: any factorization at all is functorial, since \C is a contractible groupoid.
\end{eg}

Note that in \cref{eg:fsreedy1}, we can make \C s-Reedy by shrinking $\rup\C$ and $\rdn\C$; while in \cref{eg:fsreedy2}, \C is \emph{equivalent} to an s-Reedy category, though it is not s-Reedy itself.
These are the only two things that can happen.

\begin{thm}\label{thm:fsreedy}
  Any fs-Reedy category \C has a full subcategory \D such that
  \begin{enumerate}
  \item The inclusion functor $\D\into\C$ is an equivalence, and
  \item The induced degree function on \D makes it s-Reedy with $\rup\D\subseteq \rup\C$ and $\rdn\D\subseteq \rdn\C$.
  \end{enumerate}
\end{thm}
\begin{proof}
  Let \C be fs-Reedy.
  First note that if $f$ is basic and level with specified factorization $f = \rup f \rdn f$, then since $\rup f$ non-strictly increases degree and $\rdn f$ non-strictly decreases it, in fact they must both be level, and hence identities.
  Thus all basic level morphisms are identities --- but, as \cref{eg:fsreedy2} shows, not every identity may be basic.
  
  Let \D be the full subcategory of \C determined by the objects $x$ such that $\id_x$ \textit{is} basic.
  We will show that $\D\into\C$ is an equivalence, or equivalently that every object of \C is isomorphic to an object of \D.
  In fact, we will show that every object $x\in \C$ is isomorphic to an object $y\in \D$ with $\deg(y)\le \deg(x)$, and that moreover the inverse isomorphisms $x\to y$ and $y\to x$ lie in $\rdn\C$ and $\rup\C$ respectively.
  
  Suppose $\id_x$ is not basic; we may assume by induction that the claim holds for all objects of strictly lesser degree.
  Now the argument is much like that of \cref{thm:reedy-fact}: since $\id_x$ is not basic, it has some fundamental factorization $\id_x = h_0 g_0$.
  If $g_0 h_0 = \id_{y_0}$, then $x\cong y_0$, and $\deg(y_0)<\deg(x)$ so the inductive hypothesis kicks in.
  Otherwise, $g_0 h_0$ is a nonidentity level morphism, so that $g_0 h_0 = \rup{g_0 h_0} \rdn{g_0 h_0}$ is a fundamental factorization through some object $y_1$.

  Let $g_1 = \rdn{g_0 h_0} g_0$ and $h_1 = h_0 \rup{g_0 h_0}$; then $h_1 g_1 = h_0 g_0 h_0 g_0 = \id_x$.
  If $g_1 h_1 = \id_{y_1}$, then $x\cong y_1$ and the inductive hypothesis kicks in; otherwise we proceed defining $g_2 = \rdn{g_1 h_1} g_1$ and $h_2 = h_1 \rup{g_1 h_1}$ and so on.
  We have $\deg(y_0) > \deg(y_1) >\deg(y_2)>\cdots$, so since ordinals are well-founded, the process must terminate; hence we have $x\cong y_n$ and we are done.
  Note that since $\rup\C$ and $\rdn\C$ are closed under composition, the isomorphisms $g_n:x\to y_n$ and $h_n:y_n \to x$ lie in $\rdn\C$ and $\rup\C$ respectively.

  Now we claim that \D is also fs-Reedy with the induced degree function and $\rup\D = \rup\C \cap\D$ and $\rdn\D= \rdn\C\cap\D$.
  The specified functorial factorization of a morphism $f:x\to y$ in $\D$ might, in principle, go through an object $z\in \C\setminus\D$.
  However, we have $z\cong w$ with $w\in\D$, via isomorphisms $g:z\to w$ and $h:w\to z$ in $\rdn\C$ and $\rup\C$ respectively, so that $f = (\rup f h)(g\rdn f)$ is a factorization of $f$ lying entirely in \D.
  Since this factorization is isomorphic to the given one $(\rup f,\rdn f)$, if we splice it in instead the factorization operation remains functorial.
  We can do this for all morphisms in \D, so that \D admits the desired functorial factorization.

  It remains to show that \D is s-Reedy; we will do this with \cref{thm:reedychar} and \cref{thm:reedy}.
  Condition~\ref{item:reedychar1} of \cref{thm:reedychar} holds by the above observation that all basic level morphisms in an fs-Reedy category are identities, together with the construction of \D so that all identities are basic.
  For \cref{thm:reedychar}\ref{item:reedychar2}, let $(h,g)$ and $(h',g')$ be fundamental factorizations of some morphism $f$, and apply the functorial factorization to the square
  \begin{equation}
    \vcenter{\xymatrix{
        \ar[r]^g\ar[d]_{g'} &
        \ar[d]^{h}\\
        \ar[r]_{h'} &
      }}
    \qquad\text{to get}\qquad
    \vcenter{\xymatrix{
        \ar[r]^{\rdn g}\ar[d]_{g'} &
        \ar[r]^{\rup g} \ar[d]^k &
        \ar[d]^{h}\\
        \ar[r]_{\rdn {h'}} &
        \ar[r]_{\rup{h'}}&
      }}
  \end{equation}
  The new intermediate objects must also be of strictly lesser degree than the domain and codomain of $f$.
  Thus, we have a connecting zigzag of fundamental factorizations:
  \[ \xymatrix@R=3pc@C=1.5pc{
    &&& \ar[dlll]_{g} \ar[dl]|{\rdn g} \ar[dr]|{\rdn{h'}g'} \ar[drrr]^{g'}\\
    \ar[drrr]_{h} && \ar[dr]|{h\rup g}\ar[ll]_{\rup g} \ar[rr]^k && \ar[dl]|{\rup{h'}} && \ar[dlll]^{h'}\ar[ll]_{\rdn{h'}}  \\
    &&&
  } \]
  By \cref{thm:reedychar}, therefore, \D is almost s-Reedy.
  Importantly, however, the classes of morphisms ``$\rup\D$'' and ``$\rdn\D$'' defined in \cref{sec:reedy} (consisting of basic morphisms that non-strictly increase or decrease degree, respectively) may not be the same as the subcategories $\rup\D$ and $\rdn\D$ specified in the fs-Reedy structure of \D.
  We will continue to use the notations $\rup\D$ and $\rdn\D$ for the latter only.

  It \emph{is} true that every basic morphism lies in $\rup\D$ or $\rdn\D$, since the specified functorial factorization is fundamental unless one of its factors is an identity.
  However, the converse may not hold.
  Thus, what we have to prove to apply \cref{thm:reedy} is that the composite of two basic morphisms in $\rdn\D$ is again basic (the dual statement has a dual proof).
  
  Suppose, therefore, that $f:x\to y$ and $g:y\to z$ are basic and lie in $\rdn\D$.
  Applying the functorial factorization to the commutative square on the left, we obtain the rectangle on the right:
  \begin{equation}
    \vcenter{\xymatrix{
        \ar[r]^f\ar@{=}[d] &
        \ar[d]^g\\
        \ar[r]_{g f} &
      }}
    \qquad\leadsto\qquad
    \vcenter{\xymatrix{
        \ar[r]^f\ar@{=}[d] &
        \ar@{=}[r]\ar[d]^k &
        \ar[d]^g\\
        \ar[r]_{\rdn{gf}} &
        \ar[r]_{\rup{gf}} &
      }}
  \end{equation}
  Note that since $f$ is basic, $f = \rdn f$ and $\rup f = \id$, since otherwise $f = \rup f \rdn f$ would be a fundamental factorization.
  But now $g = \rup{gf} \circ k$ is a fundamental factorization, contradicting the fact that $g$ is basic, unless $\rup{gf}= \id$.
  Therefore, $gf = \rdn{gf}$.

  Now suppose $g f = h k$ were a fundamental factorization, and once again apply the functorial factorization:
  \begin{equation}
    \vcenter{\xymatrix{
        \ar[r]^{gf}\ar[d]_{k} &
        \ar@{=}[d]\\
        \ar[r]_{h} &
      }}
    \qquad\leadsto\qquad
    \vcenter{\xymatrix{
        \ar[r]^{g f}\ar[d]_k &
        \ar@{=}[r]\ar[d]^\ell &
        \ar@{=}[d]\\
        \ar[r]_{\rdn{h}} &
        \ar[r]_{\rup{h}} &
      }}
  \end{equation}
  Since the factorization $h k$ is fundamental, $h$ must strictly increase degree, whence $\rup h$ is not an identity.
  But then $\rup h \ell = \id$ is a fundamental factorization of an identity, contradicting the fact that all identities in \D are basic.
  Thus, $g f$ is basic, as desired.
\end{proof}

\begin{cor}
  If \C is fs-Reedy and \M is a model category, then $\M^\C$ admits a model structure defined as follows:
  \begin{itemize}
  \item The weak equivalences are objectwise.
  \item $A\to B$ is a fibration (resp.\ acyclic fibration) iff for any $x\in \C$ that is not isomorphic to any object of strictly lesser degree, the induced map $A_x \to M_x A \times_{M_x B} B_x$ is a fibration (resp.\ acyclic fibration) in \M.
  \item $A\to B$ is a cofibration (resp.\ acyclic cofibration) iff for any $x\in \C$ that is not isomorphic to any object of strictly lesser degree, the induced map $L_x B \sqcup_{L_x A} A_x \to B_x$ is a cofibration (resp.\ acyclic cofibration) in \M.
  \end{itemize}
  Here $M_x$ and $L_x$ are as defined for all bistratified categories in \cref{sec:reedy}, and cannot necessarily be reformulated purely in terms of the specified subcategories $\rup\C$ and $\rdn\C$ as in the discussion preceeding \cref{thm:reedy-initial}.
\end{cor}
\begin{proof}
  With $\D$ as in \cref{thm:fsreedy}, simply transfer the Reedy model structure on $\M^\D$ across the equivalence $\M^\D \simeq\M^\C$ induced by the equivalence $\D\simeq \C$.
\end{proof}

We leave it to the reader to consider fg-Reedy and fc-Reedy categories.

\section{Remarks on the formalization}
\label{sec:formalization}

As mentioned in the introduction, I have formalized the proof of \cref{thm:creedy} in the computer proof assistant Coq, assuming as given the elementary characterization of almost c-Reedy categories from \cref{thm:creedy-char}.
The code is available from the arXiv and on my web site; it requires nothing more than a standard install of Coq v8.4.
For the most part, the definitions and proofs in the formalization follow those in the paper fairly closely.
I have included quotations from the paper as comments in the code, to help the interested reader match them up.

There are a few minor ways in which the formalized proofs diverge slightly from those in the paper.
For instance, for expositional purposes the paper built up gradually through inverse categories, stratified categories, s-Reedy categories, and finally g- and c-Reedy categories; but for the formalization it was simpler to go directly to the most general c-Reedy case.
Another small change is that, to avoid the need for a category theory library, the (equivalent) conditions \cref{thm:creedy-char}\ref{item:creedychar4} and \cref{defn:creedy}\ref{item:crd5} are phrased in a more elementary way in terms of cones (which is actually how they are used in the proof of \cref{thm:creedy-cone}).
The organization of results also differs somewhat, e.g.\ due to the requirements of formalization, it was convenient to separate out more statements as intermediate lemmas.

Some slightly more substantial differences involve the nature of the mathematical framework of Coq.
For instance, Coq's logic is by default constructive, meaning that if we want the law of excluded middle ($P$ or not $P$) we need to assert it as an additional axiom.
We did use excluded middle throughout this paper (e.g.\ ``$f$ is either basic or not basic''), so to formalize the paper as written we need some assumptions of this sort.
However, it is almost as easy, more informative, and more in the constructive spirit to assert only those particular instances of excluded middle that are needed.
This turned out to be the following:
\begin{enumerate}
\item As part of the definition of almost c-Reedy category (for which, recall, the formalization uses \cref{thm:creedy-char}), we assert that basic-ness is decidable (i.e.\ every morphism is either basic or not basic).
  In fact, since ``$f$ is basic'' is already a negative statement, it is better to assert decidability of the positive one: every morphism either admits a fundamental factorization or does not.\label{item:lem1}
\item Similarly, in the definition of c-Reedy category (\cref{defn:creedy}) we assert that the subcategories $\rup\C$ and $\rdn\C$ are decidable: every morphism is either in $\rup\C$ or not, and likewise for $\rdn\C$.\label{item:lem2}
\item In all cases, degrees of objects are assumed to take values in some ordinal that satisfies the trichotomy principle (i.e.\ for degrees $\de$ and $\de'$ we have either $\de<\de'$ or $\de=\de'$ or $\de>\de'$).
  In fact, the degrees are just assumed to come with a well-founded relation (as defined in the Coq standard library) and to satisfy trichotomy.\label{item:lem3}
\end{enumerate}
I find~\ref{item:lem1} and~\ref{item:lem2} intriguingly suggestive of a possible constructive version of the results in this paper.
Recall that the inclusion of the non-basic subset of $\C(x,y)$ has two faces: when $\deg(x)=\deg(y)$ it is the injection of a coproduct, as required for bistratification, while when $\deg(x)\neq\deg(y)$ it is an \L-map for the universally enriching wfs on \nSet, as required for cofibrancy of the appropriate bigluing data.
In both cases it \emph{is} natural, constructively, to assume it to be a decidable subset: injections of coproducts are always decidable, whereas arguably the most natural constructive replacement for the (injections, surjections) wfs on \nSet (whose existence is equivalent to the axiom of choice) is (coproduct injections, split surjections).
However,~\ref{item:lem3} is somewhat disheartening in this regard, since constructively the trichotomous ordinals may be quite rare~\cite{rosolini:ist}.
Further investigation of this question is left to the interested reader.

Another group of differences arise from the \emph{intensional} nature of Coq's type theory.
From the recent perspective of homotopy type theory~\cite{hottbook}, this means that the basic objects of Coq, called ``types'', do not necessarily behave like sets, but rather like higher groupoids.
We can make them behave like sets by assuming an axiom; but as with excluded middle, it is almost as easy, more informative, and more in the spirit of homotopy type theory to assert this only about the particular types for which we need it.
This turns out to be the following:
\begin{enumerate}
\item As part of our definition of a category, we assert that the hom-types $\C(x,y)$ are sets.
  We do \emph{not} need to assert that the type of \emph{objects} is a set; thus our categories are the ``precategories'' of~\cite{aks:rezk} and~\cite[Ch.~9]{hottbook}.
  (These references use ``category'' to refer to precategories satisfying a ``saturation'' condition ensuring that the type of objects behaves like the maximal subgroupoid of the category, but we have no need for this either.)
\item We always assume that the type of degrees of objects (which, recall, is a trichotomous ordinal) is a set.
\item We assume that for degrees $\de$ and $\de'$, the type $(\de<\de')$ is a set with at most one element, which in homotopy type theory is often called an ``h-proposition''.
  (In type theory, propositions like $(\de<\de')$ are identified with particular types; the h-propositions are those that carry no more information than a truth value.)
\end{enumerate}
These are all natural assumptions from the perspective of homotopy type theory.

One final issue has to do with the \emph{divergence} of Coq from homotopy type theory.
Coq includes a sort \texttt{Prop} of ``propositions'', but not only are these not necessarily h-propositions (hence the need for the assumption above about $(\de<\de')$), but not every h-proposition necessarily lies in \texttt{Prop}.
In particular, this means that the function comprehension principle (if for every $a\in A$ there is a unique $b\in B$ such that $P(a,b)$, then there is a function $f:A\to B$ such that $P(a,f(a))$) fails unless asserted as an axiom.
The only place this became a problem was in defining functions by cases based on the trichotomy law, so I simply stated trichotomy as a disjoint union rather than a disjunction; this is unproblematic semantically because well-founded relations are automatically irreflexive.

Finally, the reader may wonder whether the process of formalization uncovered any errors in the original proofs.
The answer is no, not substantial ones.
The original statement of \cref{thm:creedy-cone} omitted the hypothesis that $f_1,f_2,f_3$ strictly decrease degree, but this is satisfied in all cases where it is used.
And the original proof of \cref{thm:creedy-fact-uniq} tried to prove first that all the connecting maps were basic level and then use a separate induction to show that all the factorizations were Reedy, when in fact this has to be done simultaneously.
There were, of course, many places where details I had left to the reader needed to be spelled out explicitly for the computer, but this is to be expected of any formalization and doesn't necessarily indicate any problem with the informal proofs.
There was only one place where the process of formalization made me decide that I had left too much to the reader: in the proof of \cref{thm:creedy}, the proof that the specified $\rup\C$ and $\rdn\C$ agree with their similarly-named defined versions requires a more substantial argument than in the s-Reedy case.

\bibliographystyle{alpha}
\bibliography{all}

\end{document}